\documentclass{article}
\usepackage[utf8]{inputenc}
\usepackage{amsmath,amssymb,amsthm}
\usepackage{graphicx}
\usepackage[all]{xy}
\usepackage{latexsym}
\usepackage{mathrsfs}
\usepackage{fancyhdr}
\usepackage{enumerate}
\usepackage[english]{babel}
\usepackage{lipsum}
\usepackage{etoolbox}
\usepackage{colonequals}
\usepackage{comment}

\usepackage{tikz}
\usetikzlibrary{cd}
\usepackage{caption}
\usepackage{mathtools}
\usepackage{color}
\usepackage{hyperref}
\hypersetup{colorlinks=true, linktoc=all, linkcolor=blue }
\usepackage{geometry}
\usepackage[title]{appendix}
\usepackage{tabularx}

\theoremstyle{plain}
\newtheorem{theorem}{Theorem}[section]
\newtheorem{prop}[theorem]{Proposition}
\newtheorem{lemma}[theorem]{Lemma}

\newtheorem{remark}[theorem]{Remark}
\newtheorem{cor}[theorem]{Corollary}
\newtheorem{definition}[theorem]{Definition}
\newtheorem{example}[theorem]{Example}

\newcommand{\bfs}{\mathbf{s}}
\newcommand{\At}{\mathcal{A}_{t}}
\newcommand{\Aprin}{\mathcal{A}_{\textrm{prin}}}
\newcommand{\Ascat}{\mathcal{A}_{\textrm{scat}}}
\newcommand{\lN}{{{}^\mathrm{L}N}}
\newcommand{\lM}{{{}^\mathrm{L}M}}
\newcommand{\ld}{{{}^\mathrm{L}d}}
\newcommand{\lle}{{{}^\mathrm{L}e}}
\newcommand{\lf}{{{}^\mathrm{L}f}}

\newcommand{\lX}{{{}^\mathrm{L}y}}
\newcommand{\lv}{{{}^\mathrm{L}v}}
\newcommand{\lL}{{}^\mathrm{L}}

\newcommand{\rN}{{{}^\mathrm{R}N}}
\newcommand{\rM}{{{}^\mathrm{R}M}}
\newcommand{\rd}{{{}^\mathrm{R}d}}
\newcommand{\re}{{{}^\mathrm{R}e}}
\newcommand{\rf}{{{}^\mathrm{R}f}}

\newcommand{\rX}{{{}^\mathrm{R}y}}
\newcommand{\rv}{{{}^\mathrm{R}v}}
\newcommand{\rR}{{}^\mathrm{R}}

\newcommand{\upC}{{{}^{\text{\tiny{C}}}}}

\newcommand{\seed}{\mathbf{s}}
\newcommand{\init}{\mathrm{in}}

\newcommand{\cA}{\mathcal{A}}
\newcommand{\cX}{\mathcal{X}}

\newcommand{\fD}{\mathfrak{D}}
\newcommand{\fd}{\mathfrak{d}}
\newcommand{\prin}{\mathrm{prin}}

\newcommand{\A}{\mathcal{A}}

\newgeometry{left=1in,right=1in,top=1in,bottom=1in}

\usepackage[backend=bibtex]{biblatex}
\bibliography{bibliography.bib}

\title{Cluster scattering diagrams and theta functions for reciprocal generalized cluster algebras}
\author{Man-Wai Cheung\thanks{Department of Mathematics, Kavli IPMU (WPI), UTIAS, The University of Tokyo, Kashiwa, Chiba 277-8583, Japan. manwai.cheung@ipmu.jp}, Elizabeth Kelley\thanks{Department of Mathematics, University of Illinois Urbana Champaign, Champaign, IL, USA. kelleye@illinois.edu}, Gregg Musiker\thanks{Department of Mathematics, University of Minnesota, Minneapolis, MN, USA. musiker@math.umn.edu}}
\date{\today}

\begin{document}

\maketitle

\begin{abstract}
We give a construction of generalized cluster varieties and generalized cluster scattering diagrams for reciprocal generalized cluster algebras, the latter of which were defined  by  Chekhov  and  Shapiro. These constructions are analogous to the structures given for ordinary cluster algebras in the work of Gross, Hacking, Keel, and Kontsevich.  As a consequence of these constructions, we are also able to construct theta functions for generalized cluster algebras, again in the reciprocal case, and demonstrate a number of their structural properties.
\end{abstract}

\vfill

\vspace{5mm}

\vspace{5mm}
\noindent \textbf{Data availability statement:} Data sharing is not applicable to this article as no datasets were generated or analysed during the current study.

\vspace{5mm}
\noindent \textbf{Corresponding Author:} Elizabeth Kelley

\pagebreak

\begingroup
\hypersetup{linkcolor=black}
\tableofcontents
\endgroup

\section{Introduction}
\label{sec:introduction}

The theory of cluster algebras was originally introduced by Fomin and Zelevinsky in 2000 as a tool for studying total positivity \cite{FZ-I}. 
Since their initial appearance, cluster algebra structures have also arisen in a diverse array of mathematical settings, including Poisson geometry \cite{poissonGeometry_book}, higher Teichm\"{u}ller theory \cite{Le, Wienhard}, category theory \cite{Keller-derived, Reiten}, discrete dynamical systems \cite{Kedem, NAK}, Donaldson-Thomas theory \cite{Kontsevich-Soibelman, Nagao}, representation theory of quivers and finite-dimensional algebras \cite{Keller}, enumerative properties of associahedra \cite{chapoton-associahedra,chapoton-fomin-zelevinsky}, and various areas of mathematical physics. 
Two particularly celebrated structural features of cluster algebras are the \emph{Laurent phenomenon} and \emph{positivity}. Although the Laurent phenomenon was proved in the original work of Fomin and Zelevinsky, a proof of positivity for skew-symmetric (resp. arbitrary) cluster algebras did not appear until the work of Lee and Schiffler in 2015 \cite{Lee-Schiffler} (resp. Gross, Hacking, Keel, and Kontsevich in 2018 \cite{GHKK}).

\vspace{1em}

Generalized cluster algebras were introduced by Chekhov and Shapiro in order to study the Teichm\"{u}ller spaces of Riemann surfaces with holes and orbifold points of arbitrary order \cite{CS, Chekhov-Mazzocco-Yangians}.
In this generalization, the hallmark binomial exchange relations of ordinary cluster algebras are replaced with polynomials with arbitrarily many terms. Consequently, to define the exchange polynomials, we need an additional finite set of coefficients, which we denote as $\{a_{i,j}\}$ and typically treat as formal variables.  Generalized cluster algebra structures have since been discovered in the representation theory of quantum affine algebras \cite{Gleitz-II, Gleitz}, the representation theory of quantum loop algebras \cite{Gleitz-III}, the study of exact WKB analysis \cite{Iwaki-Nakanishi}, the cyclic symmetry of Grassmannians \cite{Fraser}, the study of the Drinfeld double of $GL_n$ \cite{GSV-I,GSV-IV, GSV-II, GSV-III}, and in certain Caldero-Chapoton algebras of quivers with relations \cite{L-F-Velasco}.

Generalized cluster algebras exhibit many of the same structural properties as ordinary cluster algebras.  For instance, in \cite{CS}, Chekhov and Shapiro show that generalized cluster algebras exhibit the celebrated Laurent phenomenon and admit the same finite-type classification as ordinary cluster algebras.  Further, Chekhov and Shapiro prove that generalized cluster algebras exhibit positivity when the initial cluster has size two and conjecture that positivity holds in general. This was followed by work of Nakanishi which gives structural results for the subclass of reciprocal generalized cluster algebras where the exchange polynomials are required to be palindromic and monic \cite{Nakanishi}.  Follow-up work by Nakanishi and Rupel \cite{Nakanishi-Rupel} further related reciprocal generalized cluster algebras to pairs of skew-symmetrizable cluster algebras which they refer to as \emph{companion algebras}. 

\vspace{1em}

Inspired by their proof of positivity for ordinary cluster algebras, in this paper, we describe how the methods of Gross, Hacking, Keel, and Kontsevich can be extended to the case of reciprocal generalized cluster algebras.  To this end, we show that the definitions of \emph{cluster varieties} and \emph{cluster ensembles}, as defined by Fock and Goncharov \cite{FG}, may be generalized to obtain \emph{generalized cluster varieties} associated to reciprocal generalized cluster algebras.

Because generalized cluster algebras exhibit many of the same structural properties as ordinary cluster algebras, another natural question is whether bases defined for ordinary cluster algebras have natural extensions in this generalized setting.  Many subclasses of cluster algebras have known bases, including the cluster monomial basis for finite type, the generic basis for affine type \cite{caldero-keller}, the generic basis for acyclic type \cite{GLS2, GLS1}, and the bangle and band bases for cluster algebras of surface type \cite{MSW-2}.
In addition to their proof of positivity for cluster variables, Gross, Hacking, Keel, and Kontsevich also proved the existence of the \emph{theta basis} for arbitrary cluster algebras \cite{GHKK}.  

Their proofs, both of positivity and the existence of the theta basis, used \emph{scattering diagrams}, a tool from algebraic geometry. Scattering diagrams were first introduced in two dimensions by Kontsevich and Soibelman in \cite{KS-chapter} and then in arbitrary dimension by Gross and Siebert in \cite{GS} as a tool for constructing mirror spaces in mirror symmetry.

The main results of this paper parallel the work of Gross, Hacking, Keel, and Kontsevich and construct \emph{generalized cluster scattering diagrams} and their theta functions in the context of reciprocal generalized cluster algebras.  More precisely, we construct generalized cluster varieties whose rings of regular functions are generalized cluster algebras.  Simultaneously, we develop scattering diagrams which lie in the tropicalization of the Fock-Goncharov dual of said varieties.  This allows us to establish theta functions for such algebras. 

\subsubsection*{Structure of the article}
We begin, in Section \ref{sec:background}, by providing background that compares and contrasts the cases of ordinary and generalized cluster algebras.  This section concludes with an in-depth description of cluster varieties, cluster scattering diagrams, and theta functions for the case of ordinary cluster algebras. 

In Section \ref{sec:genScatDiag}, the definitions of generalized cluster varieties and generalized cluster scattering diagrams are given, with an emphasis on the changes that need to be made to extend to 
the setting of generalized cluster algebras. 
These definitions are given in full detail so that the reader already familiar with the ordinary case can skip Section \ref{sec:background} and begin here instead.
Many of our results are given for the subclass of \emph{reciprocal generalized cluster algebras}. This restriction is necessary in the proof of Theorem \ref{theorem:mutDiagConsistent}, which establishes the mutation invariance of generalized cluster scattering diagrams. We explain in Remarks \ref{rk:toclarify} and \ref{rk:however} the motivation behind some of the definitions chosen in this paper.
A generalized cluster scattering diagram appears in Example \ref{ex:gen_T_k}, and one in the presence of principal coefficients appears in Example \ref{ex:genG2_principalCoeff}.

After constructing generalized cluster scattering diagrams, we give a construction of the scheme $\Ascat$ from these diagrams and show in Theorem \ref{theorem:AscatCommuteMutation} that $\Ascat$ is isomorphic to the generalized $\cA$ cluster varieties in the case with principal coefficients.
Note that while one may follow \cite{GHK_log} to construct schemes from generalized cluster scattering diagrams, it does not directly follow that the algebra of the scheme is the associated generalized cluster algebra until one shows that the spaces $\Ascat$ and $\cA$ are isomorphic. 

Section \ref{subsec:genBrokenLines} introduces \emph{broken lines} and \emph{theta functions} in the context of generalized cluster scattering diagrams. One major result of this section is:
\begin{theorem}[Theorem \ref{theorem:almost_positive}]
\label{thm:main4}
The generalized cluster monomials can be expressed in terms of theta functions.
\end{theorem}
Another major result is the sign-coherence of $g$-vectors:
\begin{theorem}[Theorem \ref{thm:g-coherence}]
Given fixed data as defined in Definition \ref{def:gen fixed}, 
consider an initial generalized torus seed $\bfs = \{ (e_i,(a_{i,j})) \}$ as defined in Definition \ref{def:genTorusSeed}. The torus seed $\bfs$  defines the usual set of dual vectors $\{ f_i = d_i^{-1}e_i^* \}$. If $\bfs'$ is a mutation equivalent generalized torus seed, then the $i$-th coordinates of the $g$-vectors for the cluster variables in $\bfs'$ are either all non-negative or all non-positive when expressed in the basis $\{ f_1, \dots, f_n \}$.
\end{theorem} 

In Section \ref{sec:genThetaBasis}, we define theta functions on $\mathcal{A}$ and $\mathcal{X}$ and describe the product structures of particular collections of theta functions on $\Aprin$, $\cA$, and $\cX$. As a consequence, we show that theta functions on $\Aprin$ form a (topological) basis for a topological $R$-algebra completion of of $\textrm{up}(\Aprin)$ in Corollary \ref{cor:6-11}.

Finally, we conclude with Section \ref{sec:companionDiagrams} where we illustrate the relationship between generalized cluster algebras, their associated generalized cluster scattering diagrams, and the aforementioned companion algebras.  These companion algebras were introduced by Nakanishi and Rupel \cite{Nakanishi-Rupel} and encode much of the same structural data as the associated generalized cluster algebra.

\begin{remark}
As we were completing our paper, the preprint \cite{langmou} by Mou was posted.  Our work in this paper is independent of and contemporaneous to his work, but comparing our distinct approaches may be useful in future work. In particular, we note that we work with a different subclass of generalized cluster algebras. Our work concerns generalized cluster algebras in the sense of \cite{CS} that satisfy the reciprocity condition used by \cite{Nakanishi, Nakanishi-Rupel}. In the subclass that we consider, the exchange polynomials and wall-crossing automorphisms cannot necessarily be written as products of binomials with positive coefficients.
We explain the relevance of this hypothesis on the format of wall-crossing automorphisms in Remark \ref{rk:wecaretoomuch}.
\end{remark}

\begin{remark}
Preliminary versions of the present paper appeared in \cite{FPSAC} and \cite{Kelley}.
\end{remark}

\section{Background}
\label{sec:background}

Ordinary cluster algebras were introduced by Fomin and Zelevinsky in 2002 in order to study total positivity and dual canonical bases in semisimple groups \cite{FZ-I}. An \emph{ordinary cluster algebra $\mathcal{A}$} with clusters of size $n$ is a commutative subring of an ambient field $\mathcal{F}$ of rational functions in $n$ variables. One of the hallmark structural properties of an ordinary cluster algebra is that it can be presented without enumerating its entire set of generators and relations. Instead, an ordinary cluster algebra can be presented by specifying a set of \emph{cluster seed} data: a collection of $n$ distinguished generators $\{x_1, \dots, x_n \}$, where the $x_i$ are known as \emph{cluster variables} and the entire subset as a \emph{cluster}; a collection of \emph{coefficients} $\{ y_1, \dots, y_n \}$;  and an \emph{exchange matrix} $B$ which encodes the \emph{exchange relations} between cluster variables. 
From the seed data, one can generate the remainder of the cluster variables and coefficients via an involutive process called \emph{mutation}. The full set of cluster variables generates $\mathcal{A}$ as a subring of $\mathcal{F}$.

Another hallmark property of ordinary cluster algebras is that the exchange relations encoded by $B$ are binomial - i.e., they have the form
\[ x_kx_k' = \textrm{monomial} ~+~ \textrm{monomial}. \]
One natural question is to ask what happens when these binomial relations are replaced by other types of polynomials.
In this paper, we consider a generalization of this type, introduced by Chekhov and Shapiro, which is described in more detail in the following subsection.

Computing the full set of cluster variables requires multiple iterations of mutation. For any choice of initial cluster, a sequence of mutations can be used to express every cluster variable in terms of that initial cluster. Remarkably, in a phenomenon known as the \emph{Laurent phenomenon}, these expressions turn out to be Laurent polynomials. Perhaps even more remarkably, these Laurent polynomials have strictly non-negative coefficients. This property is referred to as \emph{positivity}. Although the Laurent phenomenon was proved in Fomin and Zelevinsky's original paper, positivity for arbitrary ordinary cluster algebras remained conjectural until the work of Gross, Hacking, Keel, and Kontsevich in 2018 \cite{GHKK}. 

\subsection{Generalized cluster algebras}
\label{subsec:genCA}

One natural generalization of a cluster algebra, introduced by Chekhov and Shapiro \cite{CS}, is to allow the characteristic binomial exchange relations to instead contain arbitrarily many terms. The resulting algebras, referred to as \emph{generalized cluster algebras}, have already been the object of significant study.

The introduction of generalized cluster algebras was originally  motivated by the study of Teichm\"{u}ller spaces of Riemann surfaces with holes and orbifold points of arbitrary order \cite{Chekhov, Chekhov-Mazzocco-Yangians}.  In particular, drawing on the work of Felikson, Shapiro, and Tumarkin \cite{FST-orbTriang, Felikson-Shapiro-Tumarkin} which defines ordinary cluster algebras from orbifolds, in \cite{CS}, Chekhov and Shapiro show that triangulations of orbifolds provide a geometric model for a certain subclass of generalized cluster algebras and demonstrate positivity for such cases.   

For the somewhat broader subclass of \emph{reciprocal generalized cluster algebras}, which we will define later in this section, Nakanishi shows in \cite{Nakanishi} that much of the structural theory of cluster seeds for ordinary cluster algebras still holds. In particular, Nakanishi extends the notions of $c$-vectors, $g$-vectors, and $F$-polynomials to the generalized setting and then used these notions to write formulas for the generalized cluster variables and coefficients. In \cite{Nakanishi-Rupel}, Nakanishi and Rupel subsequently define \emph{companion algebras} of reciprocal generalized cluster algebras and showed that these companion algebras, which are themselves ordinary cluster algebras, encode much of the structural data of the original generalized cluster algebra. We discuss companion algebras in much greater detail in Section~\ref{sec:companionDiagrams}.

In this section, we review some of the basic definitions and properties of generalized cluster algebras, largely following the structure of definitions in \cite{Nakanishi}. Two important differences in our presentation, however, are that we use the wide convention rather than the tall convention (i.e., our $B$ matrix would be the matrix $B^T$ in \cite{Nakanishi}) and that we use $r_i$ in the place of $d_i$. We make this latter change in notation to avoid a conflict with the usage of $d_i$ that arises in our discussion of `fixed data' in the context of scattering diagrams.

 Let $(\mathbb{P}, \oplus)$ be a semifield and $\mathcal{F}$ be isomorphic to the field of rational functions in $n$ independent variables with coefficients in $\mathbb{P}$.

\begin{definition}
\label{def:gen_clusterSeed}
A \emph{labeled generalized cluster seed} is a quintuple ${\Sigma = (\mathbf{x},
\mathbf{y},B,[r_{ij}],\mathbf{a})}$ such that
\begin{itemize}
    \item $\mathbf{x} = (x_1,\dots,x_n)$ is a free generating set for $\mathcal{F}$,
    \item $\mathbf{y}$ is an $n$-tuple with elements in $\mathbb{P}$,
    \item $B = [b_{ij}]$ is an $n \times n$ skew-symmetrizable matrix with entries in $\mathbb{Z}$,
    \item $[r_{ij}]$ is an $n \times n$ diagonal matrix with positive integer entries whose $i$-th diagonal entry is denoted $r_i$,
    \item and $\mathbf{a} = (a_{i,j})_{i \in [n], s \in [r_i-1]}$ is a collection of elements in $\mathbb{P}$. We later will consider these $a_{i,j}$'s to be formal variables in a ground ring $R = \Bbbk [a_{i,j}]$.
\end{itemize}
We refer to $\mathbf{x} = (x_1, \dots, x_n)$ as the \emph{cluster} of $\Sigma$, $\mathbf{y} = (y_1, \dots, y_n)$ as the \emph{coefficient tuple}, $B$ as the \emph{generalized exchange matrix}, $[r_{ij}]$ as the exchange degree matrix, and $\mathbf{a}$ as the exchange coefficient collection. The elements $x_1, \dots, x_n$ are the \emph{cluster variables} of $\Sigma$ and $y_1, \dots, y_n$ are its \emph{coefficient variables}.
\end{definition}

Together, the exchange degrees $\{ r_i \}$ and the exchange coefficient collection $\mathbf{a}$ determine a set of \emph{exchange polynomials} $1 + a_{i,1}u + \cdots + a_{i,r_i-1}u^{r_i-1} + u^{r_i} \in \mathbb{ZP}[u].$ The structure of the exchange relations for mutation in direction $k$ are determined by the $k$-th exchange polynomial. We will work in the same specialized setting as \cite{Nakanishi} and impose the additional requirement that \[{a_{i,j} = a_{i,r_i - s}}, \] i.e., all exchange polynomials are \emph{reciprocal polynomials}. Although this condition does not hold for all generalized cluster algebras, restricting our attention to this subset allows us to focus on a more tractable subclass of generalized cluster algebras.

\begin{definition}
\label{def:genMutation}
For a generalized cluster seed $\Sigma = (\mathbf{x},\mathbf{y},B,[r_{ij}],\mathbf{a})$, generalized mutation $\mu_k^{(r)}$ in direction $k$ is defined by the following exchange relations:
\begin{align*}
    b_{ij}' &= \begin{cases}
                    -b_{ij} & i = k \textrm{ or } j = k \\
                    b_{ij} + r_k\left( [-b_{ik}]_+b_{kj} + b_{ik}[b_{kj}]_+ \right) & i,j \neq k
                \end{cases} \\
    y_i' &= \begin{cases}
                y_k^{-1} & i = k \\
                y_i \left( y_k^{[ b_{ik}]_+} \right)^{r_k}\left( \bigoplus_{s=0}^{r_k} a_{k,s}y_k^{ s} \right)^{-b_{ik}} & i \neq k
            \end{cases} \\
    x_i' &= \begin{cases}
                x_k^{-1}\left(\prod_{j=1}^n x_j^{[- b_{kj}]_+} \right)^{r_k} \frac{ \sum_{s=0}^{r_k} a_{k,s} \widehat{y}_k^{ s}}{\oplus_{s=0}^{r_k} a_{k,s}y_k^{ s}} & i = k \\
                x_i & i \neq k
            \end{cases} \\
    a_{k,s}' &= a_{k,r_k - s}
\end{align*}
where $[\ \boldsymbol{\cdot}\ ]_+ = \textrm{max}(\ \boldsymbol{\cdot}\ ,0)$ and
\begin{align} \label{eq:changeofvar}
    \widehat{y}_i := y_i \prod_{j=1}^n x_j^{b_{ij}} 
\end{align}
\end{definition}

\begin{remark}
The mutation relation for $b_{ij}'$ in Definition~\ref{def:genMutation} is for the generalized exchange matrix, $B$. This is equivalent to writing that the matrix formed by the product $B[r_{ij}]$, whose entries we abbreviate as $(br)_{ij}$, mutates according to the relation
\[ (br)_{ij}' = (br)_{ij} + ([-(br)_{ik}]_+(br)_{kj} + (br)_{ik}[(br)_{kj}]_+). \]
On the matrix level, this reflects the fact that mutation commutes with right multiplication by $[r_{ij}]$: that is, $\mu_k(B[r_{ij}]) = \mu^{(r)}_k(B)[r_{ij}]$, where $\mu_k$ denotes ordinary matrix mutation and $\mu_k^{(r)}$ denotes the generalized matrix mutation given in Definition~\ref{def:genMutation}.
\end{remark}

\begin{remark}
In their original paper, Chekhov and Shapiro use matrices $B$ and $\beta$ in their exchange relations \cite{CS}. Their $B$ matrix is our $B[r_{ij}]$ matrix and their $\beta$ matrix is our $B$ matrix. Note that if the matrix $B$ is skew-symmetrizable, then the matrix $B[r_{ij}]$ is also skew-symmetrizable.
\end{remark}

The generalized cluster algebra associated to a particular generalized cluster seed can then be defined as:

\begin{definition}
The \emph{generalized cluster algebra} $\mathcal{A} = \mathcal{A}(\mathbf{x},\mathbf{y},B,[r_{ij}],\mathbf{a})$ associated to a generalized cluster seed $\Sigma$ is the $\mathbb{ZP}$-subalgebra of $\mathcal{F}$ generated by the cluster variables $\mathbf{x} = \{ x_i \}_{i \in [n]}$ of $\Sigma$.
\end{definition}

It is important to note that there is some potential for notational confusion because it is also common to simply write $\mathcal{A}$ when referring to the $\mathcal{A}$-variety. Typically, it is clear from context whether $\mathcal{A}$ refers to a (generalized) cluster algebra or (generalized) cluster variety and therefore we use this notation in order to be consistent with the literature.

Finally, we can give a statement of the Laurent phenomenon for generalized cluster algebras.

\begin{theorem}[Theorem 2.5 of \cite{CS}]
\label{theorem:genLaurentPhenom}
Let $\mathcal{A} = \mathcal{A}(\mathbf{x},\mathbf{y},B,[r_{ij}],\mathbf{a})$ be an arbitrary generalized cluster algebra. Its cluster variables can be expressed in terms of any cluster of $\mathcal{A}$ as Laurent polynomials with coefficients in $\mathbb{ZP}.$
\end{theorem}

Note that in the above theorem, the Laurent polynomial coefficients are in $\mathbb{ZP}$. Although these coefficients are in fact strictly non-negative for certain subclasses of generalized cluster algebras \cite{Banaian-Kelley, CS}, there is currently no proof of positivity for arbitrary generalized cluster algebras. Because ordinary cluster scattering diagrams were used to prove positivity for arbitrary ordinary cluster algebras, the conjectural positivity of arbitrary generalized cluster algebras is a powerful motivator for defining generalized cluster scattering diagrams.  

\begin{remark} 
For the remainder of the paper, we use the geometric language of (generalized) cluster varieties rather than the combinatorial language of (generalized) cluster algebras.  These two viewpoints are two sides of the same coin but the geometric language will be more conducive for defining (generalized) cluster scattering diagrams.
\end{remark}

\subsection{Ordinary cluster varieties}
\label{subsec:ordVarieties}

Before we can embark on any discussion of scattering diagrams, we must first establish some basic definitions. In this section, we will largely follow the exposition in \cite{FG} and \cite{GHK} and will develop these ideas in the context of ordinary cluster algebras. In Section \ref{sec:genScatDiag}, we modify these definitions for the new context of generalized cluster algebras and explain how the ordinary definitions in this section can be recovered as specializations.

We begin with definitions of \emph{fixed data} and \emph{torus seed data}, from \cite{GHK}, which together encode the information of a cluster seed.
We will work over characteristic zero field $\Bbbk$.

\begin{definition}
    The following data is referred to as \emph{fixed data}, denoted by $\Gamma$:
    \begin{itemize}
        \item The \emph{cocharacter lattice} $N$ with skew-symmetric bilinear form ${\{ \cdot, \cdot \}: N \times N \rightarrow \mathbb{Q}}$.
        \item A saturated sublattice $N_{\textrm{uf}} \subseteq N$ called the \emph{unfrozen sublattice}.
        \item An index set $I$ with $| I | = \textrm{rank}(N)$ and subset $I_{\textrm{uf}} \subseteq I$ such that ${| I_{\textrm{uf}}| = \textrm{rank}(N_{\textrm{uf}})}$
        \item A set of positive integers $\{ d_i \}_{i \in I}$ such that $\mathrm{gcd}(d_i) = 1$
        \item A sublattice $N^{\circ} \subseteq N$ of finite index such that ${\{ N_{\textrm{uf}}, N^{\circ} \} \subseteq \mathbb{Z}}$ and ${\{ N, N_{\textrm{uf}} \cap N^{\circ} \} \subseteq \mathbb{Z}}$
        \item A lattice $M = \textrm{Hom}(N,\mathbb{Z})$ called the \emph{character lattice} and sublattice ${M^{\circ} = \textrm{Hom}(N^{\circ},\mathbb{Z})}$.
    \end{itemize}
    The name `fixed data' refers to the fact that this data is fixed under mutation.
\end{definition}

\begin{definition} \label{def:ordseed}
    Given a set of fixed data, the associated \emph{torus seed data} is a collection $\bfs = \{e_i \}_{i \in I}$ such that $\{ e_i \}_{i \in I}$ is a basis for $N$, $\{e_i \}_i \in I_{\textrm{uf}}$ is a basis for $N_{\textrm{uf}}$, and $\{d_ie_i \}_{i \in I}$ is a basis for $N^{\circ}$. The torus seed data defines a new bilinear form
    \begin{align*}
        &[\cdot,\cdot]_{\bfs}: N \times N \rightarrow \mathbb{Q} \\
        &[e_i,e_j]_{\bfs} = \epsilon_{ij} = \{e_i,e_j \}d_j
    \end{align*}
    which gives the skew-symmetrizable matrix as in the cluster literature.
\end{definition}

\begin{remark}
Because we use the ``wide" convention for the exchange matrices of cluster algebras, our $B$ and $\epsilon$ matrices coincide. In the ``tall" convention, used by Fomin and Zelevinsky \cite{FZ-I}, the matrices are instead related by a transpose, i.e. then $\epsilon = B^T$.
\end{remark}

A choice of torus seed $\bfs = \{e_i \}_{i \in I}$ defines a dual basis $\{e_i^* \}_{i \in I}$ for $M$ and a basis $\{f_i = d_i^{-1}e_i^* \}_{i \in I}$ for $M^{\circ}$. It also defines two associated algebraic tori:
    \begin{align*}
        \mathbf{\mathcal{X}_s} &= T_M = \textrm{Spec}~\Bbbk[N], \\
        \mathbf{\mathcal{A}_s} &= T_{N^\circ} = \textrm{Spec}~\Bbbk[M^\circ].
    \end{align*}
The torus $\mathcal{X}_{\bfs}$ has coordinates $y_1, \dots, y_n$, where $y_i= z^{e_i}$ and the torus $\mathcal{A}_{\bfs}$ has coordinates $x_1, \dots, x_n$, where $x_i= z^{f_i}$. Although it is common in some portions of the literature to use the notation $A_1,\dots,A_n$ for the coordinates of $\mathcal{A}_{\bfs}$ and $X_1, \dots, X_n$ for the coordinates of $\mathcal{X}_{\bfs}$, our choice of notational convention is consistent with a large portion of the cluster algebra literature (c.f.: \cite{FZ-I, Nakanishi, Nakanishi-Rupel}) and, in particular, is consistent with the original definition of generalized cluster algebras \cite{CS}.

The bilinear form $\{\cdot, \cdot \}: N \times N \rightarrow \mathbb{Q}$ naturally defines maps ${p_1^*: N_{\textrm{uf}} \rightarrow M^{\circ}}$ and ${p_2^*: N \rightarrow M^{\circ} / N_{\textrm{uf}}^{\perp}}$ as
    \begin{align*}
        p_1^*\left(n \in N_{\textrm{uf}}\right) &= \left(n' \in N^{\circ} \mapsto \{n,n' \} \right), \\
        p_2^*\left(n \in N \right) &= \left( n' \in N_{\textrm{uf}} \cap N^{\circ} \mapsto \{n,n' \} \right).
    \end{align*}
Based on these maps, we can then choose a map $p^*: N \rightarrow M^{\circ}$ such that we have the restriction $\left.p^*\right|_{N_{\textrm{uf}}} = p_1^*$ and the composition of $p^*$ with the quotient map $M^{\circ} \rightarrow M^{\circ} / N_{\textrm{uf}}^{\perp}$ agrees with $p_2^*$. It is important to note that the choice of $p^*$ is not  unique because there is more than one possible choice of map $N / N_{\textrm{uf}} \rightarrow N_{\textrm{uf}}^{\perp}$. It is also important to note that for an arbitrary choice of fixed data, the map $p_1^*: N_{\textrm{uf}} \rightarrow M^{\circ}$ is not necessarily injective. It is, however, always injective for the principal coefficient case discussed in Section~\ref{subsec:ordVarPrincipal}. The assumption that $p_1^*$ is injective is sometimes referred to as the \emph{injectivity assumption}.

The injectivity assumption is, in fact, a  crucial ingredient in many of the arguments given by Gross, Hacking, Keel, and Kontsevich for arbitrary ordinary cluster algebras and therefore many of their results are proved via the principal coefficient case. For the same reason, we will also work via the principal coefficient case in our generalized setting.

Because the fixed data and torus seed data encode information from a cluster seed, there should also be a notion of mutation.

\begin{definition}
\label{def:basisMut}
    Given torus seed data $\bfs$ and some $k \in I_{\textrm{uf}}$, a \emph{mutation in direction $k$} of the torus seed data is defined by the following transformations of basis vectors:
    \begin{align*}
        e_i' &:= \begin{cases}
                    e_i + [\epsilon_{ik}]_+e_k & i \neq k \\
                    -e_k & i = k
                \end{cases} \\
        f_i' &:= \begin{cases}
                    -f_k + \sum_{j \in I_{\textrm{uf}}} [-\epsilon_{kj}]_{+}f_j & i = k \\
                    f_i & i \neq k
                \end{cases}
    \end{align*}
    The basis mutation induces the following mutation of the matrix $[\epsilon_{ij}]$:
    \begin{align*}
        \epsilon_{ij}' &:= \{e_i',e_j' \}d_j
                        = \begin{cases}
                            -\epsilon_{ij} & k = i \textrm{ or } k = j \\
                            \epsilon_{ij} & k \neq i,j \textrm{ and } \epsilon_{ik}\epsilon_{kj} \leq 0 \\
                            \epsilon_{ij} + |\epsilon_{ik}|\epsilon_{kj} & k \neq i,j \textrm{ and } \epsilon_{ik}\epsilon_{kj} \geq 0
                        \end{cases} 
    \end{align*}
\end{definition}
Note that mutation of torus seeds is not an involution, so $\mu_k(\mu_k(\bfs)) \neq \bfs$.
Mutation of torus seed data $\bfs$ in direction $k$ defines birational maps ${\mu_k : \mathbf{\mathcal{X}_s} \rightarrow \mathbf{\mathcal{X}_{\mu_k(s)}}}$ and ${\mu_k : \mathbf{\mathcal{A}_s} \rightarrow \mathbf{\mathcal{A}_{\mu_k(s)}}}$ via the pull-backs
\begin{align}
    \label{eq:A_pullback}
    \mu_k^* z^m &= z^m (1+z^{v_k})^{-\langle d_ke_k,m \rangle}  \textrm{ for } m \in M^{\circ}, \\
    \label{eq:X_pullback}
    \mu_k^* z^n &= z^n (1+z^{e_k})^{-[n,e_k]}  \textrm{ for } n \in N,
\end{align}
where $v_k := p_1^*(e_k)$. Explictly, using dual bases, one can compute 
$$v_k = e_k\left[ \epsilon_{ij} \right] = \sum_{j \in I_{uf}} \epsilon_{kj} f_j.$$

Some of the most iconic equations in the study of cluster algebras are the mutation relations for the cluster variables and coefficients. We can explicitly see the familiar forms of those mutation relations by applying $\mu_k^*$ to the cluster variables $x_i = z^{f_i}$ and $y_i = z^{e_i}$:
\begin{align}
    \label{eq:x_mutation}
    \mu_k^* x_i' &= \begin{cases}
                    x_k^{-1}\left( \prod\limits_{\epsilon_{kj} > 0} x_j^{\epsilon_{kj}} + \prod\limits_{\epsilon_{kj} < 0} x_j^{-\epsilon_{kj}} \right) & i = k \\
                    x_i & i \neq k
                    \end{cases} \\
    \label{eq:y_mutation}
    \mu_k^* y_i' &= \begin{cases}
                        y_i \left( 1 + y_k^{-\textrm{sgn}(\epsilon_{ik})} \right)^{-\epsilon_{ik}} & i \neq k \\
                        y_k^{-1} & i = k
                    \end{cases}
\end{align}
\begin{remark}
Equations (\ref{eq:x_mutation}) and (\ref{eq:y_mutation}) can be obtained from equations (\ref{eq:A_pullback}) and (\ref{eq:X_pullback}) by setting $n = e_i$ and $m = f_i$. For example, consider the mutation of $x_i = z^{f_i}$ and $y_i = z^{e_i}$ in direction $k$. If $i = k$, then
\begin{align*} 
    \mu_k^*(y_k') = \mu_k^*\left(z^{e_k'} \right) &= \mu_k^*\left( z^{-e_k}\right) = z^{-e_k}\left(1+z^{e_k} \right)^{-[-e_k,e_k]} = z^{-e_k} = y_k^{-1}
\end{align*}
and
\begin{align*}
    \mu_k^*(x_k') = \mu_k^*\left( z^{f_k'} \right) &= \mu_k^*\left( z^{-f_k + \sum_{j \in I_{\textrm{uf}}} [-\epsilon_{kj}]_+f_j}\right) \\
    &= z^{-f_k + \sum_{j \in I_{\textrm{uf}}} [-\epsilon_{kj}]_+f_j} \left(1 + z^{v_k} \right)^{-\langle d_ke_k, -f_k + \sum_{j \in I_{\textrm{uf}}} [-\epsilon_{kj}]_+f_j \rangle} \\
    &= z^{-f_k} \left( \prod_{j \in I_{\textrm{uf}}} z^{[-\epsilon_{kj}]_{+}f_j} \right)\left( 1 + z^{v_k} \right)^{\langle d_ke_k,f_k \rangle} \\
    &= z^{-f_k}\left( \prod_{j \in I_{\textrm{uf}}} z^{[-\epsilon_{kj}]_{+}f_j} \right) \left( 1 + \prod_{j \in I_{\textrm{uf}}} z^{\epsilon_{kj}f_j} \right)^1 \\
    &= z^{-f_k}\left( \prod_{j \in I_{\textrm{uf}},~\epsilon_{kj} < 0} z^{-\epsilon_{kj}f_j} + \prod_{j \in I_{\textrm{uf}},~\epsilon_{kj} > 0} z^{\epsilon_{kj}f_j} \right) \\
    &= x_k^{-1} \left( \prod_{\epsilon_{kj} < 0} x_j^{-\epsilon_{kj}} + \prod_{\epsilon_{kj} > 0} x_j^{\epsilon_{kj}} \right)
\end{align*}
If $i \neq k$, then
    \begin{align*}
        \mu_k^*(y_i') = \mu_k^*\left(z^{e_i'} \right) &= \mu_k^*\left( z^{e_i + [\epsilon_{ik}]_+e_k} \right) \\
        &= z^{e_i + [\epsilon_{ik}]_+e_k}\left( 1 + z^{e_k} \right)^{-[e_i + [\epsilon_{ik}]_+e_k,e_k]} \\
        &= z^{e_i}z^{[\epsilon_{ik}]_+e_k}\left( 1+z^{e_k} \right)^{-[e_i,e_k]} \\
        &= z^{e_i}z^{[\epsilon_{ik}]_+e_k}\left( 1+z^{e_k} \right)^{-\epsilon_{ik}}
    \end{align*}
If $\epsilon_{ik} > 0$, then
\begin{align*}
    z^{e_i}z^{[\epsilon_{ik}]_{+}e_k}\left( 1 + z^{e_k} \right)^{-\epsilon_{ik}} = z^{e_i}\left(1+z^{-e_k} \right)^{-\epsilon_{ik}} = z^{e_i}\left( 1 + z^{-\textrm{sgn}(\epsilon_{ik})e_k}\right)^{-\epsilon_{ik}}.
\end{align*}
If $\epsilon_{ik} < 0$, then
\begin{align*}
    z^{e_i}z^{[\epsilon_{ik}]_+e_k}\left( 1+z^{e_k} \right)^{-\epsilon_{ik}} = z^{e_i}\left( 1 + z^{-\textrm{sgn}(\epsilon_{ik})e_k} \right)^{-\epsilon_{ik}}.
\end{align*}
Hence, in both cases we have
\begin{align*}
    \mu_k^*\left(y_{i}' \right) = y_i\left( 1 + y_k^{-\textrm{sgn}(\epsilon_{ik})} \right)^{-\epsilon_{ik}}.
\end{align*}
Finally, $\mu_k^*(x_i') = \mu_k^*\left(z^{f_i'} \right) = \mu_k^*\left(z^{f_i} \right) = z^{f_i}(1+z^{v_k})^{-\langle d_ke_k,f_i \rangle} = z^{f_i} = x_i$.
\end{remark}

Let $\mathfrak{T}$ be a directed infinite rooted tree where each vertex has $|I_{\textrm{uf}}|$ outgoing edges, labeled by the elements of $I_{\textrm{uf}}$. Let $v$ be the root of the tree and associate some initial torus seed $\bfs$ with mutation class $[\bfs]$ to $v$. To indicate this choice of initial seed, we write $\mathfrak{T}_{v}$. {In some contexts, we write $\mathfrak{T}_{\mathbf{s}}$ to indicate both the rooted infinite tree and the torus seed associated to the root vertex.} An edge with label $k \in I_{\textrm{uf}}$ corresponds to mutation in direction $k$. Hence, any simple path beginning at vertex $v$ determines a sequence of mutations according to the sequence of attached edge labels. These mutation sequences determine an associated torus seed $\bfs_{w}$ for each vertex $w$ of $\mathfrak{T}_v$. We can also attach copies of $\mathcal{X}_{\bfs_w}$ and $\mathcal{A}_{\bfs_w}$ to each vertex $w$.

Proposition 2.4 of \cite{GHK} then allows the collection $\{ \mathcal{A}_{\bfs_w} \}$, where $w$ ranges over all vertices of $\mathfrak{T}$, to be glued along the open pieces where the $\mu_k$ given in Equation \eqref{eq:A_pullback} are defined. This produces a scheme $\mathcal{A}$, known as the \emph{$\mathcal{A}$ cluster variety}. Similarly, the collection $\{ \mathcal{X}_{\bfs_w} \}$ can be glued using the $\mu_k$ given in Equation~\eqref{eq:X_pullback} to obtain a scheme $\mathcal{X}$, known as the \emph{$\mathcal{X}$ cluster variety}.

The \emph{upper cluster algebra} associated to a cluster variety $V$ is defined as $\textrm{up}(V):= \Gamma(V,\mathcal{O}_V)$ \cite{CA-III}. Let $L$ be an arbitrary lattice and $T_L := \textrm{Spec}~\Bbbk[L^*]$. A \emph{global monomial} on $V = \bigcup_{\bfs} T_{L,\bfs}$ is a regular function on $V$ that restricts to a character on some torus $T_{L,\bfs}$ in the atlas for $V$. For $\mathcal{A}$-type cluster varieties, the set of global monomials is exactly the set of cluster monomials. The \emph{ordinary cluster algebra} $\textrm{ord}(V)$ is defined as the subalgebra of $\textrm{up}(V)$ generated by the set of global monomials on $V$.

\subsubsection{Principal coefficients}
\label{subsec:ordVarPrincipal}

As mentioned in the previous section, many of the important results of \cite{GHKK} were obtained via the principal coefficient case. Recall that an ordinary cluster algebra with principal coefficients, using the wide convention for exchange matrices, is an ordinary cluster algebra where the $n \times n$ exchange matrix has been extended to a $2n \times 2n$ skew-symmetrizable block matrix whose upper right block is the $n \times n$ identity matrix and whose lower right block is the $n \times n$ zero matrix \cite{FZ-IV}.  Including this additional information requires the following modifications to the fixed and torus seed data:
\begin{definition} \cite[Construction 2.11]{GHK}
\label{def:principalFixedData}
Given fixed data $\Gamma$, the fixed data for the cluster variety with principal coefficients, $\Gamma_{\textrm{prin}}$, is defined by:
\begin{itemize}
    \item The \emph{double} of the lattice $N$, $\widetilde{N} := N \oplus M^{\circ}$, with skew-symmetric bilinear form given by
    \[ \{ (n_1,m_1),(n_2,m_2) \} = \{ n_1, n_2 \} + \langle n_1, m_2 \rangle - \langle n_2, m_1 \rangle. \]
    Here, $\langle \cdot, \cdot \rangle: N \times M^{\circ} \rightarrow \mathbb{Q}$ denotes the canonical pairing given by evaluation, $\langle n, m \rangle \mapsto m(n)$.
    \item The unfrozen sublattice $\widetilde{N}_{\textrm{uf}} := N_{\textrm{uf}} \oplus 0 \cong N_{\textrm{uf}}$. 
    \item The sublattice $\widetilde{N}^{\circ} := N^{\circ} \oplus M$ of $\widetilde{N}$.
    \item The lattice $\widetilde{M} = \textrm{Hom}(\widetilde{N},\mathbb{Z}) = M \oplus N^{\circ}$.
    \item The lattice $\widetilde{M}^{\circ} = M^{\circ} \oplus N$, which has sublattice $\widetilde{M}$.
    \item The index set $\widetilde{I}$ given by the disjoint union of two copies of $I$.
    \item The unfrozen index set, $\widetilde{I}_{\textrm{uf}}$ given by thinking of the original $I_{\textrm{uf}}$ as a subset of the first copy of $I$.
    \item A collection of integers $\{ d_i \}_{i \in \widetilde{I}}$ taken such that
    within each disjoint copy of $I$, the $d_i$ agree with the original torus seed $\bfs$.
\end{itemize}
\end{definition}

For a cluster algebra with principal coefficients, there is an implicit choice of the tropical semifield $\mathbb{P} = \textrm{Trop}(y_1, \dots, y_n)$ in Definition~\ref{def:gen_clusterSeed}.

\begin{definition}[Construction 2.11 of \cite{GHK}]
\label{def:principalTorusSeed}
Given a torus seed $\bfs$, the torus seed with principal coefficients $\bfs_{\textrm{prin}}$ is defined as
\[ \bfs_{\textrm{prin}} := \{ (e_i,0), (0,f_i) \}_{i \in \widetilde{I}} \]
\end{definition}

For ease of notation, we will use $i$ and $j$ to denote indices corresponding to basis elements of the form $(e_i,0)$ and $\alpha$ and $\beta$ to denote indices corresponding to basis elements of the form $(0,f_\alpha)$. Because of the way the collection $\{ d_i \}$ is chosen, the entries of the matrix $\widetilde{\epsilon}$ defined by the principal fixed data are determined by the following relationships:
\begin{align*}
    \widetilde{\epsilon}_{ij}  = \epsilon_{ij}, \qquad \widetilde{\epsilon}_{i\beta} = \delta_{i\beta}, ~~~ \textrm{and}~~~ \widetilde{\epsilon}_{\alpha j} = -\delta_{\alpha j}.
\end{align*}
That is, $\widetilde{\epsilon}$ is a block matrix of the form
\begin{align*}
    \widetilde{\epsilon} &= \left[ \begin{array}{c|c}
        \epsilon & Id \\ \hline
        -Id & 0
    \end{array} \right]
\end{align*}
where $Id$ denotes the $|I| \times |I|$ identity matrix.

As before, the choice of $\bfs_{\textrm{prin}}$ defines dual bases for $\widetilde{M}$ and $\widetilde{M}^{\circ}$. The previous choice of a map $p^*: N \rightarrow M^{\circ}$ allows us to define a map $p^*: \widetilde{N} \rightarrow \widetilde{M}^{\circ}$ as
\begin{align*}
    p^*(e_i,0) &= (p^*(e_i),e_i), \\
    p^*(0,f_{\alpha}) &= (-f_{\alpha},0).
\end{align*}
The new map $p^*: \widetilde{N} \rightarrow \widetilde{M}^{\circ}$ is now necessarily injective. In fact, $p^*: \widetilde{N} \rightarrow \widetilde{M}^{\circ}$ is actually an isomorphism.

The choice of $\mathbf{s}_{\textrm{prin}}$ also defines the associated algebraic tori
\begin{align*}
     \mathcal{X}_{\bfs_{\textrm{prin}}} := T_{\widetilde{M}} = \textrm{Spec}~\Bbbk[\widetilde{N}], \\
     \mathcal{A}_{\bfs_{\textrm{prin}}} := T_{\widetilde{N}^{\circ}} = \textrm{Spec}~\Bbbk[\widetilde{M}^{\circ}].
\end{align*}
The principal cluster varieties $\mathcal{X}_{\textrm{prin}}$ and $\mathcal{A}_{\textrm{prin}}$ are then obtained by gluing along the birational mutation maps $\mu_k : \mathcal{X}_{\bfs_{\textrm{prin}}} \rightarrow \mathcal{X}_{\mu_k(\bfs_{\textrm{prin}})}$ and ${\mu_k : \mathcal{A}_{\bfs_{\textrm{prin}}} \rightarrow \mathcal{A}_{\mu_k(\bfs_{\textrm{prin}})}}$, as previously.

There are several important observations to make about the principal cluster varieties. First, the ring of global functions on $\mathcal{A}_{\textrm{prin}}$ is the \emph{upper cluster algebra with principal coefficients at the seed $\bfs$}. Second, $\mathcal{A}_{\textrm{prin}}$ has useful relationships with the cluster varieties $\mathcal{X}$ and $\mathcal{A}$ which arise from the natural inclusions
\begin{align*}
    \widetilde{p}^*: N &\rightarrow \widetilde{M}^{\circ}, \\
    n &\mapsto (p^*(n),n)
\end{align*}
and
\begin{align} \label{eq:pimap}
    \widetilde{\pi} : N &\rightarrow \widetilde{M}^{\circ}, \\
    n &\mapsto (0,n).
\end{align}

For any torus seed $\bfs$, the map $\widetilde{p}^*$ induces the exact sequence of algebraic tori:
\[ 1 \rightarrow T_{N^{\circ}} \rightarrow \mathcal{A}_{\bfs_{\textrm{prin}}} \xrightarrow{\widetilde{p}} \mathcal{X}_{\bfs} \rightarrow 1 \]
The map $\widetilde{p}: \mathcal{A}_{\bfs_{\textrm{prin}}} \rightarrow \mathcal{X}_{\bfs}$ defined by this exact sequence commutes with the mutations $\mu_k$ on $\mathcal{A}_{\bfs_{\textrm{prin}}}$ and $\mathcal{X}_{\bfs}$, yielding a morphism $\widetilde{p} : \Aprin \rightarrow \mathcal{X}$. Similarly, the $T_{N^{\circ}}$ action on $\mathcal{A}_{\bfs_{\textrm{prin}}}$ yields a $T_{N^{\circ}}$ action on $\Aprin$. 
The quotient $\Aprin/T_{N^{\circ}}$ is the $\mathcal{X}$-variety.

The map $\pi^*$ induces the projection $\pi: \Aprin \rightarrow T_{M}$. Let ${\mathcal{A}_{t} := \pi^{-1}(t)}$. Then the fiber $\mathcal{A}_{e}$, where $e \in T_{M}$ is the identity element, is the $\mathcal{A}$-variety.

\subsection{Cluster scattering diagrams}
\label{subsec:ordDiagrams}

Scattering diagrams first appeared in the literature in two dimensions, in work by Kontsevich and Soibelman \cite{KS-chapter}, and then in arbitrary dimension in the work of Gross and Siebert \cite{GS}. Our discussion of cluster scattering diagrams will loosely follow the structure of the exposition in Section 1 of \cite{GHKK}.

To construct a cluster scattering diagram, we begin with a choice of fixed data $\Gamma$ and initial seed data $\bfs$ and let $\Bbbk$ be a field of characteristic zero. Let $\sigma \subseteq M_{\mathbb{R}}$ be a strictly convex top-dimensional cone and define the associated monoid $P:=\sigma \cap M^{\circ}$ such that $p_1^*(e_i) \in J := P \backslash P^{\times}$ for all $i \in I_{\textrm{uf}}$. Here, $P^{\times} = \{ 0 \}$ is the group of units of $P$ and $J$ is a monomial ideal in the polynomial ring $\Bbbk[P]$. Let $\widehat{\Bbbk[P]}$ denote the completion of $\Bbbk[P]$ with respect to $J$.

The construction also requires the assumption that $p_1^*: N_{\textrm{uf}} \rightarrow M^\circ$ is an injective map. It is important to note that this assumption does not hold for all choices of fixed data, but does hold for fixed data corresponding to the principal coefficient case. Because arbitrary cluster algebras can be considered as specializations of the principal coefficient case, it is sufficient for the injectivity assumption to hold for that case.

Set
\[N^+ := N_{\bfs}^+ := \left\{ \left. \sum_{i \in I_{\textrm{uf}}} a_ie_i \right| a_i \geq 0, \sum a_i > 0 \right\} \]
and choose a linear function $d: N \rightarrow \mathbb{Z}$ such that $d(n) > 0$ for $n \in N^+$.
\begin{definition} \cite[Definition 1.4]{GHKK}
\label{def:walls1}
    A \emph{wall} in $M_\mathbb{R}$ is a pair $(\mathfrak{d},f_\mathfrak{d}) \in (N^+,\widehat{\Bbbk[P]})$ such that for some primitive $n_0 \in N^{+}$,
    \begin{enumerate}
        \item $f_\mathfrak{d} \in \widehat{\Bbbk[P]}$ has the form $1 + \sum_{j=1}^{\infty} c_j z^{jp_1^*(n_0)}$ with $c_j \in \Bbbk$
        \item $\mathfrak{d} \subset n_0^{\perp} \subset M_\mathbb{R}$ is a $({\mathrm{rank}(M) -1})$-dimensional convex rational polyhedral cone.
    \end{enumerate}
    We refer to $\mathfrak{d} \subset M_\mathbb{R}$ as the \emph{support} of the wall $(\mathfrak{d},f_\mathfrak{d})$.
\end{definition}

Let $\mathfrak{m}$ denote the ideal in $\widehat{\Bbbk[P]}$ which consists of formal power series with constant term zero.

\begin{definition} \cite[Definition 1.6]{GHKK}
\label{def:1.6-GHKK}
A \emph{scattering diagram} $\mathfrak{D}$ for $N^+$ and $\bfs$ is a set of walls $\{ (\mathfrak{d},f_{\mathfrak{d}}) \}$ such that for every degree $k > 0$, there are a finite number of walls $(\mathfrak{d},f_{\mathfrak{d}}) \in \mathfrak{D}$ with $f_{\mathfrak{d}} \neq 1 \mod \mathfrak{m}^{k+1}$.
\end{definition}

For a scattering diagram $\mathfrak{D}$,
\begin{align*}
    \textrm{Supp}(\mathfrak{D}) &:= \bigcup_{\mathfrak{d} \in \mathfrak{D}} \mathfrak{d}, \\
    \textrm{Sing}(\mathfrak{D}) &:= \left(\bigcup_{\mathfrak{d} \in \mathfrak{D}} \partial\mathfrak{d} \right) \cup \left( \bigcup_{\stackrel{\mathfrak{d}_1, \mathfrak{d}_2 \in \mathfrak{D}}{\textrm{dim}(\mathfrak{d}_1 \cap \mathfrak{d}_2) = n-2}} \mathfrak{d}_1 \cap \mathfrak{d}_2 \right)
\end{align*}
are defined as the support and singular locus of the scattering diagram. When $\mathfrak{D}$ is finite, its support is a finite polyhedral cone complex. A $(n-2)$-dimensional cell of this complex is referred to as a \emph{joint}. In this case, $\textrm{Sing}(\mathfrak{D})$ is simply the union of the set of all joints of $\mathfrak{D}$. A wall $\mathfrak{d} \subset n_0^\perp$ is called \emph{incoming} if $p_1^*(n_0) \in \mathfrak{d}$. Otherwise, $\mathfrak{d}$ is called \emph{outgoing}.

Each wall $\mathfrak{d} \in \mathfrak{D}$ has an associated wall-crossing automorphism.

\begin{definition} \cite[Definition 1.2]{GHKK}
\label{def:wallcrossing}
For $n_0 \in N^{+}$, let $m_0 := p_1^*(n_0)$ and $f = 1 + \sum_{k = 1}^{\infty} c_kz^{km_0} \in \widehat{\Bbbk[P]}$. Then $\mathfrak{p}_{f} \in \widehat{\Bbbk[P]}$ denotes the automorphism
\[ \mathfrak{p}_f(z^m) = z^mf^{\langle n_0',m \rangle} \]
where $n_0'$ generates the monoid $\mathbb{R}_{\geq 0}n_0 \cap N^{\circ}$.
\end{definition}

These wall-crossing automorphisms can be composed in order to compute automorphisms associated to paths on the scattering diagram that pass through multiple walls. Such compositions are called \emph{path-ordered products}.

\begin{definition}
Let $\gamma: [0,1] \rightarrow M_\mathbb{R} \backslash \textrm{Sing}(\mathfrak{D})$ be a smooth immersion which crosses walls transversely and whose endpoints are not in the support of $\mathfrak{D}$. Let $0 < t_1 \leq t_2 \leq \cdots \leq t_s < 1$ be a sequence such that at time $t_i$ the path $\gamma$ crosses the wall $\mathfrak{d}_i$ such that $f_i \neq 1~\mod \mathfrak{m}^{k+1}$. Definition~\ref{def:1.6-GHKK} ensures that this is a finite sequence. For each $i \in \{ 1, \dots, s \}$, set $\epsilon_i := -\textrm{sgn}(\langle n_i, \gamma'(t_i) \rangle)$ where $n_i \in N^{+}$ is the primitive vector normal to $\mathfrak{d}_i$. For each degree $k > 0$, define
\[ \mathfrak{p}_{\gamma,\mathfrak{D}}^k := \mathfrak{p}_{f_{\mathfrak{d}_{t_s}}}^{\epsilon_s} \circ \cdots \circ \mathfrak{p}_{f_{\mathfrak{d}_{t_1}}}^{\epsilon_1},  \]
where $\mathfrak{p}_{f_{\mathfrak{d}_{t_i}}}$ is defined as in Definition \ref{def:wallcrossing}. Then,
\[ \mathfrak{p}_{\gamma,\mathfrak{D}} := \lim_{k \rightarrow \infty} \mathfrak{p}_{\gamma,\mathfrak{D}}^k. \]
We refer to $ \mathfrak{p}_{\gamma,\mathfrak{D}}$ as a \emph{path-ordered product}.
\end{definition}

A  scattering diagram $\mathfrak{D}$ is \emph{consistent} if $\mathfrak{p}_{\gamma,\mathfrak{D}}$ depends only on the endpoints of $\gamma$. Two scattering diagrams, $\mathfrak{D}$ and $\mathfrak{D}'$, are considered \emph{equivalent} if  $\mathfrak{p}_{\gamma,\mathfrak{D}} = \mathfrak{p}_{\gamma,\mathfrak{D}'}$ for all paths $\gamma$ for which both path-ordered products are defined.

Gross, Hacking, Keel, and Kontsevich consider a particular scattering diagram, which we refer to as the \emph{cluster scattering diagram}. This diagram is defined by the fixed and torus seed data as follows.

\begin{definition}
Given a set of fixed data $\Gamma$ and torus seed $\bfs$, let $v_i = p_1^*(e_i)$ for all $i \in I_{\textrm{uf}}$. The \emph{initial scattering diagram} $\mathfrak{D}_{\textrm{in},\bfs}$ is defined as
\[ \mathfrak{D}_{\textrm{in},\bfs} := \{ (e_i^{\perp},1+z^{v_i}) : i \in I_{\textrm{uf}} \} \]
\end{definition}

The following theorems are two of the main results of \cite{GHKK}.

\begin{theorem} \cite[Theorem 1.12, 1.13]{GHKK} There is a scattering diagram $\mathfrak{D}_{\bfs}$ such that
\begin{enumerate}
    \item $\mathfrak{D}_{\bfs}$ is consistent
    \item $\mathfrak{D}_{\bfs} \supset \mathfrak{D}_{in,\bfs}$
    \item $\mathfrak{D}_{\bfs} \backslash \mathfrak{D}_{in,\bfs}$ consists of only outgoing walls
\end{enumerate}
The scattering diagram $\mathfrak{D}_{\bfs}$ is equivalent to a scattering diagram whose walls $(\mathfrak{d},f_{\mathfrak{d}})$ all have wall-crossing automorphisms of the form $f_{\mathfrak{d}} = (1 + z^m)^c$ for some $m = p^*(n)$, $n \in N^{+}$, and positive integer $c$. In particular, all the nonzero coefficients of $f_{\mathfrak{d}}$ are positive integers. The diagram is unique up to equivalence. 
\end{theorem}

\begin{example}
Consider the ordinary cluster algebra $\mathcal{A}\left(\mathbf{x},\mathbf{y},\begin{bmatrix} 0 & 1 \\ -1 & 0 \end{bmatrix} \right)$ with seed data $\bfs = ((1,0),(0,1))$. This algebra has initial scattering diagram
\[ \mathfrak{D}_{\textrm{in},\bfs} = \left\{ \left( (0,1)^{\perp}, 1 + z^{(-1,0)} \right), \left( (1,0)^{\perp}, 1 + z^{(0,1)} \right) \right\}, \]
which can be drawn as
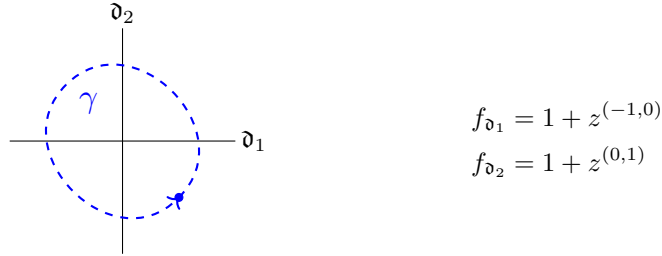
\begin{center}
    \begin{minipage}{0.3\textwidth}
    \begin{tikzpicture}
        \draw (-1.5,0) to (1.5,0);
        \draw (0,-1.5) to (0,1.5);
        
        \node at (1.75,0) {$\mathfrak{d}_1$};
        \node at (0,1.7) {$\mathfrak{d}_2$};
        
        \draw[dashed,thick,blue,out=45,in=45,looseness=1.5] (0.75,-0.75) to (-0.75,0.75);
        \draw[dashed,thick,blue,->,looseness=1.5,out=225,in=225] (-0.75,0.75) to (0.75,-0.75);
        
        \node[blue,scale=1.25] at (-0.45,0.5) {$\gamma$};
        
        \filldraw[blue] (0.75,-0.75) circle (1.5pt);
    \end{tikzpicture}
    \end{minipage}
    \begin{minipage}{0.3\textwidth}
        \vspace{20mm}
        \begin{align*}
            f_{\mathfrak{d}_1} &= 1 + z^{(-1,0)} \\
            f_{\mathfrak{d}_2} &= 1 + z^{(0,1)}
        \end{align*}
        \vspace{15mm}
    \end{minipage}
\end{center}
We can see that $\mathfrak{D}_{\textrm{in},\bfs}$ is not consistent by computing $\mathfrak{p}_{\gamma,\mathfrak{D}_{\textrm{in},\bfs}}(z^{(0,1)})$ as:
\begin{align*}
    z^{(0,1)} &\xmapsto{\mathfrak{d}_1} z^{(0,1)}\left(1 + z^{(-1,0)} \right)^{\langle (0,1),(0,-1) \rangle} \\
    &= \frac{z^{(0,1)}}{1 + z^{(-1,0)}} \\
    &\xmapsto{\mathfrak{d}_2} \frac{z^{(0,1)}\left(1 + z^{(0,1)} \right)^{\langle (0,1),(1,0) \rangle}}{1 + z^{(-1,0)}\left( 1 + z^{(0,1)} \right)^{\langle (-1,0),(1,0) \rangle}} \\
    &= \frac{z^{(0,1)}\left(1 + z^{(0,1)} \right)}{1 + z^{(0,1)} + z^{(-1,0)}} \\
    &\xmapsto{\mathfrak{d}_1}  \frac{z^{(0,1)}\left(1 + z^{(-1,0)} \right)^{\langle (0,1),(0,1) \rangle}\left(1 + z^{(0,1)}\left(1 + z^{(-1,0)} \right)^{\langle (0,1),(0,1) \rangle} \right)}{1 + z^{(0,1)}\left(1 + z^{(-1,0)} \right)^{\langle (0,1),(0,1) \rangle} + z^{(-1,0)}\left(1 + z^{(-1,0)} \right)^{\langle (-1,0),(0,1) \rangle}} \\
    &=\frac{z^{(0,1)}\left(1 + z^{(0,1)} + z^{(-1,1)}\right)}{1 + z^{(0,1)}} \\
    &\xmapsto{\mathfrak{d}_2} \frac{z^{(0,1)}\left(1 + z^{(0,1)} \right)^{\langle (0,1),(-1,0) \rangle}\left(1 + z^{(0,1)}\left(1 + z^{(0,1)} \right)^{\langle (0,1),(-1,0) \rangle} + z^{(-1,1)}\left(1 + z^{(0,1)} \right)^{\langle (-1,1),(-1,0) \rangle}\right)}{1 + z^{(0,1)}\left(1 + z^{(0,1)} \right)^{\langle (0,1),(-1,0) \rangle}} \\
    &= z^{(0,1)}\left(1 + z^{(-1,1)} \right)
\end{align*}
In a consistent diagram, for any loop $\gamma$, we should get $\mathfrak{p}_{\gamma,\mathfrak{D}_{\textrm{in},\bfs}}\left( z^{m} \right) = z^{m}$. Whereas, here $\mathfrak{p}_{\gamma,\mathfrak{D}_{\textrm{in},\bfs}}\left( z^{(0,1)} \right) \neq z^{(0,1)}$. Making the diagram $\mathfrak{D}_{\textrm{in},\bfs}$ consistent requires adding the single wall $\mathfrak{d}_3 = \left(\mathbb{R}_{\geq 0}(1,-1), 1 + z^{(-1,1)} \right)$. Following the same loop $\gamma$ as above, our calculation now has the additional step
\begin{align*}
    z^{(0,1)}\left(1 + z^{(-1,1)} \right) &\xmapsto{\mathfrak{d}_3} z^{(0,1)}\left(1 + z^{(-1,1)} \right)^{\langle (0,1),(-1,-1) \rangle} \cdot \\
    & \qquad \qquad \left(1 + z^{(-1,1)}\left(1 + z^{(-1,1)} \right)^{\langle (-1,1),(-1,-1) \rangle} \right) \\
    &= \frac{z^{(0,1)}}{1 + z^{(-1,1)}}\left( 1 + z^{(-1,1)} \right) \\
    &= z^{(0,1)}
\end{align*}
and now $\mathfrak{p}_{\gamma,\mathfrak{D}_{\bfs}}\left( z^{(0,1)} \right) = z^{(0,1)}$. In this example, the necessary wall-crossing automorphism can be seen by inspection of $\mathfrak{p}_{\gamma,\mathfrak{D}_{\textrm{in},\bfs}}\left( z^{(0,1)} \right)$. In cluster scattering diagrams with more than one outgoing wall, it is difficult to determine the support and associated wall-crossing automorphisms of those walls by inspection of a similar calculation. Instead, there is a simple algorithm to produce $\mathfrak{D}_{\bfs}$ from $\mathfrak{D}_{\textrm{in},\bfs}$ which was introduced in two dimensions by Kontsevich and Soibelman \cite{KS-chapter} and then extended for arbitrary dimension by Gross and Siebert \cite{GS}.
\end{example}

\subsubsection*{Mutation invariance of ordinary cluster scattering diagrams}
\label{subsec:ordMutInvar}

An ordinary cluster algebra can be equivalently specified by any possible choice of initial cluster - there is no particular canonical choice. In the language of scattering diagrams, this means there should be no special choice of torus seed data. If two torus seeds, $\bfs$ and $\bfs'$ are mutation equivalent, we should therefore expect that the corresponding cluster scattering diagrams $\mathfrak{D}_{\bfs}$ and $\mathfrak{D}_{\bfs'}$ are also equivalent. This expectation reflects the fact that $\mathfrak{D}_{\bfs}$ and $\mathfrak{D}_{\bfs'}$  encode information about the same ordinary cluster algebra.

In order to make the notion of mutation invariance precise, we must define the following half-spaces and piecewise linear transformation:

\begin{definition}[Definition 1.22 of \cite{GHKK}]
\label{def:ordT_k}
For $k \in I_{\textrm{uf}}$, define
\begin{align*}
 \mathcal{H}_{k,+} &:= \{ m \in M_\mathbb{R} : \langle e_k,m \rangle \geq 0 \}, \\
 \mathcal{H}_{k,-} &:= \{ m \in M_\mathbb{R} : \langle e_k,m \rangle \leq 0 \}.
\end{align*}
The piecewise linear transformation $T_k: M^{\circ} \rightarrow M^{\circ}$ is defined as
\[ T_k(m) := \begin{cases}
                m + v_k\langle d_ke_k,m \rangle & m \in \mathcal{H}_{k,+} \\
                m & m \in \mathcal{H}_{k,-}
            \end{cases}\]
The shorthand notation $T_{k,-}$ and $T_{k,+}$ is sometimes used to refer to $T_k$ in the respective regions $\mathcal{H}_{k,-}$ and $\mathcal{H}_{k,+}$.
\end{definition}

Intuitively, the map $T_k$ gives us a way to ``mutate" cluster scattering diagrams. By applying the map $T_k$ to $\mathfrak{D}_{\bfs}$, we obtain a new cluster scattering diagram $T_k(\mathfrak{D}_{\bfs})$ via the following algorithm (Definition 1.22 of \cite{GHKK}):
\begin{enumerate}
    \item The wall $\mathfrak{d}_k = (e_k^{\perp}, 1 + z^{v_k})$ is replaced by $\mathfrak{d}_k' := (e_k^{\perp}, 1 + z^{-v_k})$,
    \item For each wall $(\mathfrak{d},f_{\mathfrak{d}}) \in \mathfrak{D}_{\bfs} \backslash \{ \mathfrak{d}_k \}$, there are either one or two walls in $T_k(\mathfrak{D}_{\bfs})$. The potential walls are
    \[ (T_k(\mathfrak{d} \cap \mathcal{H}_{k,-}),T_{k,-}(f_{\mathfrak{d}})) ~~ \textrm{and} ~~ (T_k(\mathfrak{d} \cap \mathcal{H}_{k,+}),T_{k,+}(f_{\mathfrak{d}})), \]
    where $T_{k,\pm}(f_{\mathfrak{d}})$ denotes the formal power series obtained by applying $T_{k,\pm}$ to each exponent in $f_{\mathfrak{d}}$. The first wall is dropped if ${\textrm{dim}(\mathfrak{d}) \cap \mathcal{H}_{k,-} < \textrm{rank}(M)-1}$ and the second wall is dropped if $\textrm{dim}(\mathfrak{d}) \cap \mathcal{H}_{k,+} < \textrm{rank}(M)-1$.
\end{enumerate}
The following theorem justifies why we should think of the action of $T_k$ as mutation of the cluster scattering diagram.

\begin{theorem}\cite[Theorem 1.24]{GHKK}
\label{theorem:mutationInvariance}
If the injectivity assumption holds, then $T_k(\mathfrak{D}_{\bfs})$ is a consistent scattering diagram and the diagrams $\mathfrak{D}_{\mu_k(\bfs)}$ and $T_k(\mathfrak{D}_{\bfs})$ are equivalent.
\end{theorem}

Applying Theorem~\ref{theorem:mutationInvariance} multiple times gives the equivalence of ${T_{k_\ell} \circ \cdots \circ T_{k_1} (\mathfrak{D}_{\bfs})}$ and $\mathfrak{D}_{\bfs'}$ where $\bfs$ and $\bfs'$ are related by an arbitrary  sequence of mutations $\mu_{k_1}, \dots, \mu_{k_\ell}$, i.e.
$\bfs' = \mu_{k_\ell} \circ \cdots \circ \mu_{k_1}(\bfs)$.

One frequently used semifield structure on $\mathbb{Z}$ is the max-plus structure, where addition of elements is taken to be the maximum and multiplication to be ordinary addition. We denote this semifield structure as $\mathbb{Z}^T$ and refer to it as the \emph{Fock-Goncharov tropicalization}. The map $T_k$ given in Definition~\ref{def:ordT_k} is in fact the Fock-Goncharov tropicalization of the birational mutation map $\mu_k: \mathcal{A}_{\bfs} \rightarrow \mathcal{A}_{\bfs}$ defined by the pullback $\mu_k^*z^m = z^m(1+z^{v_k})^{- \langle d_ke_k, m \rangle}$ as in Equation~\eqref{eq:A_pullback}. 

\subsubsection*{Chamber Structure}

The map $T_k$ also gives rise to a chamber structure on the cluster scattering diagram $\mathfrak{D}_{\bfs}$. Within $\mathfrak{D}_{\bfs}$, there are two important named chambers.

\begin{definition}
Given a torus seed $\bfs$, we define $\mathcal{C}_{\bfs}^{\pm} \subseteq M_\mathbb{R}$ as
\begin{align*}
    \mathcal{C}^+_\bfs &:= \{ \left. m \in M_\mathbb{R} \right| \langle e_i,m \rangle \geq 0 \textrm{ for } i \in I_\textrm{uf} \}, \\
    \mathcal{C}^{-}_\bfs &:= \{ \left. m \in M_\mathbb{R} \right| \langle e_i,m \rangle \leq 0 \textrm{ for } i \in I_\textrm{uf} \}.
\end{align*}
When $\bfs$ is clear from context, we omit the subscript and simply write $\mathcal{C}^\pm$. We refer to $\mathcal{C}^+$ as the \emph{positive chamber} and $\mathcal{C}^{-}$ as the \emph{negative chamber}.
\end{definition}

The chambers $\mathcal{C}_{\bfs}^{\pm}$ are closures of connected components of $M_{\mathbb{R}} \backslash \textrm{Supp}(\mathfrak{D}_{\bfs})$. Similarly, the chambers  $\mathcal{C}_{\mu_k(\bfs)}^{\pm}$ are closures of connected components of $M_{\mathbb{R}} \backslash \textrm{Supp}( \mathfrak{D}_{\mu_k(\bfs)})$. We can observe that this means the chambers $T_k^{-1}( \mathcal{C}_{\mu_k(\bfs)}^{\pm} )$ are closures of connected components of $M_{\mathbb{R}} \backslash \textrm{Supp}(\mathfrak{D}_{\bfs})$. Further, $\mathcal{C}_{\bfs}^{\pm}$ and $T_k^{-1}( \mathcal{C}_{\mu_k(\bfs)}^{\pm})$ share a codimension one face with support $e_k^{\perp}$. This creates the following chamber structure on a subset of $M_{\mathbb{R}} \backslash \textrm{Supp}(\mathfrak{D}_{\bfs})$.

Let $\mathfrak{T}_v$ be the infinite tree with root $v$ defined in Section~\ref{subsec:ordVarieties}. Let $w$ be an arbitrary distinct vertex of $\mathfrak{T}_{\bfs}$. Then the sequence of edges between $v$ and $w$ determines a map $T_w: T_{k_{\ell}} \circ \cdots \circ T_{k_{1}}: M_{\mathbb{R}} \rightarrow \mathbb{R}$. By Theorem~\ref{theorem:mutationInvariance}, we know that $T_w(\mathfrak{D}_{\bfs}) = \mathfrak{D}_{\bfs_w}$ and so the chambers $\mathcal{C}_{w}^{\pm} := T_w^{-1}(\mathcal{C}_{\bfs_w}^{\pm})$ are closures of connected components of $M_{\mathbb{R}} \backslash \textrm{Supp}(\mathfrak{D}_{\bfs})$.

\begin{definition}[Definition 1.32 of \cite{GHKK}]
Let $\Delta^{+}_{\bfs}$ denote the set of chambers $\{ \mathcal{C}_{w}^{\pm} \}$ where $w$ runs over the vertices of $\mathfrak{T}_{\bfs}$. The elements of $\Delta^{+}_{\bfs}$ are referred to as \emph{cluster chambers}.
\end{definition}

In fact, this chamber structure coincides with the \emph{Fock-Goncharov cluster complex} \cite{FG}. For further details, see Construction 1.30 and Section 2 of \cite{GHKK}.

The chamber structure of a cluster scattering diagram is determined by the structural properties of the associated cluster algebra. The Laurent expansion of a cluster variable with respect to a particular initial cluster can be specified via two statistics: its $F$-polynomial and its $g$-vector. In \cite{FZ-IV}, Fomin and Zelevinsky give a definition of $g$-vectors in terms of a particular $\mathbb{Z}^n$-grading of the ring of Laurent polynomials in $\mathbf{x}$ whose coefficients are integer polynomials in $\mathbf{y}$. Gross, Hacking, Keel, and Kontsevich \cite{GHKK} give a description of $g$-vectors as the tropical points of theta functions. In the context of cluster scattering diagrams, there is a correspondence between the $g$-vectors and the support of the walls of the diagram. There is then also a correspondence between the cluster chambers and initial seeds of the cluster algebra.

\subsection{Theta basis for ordinary cluster algebras}
\label{subsec:ordThetaBasis}

One of the major results of the work of Gross, Hacking, Keel, and Kontsevich \cite{GHKK} is the existence of the \emph{theta basis}, a canonical basis for ordinary cluster algebras. The collections of the \emph{theta functions} which form this basis can be defined on scattering diagrams via combinatorial objects called \emph{broken lines}.

\begin{definition}\cite[Definition 3.1]{GHKK}
\label{def:brokenline}
Let $\mathfrak{D}$ be a scattering diagram, $m_0$ be a point in $M^\circ \backslash \{ 0 \}$, and $Q$ be a point in $M_\mathbb{R} \backslash \textrm{Supp}(\mathfrak{D})$. A \emph{broken line} with endpoint $Q$ and initial slope $m_0$ is a piecewise linear path $\gamma: (-\infty,0] \rightarrow M_\mathbb{R} \backslash \textrm{Sing}(\mathfrak{D})$ with finitely many domains of linearity. Each domain of linearity, $L$, has an associated monomial $c_Lz^{m_L} \in \Bbbk[M^\circ]$ such that the following conditions are satisfied:
\begin{enumerate}
    \item $\gamma(0) = Q$
    \item If $L$ is the first domain of linearity of $\gamma$, then $c_Lz^{m_L} = z^{m_0}$.
    \item Within the domain of linearity $L$, the broken line has slope $-m_L$ - in other words, $\gamma'(t) = -m_L$ on $L$.
    \item Let $t$ be a point at which $\gamma$ is non-linear and is passing from one domain of linearity, $L$, to another domain of linearity, $L'$, and define
    \[ \mathfrak{D}_t = \{ (\mathfrak{d},f_\mathfrak{d}) \in \mathfrak{D} : \gamma(t) \in \mathfrak{d}. \} \]
    Then the formal power series $\mathfrak{p}_{\left.\gamma\right|_{(t-\epsilon,t+\epsilon)},\mathfrak{D}_t}(c_Lz^{m_L})$ contains the term $c_{L'}z^{m_{L'}}$.
\end{enumerate}
\end{definition}

Broken lines allow for a beautifully concrete and combinatorial definition of a theta function:

\begin{definition}\cite[Definition 3.3]{GHKK}
\label{def:thetaFunction}
Suppose $\mathfrak{D}$ is a scattering diagram and consider points $m_0 \in M^\circ \backslash \{ 0 \}$ and $Q \in M_\mathbb{R} \backslash \textrm{Supp}(\mathfrak{D})$. For a broken line $\gamma$ with initial exponent $m_0$ and endpoint $Q$, we define $I(\gamma) = m_0$, $b(\gamma) = Q$, and $\textrm{Mono}(\gamma) = c(\gamma)z^{F(\gamma)}$ where $\textrm{Mono}(\gamma)$ is the monomial attached to the final domain of linearity of $\gamma$. We then define
\[ \vartheta_{Q,m_0} := \sum_\gamma \textrm{Mono}(\gamma) \]
where the summation ranges over all broken lines $\gamma$ with initial exponent $m_0$ and endpoint $Q$. When $m_0 = 0$, then for any endpoint $Q$ we define $\vartheta_{Q,0} = 1$.
\end{definition}

One of the key steps in proving that the theta functions form a basis is to show that the cluster variables and \emph{cluster monomials}, i.e. products of cluster variables from a particular cluster, are themselves theta functions. Although we will not reproduce the full proof, we will highlight several important intermediate properties and results.
For the full set of definitions and technical details of the proof, we refer the reader to \cite[Sections 3, 4, 6, and 7]{GHKK}.

One important property is that theta functions with the same initial slope $m_0$ but with distinct endpoints $Q$ and $Q'$ are related by a path-ordered product.

\begin{theorem}\cite[Theorem 3.5]{GHKK}
\label{theorem:thetafuncomposition}
Let $\mathfrak{D}$ be a consistent scattering diagram, $m_0$ be a point in $M \backslash \{ 0 \}$, and consider a pair of points $Q$ and $Q'$ in $M_\mathbb{R} \backslash \textrm{Supp}(\mathfrak{D})$ such that $Q$ and $Q'$ are linearly independent over $\mathbb{Q}$. Then for any path $\gamma$ with endpoints $Q$ and $Q'$ for which $\mathfrak{p}_{\gamma,\mathfrak{D}}$ is defined, we have
\[ \vartheta_{Q',m_0} = \mathfrak{p}_{\gamma,\mathfrak{D}}(\vartheta_{Q,m_0}) \]
\end{theorem}

Another important property is the existence of a bijection between broken lines in the  diagram $\mathfrak{D}_{\bfs}$ and in the diagram $\mathfrak{D}_{\mu_k(\bfs)}$. Because the diagrams $\mathfrak{D}_{\bfs}$ and $\mathfrak{D}_{\mu_k(\bfs)}$ correspond to the same cluster algebra, this is a clearly desirable property if the theta functions are going to form a canonical basis. That is, any choice of initial cluster (i.e., cluster corresponding to the positive chamber) should produce the same canonical basis, up to isomorphism.
For $V=\mathcal{A}_t$,
\begin{prop}\cite[Proposition 3.6]{GHKK}
\label{prop:brokenLineBijectionOrd}
The transformation $T_k$ gives a bijection between broken lines with endpoint $Q$ and initial slope $m_0$ in $\mathfrak{D}_\bfs$ and broken lines with endpoint $T_k(Q)$ and initial slope $T_k(m_0)$ in $\mathfrak{D}_{\mu_k(\bfs)}$. In particular,
\begin{align*}
    \vartheta^{\mu_k(\bfs)}_{T_k(Q),T_k(m_0)} &= \begin{cases}
            T_{k,+}\left( \vartheta^\bfs_{Q,m_0} \right) & Q \in \mathcal{H}_{k,+} \\
            T_{k,-}\left( \vartheta^\bfs_{Q,m_0} \right) & Q \in \mathcal{H}_{k,-}
        \end{cases}
\end{align*}
where the superscript indicates which scattering diagram is used to define the theta function and  $T_{k,\pm}$ acts linearly on the exponents in $\vartheta^\bfs_{Q,m_0}$.
\end{prop}

Although Proposition~\ref{prop:brokenLineBijectionOrd} gives a bijection between cluster scattering diagrams generated by seeds that are related by a single mutation, repeated applications of the proposition yield such bijections for any pair of cluster scattering diagrams which correspond to the same cluster algebra.

Finally, the theta functions have structure constants with an elegant combinatorial definition in terms of pairs of broken lines. Let $p_1, p_2,$ and $q$ be points in $\widetilde{M}_{\bfs}^{\circ}$ and $z$ be a generic point in $\widetilde{M}_{\mathbb{R},\bfs}^{\circ}$. There are finitely many pairs of broken lines $\gamma_1, \gamma_2$ such that $\gamma_i$ has initial slope $p_i$, both $\gamma_1$ and $\gamma_2$ have endpoint $z$, and the sum of the slopes of their final domains of linearity is $q$. Define
\[ \alpha_{z}(p_1,p_2,q) := \sum_{\substack{ (\gamma_1,\gamma_2) \\ I(\gamma_i) = p_i, b(\gamma_i) = z \\ F(\gamma_1) + F(\gamma_2) = q}} c(\gamma_1)c(\gamma_2). \]
For points $z$ sufficiently close to $q$, the constant $\alpha_{z}(p_1,p_2,q)$ is independent of the choice of $z$ and we write $\alpha(p_1,p_2,q) := \alpha_{z}(p_1,p_2,q)$ \cite[Proposition 6.4]{GHKK}. Products of theta functions can then be written as
\[ \vartheta_{p_1} \cdot \vartheta_{p_2} = \sum_{q \in \widetilde{M}_{\bfs}} \alpha(p_1,p_2,q) \vartheta_{q} ~\in~ \widehat{\textrm{up}(\overline{\mathcal{A}}_{\textrm{prin}}^{\bfs})} \otimes_{\Bbbk[N_{\bfs}^{+}]} \Bbbk[N]. \]
Hence, the theta functions form a legitimate vector space basis.

Moving from proving results about theta functions on scattering diagrams to proving results about the ordinary cluster algebras requires formalizing the connection between cluster scattering diagrams and cluster varieties. To do so, Gross, Hacking, Keel, and Kontsevich construct a space $\mathcal{A}_{\textrm{scat}}$ from the cluster scattering diagram $\mathfrak{D}_{\bfs}$ by attaching a copy of the torus $T_{N^{\circ}}$ to each cluster chamber of $\mathfrak{D}_{\bfs}$ and then gluing these copies according to the birational maps given by the wall-crossing automorphisms. Up to isomorphism, this space is independent of the choice of torus seed $\bfs$ within a given mutation class. Gross, Hacking, Keel, and Kontsevich then show that the space $\Ascat$ is isomorphic to the cluster variety $\A_{\bfs}$.

Once this identification is made, it is then possible to formalize the relationship between the theta functions and cluster monomials. Consider a set of fixed data $\Gamma$ and torus seeds $\bfs$,  $\bfs_{w} = (e_1', \dots, e_n')$. In this geometric context, a \emph{cluster monomial on $\bfs_{w}$} is defined as a monomial on $T_{N^{\circ},w} \subset \A$ of the form $z^m$ where $m = \sum_{i=1}^n a_if_i'$ with all $a_i$ non-negative. Such monomials extend to regular functions on $\mathcal{A}$. A \emph{cluster monomial on $\mathcal{A}$} is then a regular function which is a cluster monomial on some torus seed of $\mathcal{A}$. The following theorem identifies the cluster monomials on $\mathcal{A}$ with theta functions.

\begin{theorem}\cite[Theorem 4.9]{GHKK}
\label{theorem:positivityPhenomenon}
Let $\Gamma$ be a set of fixed data which satisfies the Injectivity Assumption and $\bfs$ be a choice of torus seed. Consider a point $Q \in \mathcal{C}_{\bfs}^{+}$ and $m \in \sigma \cap M^{\circ}$ for some cluster chamber $\sigma \in \Delta_{\bfs}^{+}$. Then $\vartheta_{Q,m}$ is a positive Laurent polynomial which expresses a cluster monomial of $\mathcal{A}$ in terms of the initial torus seed $\bfs$. Further, all cluster monomials can be expressed in this way.
\end{theorem}
 Gross, Hacking, Keel, and Kontsevich  define the \emph{middle cluster algebra} in order to prove that the theta functions form a basis for the ordinary cluster algebra. Their proof works primarily with $\mathcal{A}_{\textrm{prin}}$ and is then extended to $\mathcal{A}$ and $\mathcal{X}$ by using the fact that these varieties appear as a fiber and quotient, respectively, of $\mathcal{A}_{\textrm{prin}}$. Let $\mathcal{A}_{\textrm{prin}}^{\vee}$ denote the \emph{Langlands dual} of $\mathcal{A}_{\textrm{prin}}$, defined in Section~\ref{sec:duality}.  The \emph{middle cluster algebra} associated to $\Aprin$ is then defined as
\[ \textrm{mid}\left(\mathcal{A}_{\textrm{prin}} \right) := \bigoplus_{q \in \Theta} \Bbbk \cdot \vartheta_q, \]
where $\Theta \subset \Aprin^{\vee}(\mathbb{Z}^T)$ is the collection of $m_0$ such that for any generic point ${Q \in \sigma \in \Delta^{+}}$, there are only finitely many broken lines with initial slope $m_0$ and endpoint $Q$. The structure constants $\alpha_{z}(p_1,p_2,q)$ make $\textrm{mid}\left( \Aprin \right)$ into an associative and commutative $\Bbbk[N]$-algebra.

Gross, Hacking, Keel, and Kontsevich show there are canonical inclusions
\[ \textrm{ord}\left(\Aprin \right) \subset \textrm{mid}\left( \Aprin \right) \subset \textrm{up}\left(\Aprin \right), \]
and that therefore the theta functions form a basis for the ordinary cluster algebra when the ordinary cluster algebra and upper cluster algebra coincide. For the definitions of $\textrm{ord}(\mathcal{A}_{\textrm{prin}})$ and $\textrm{up}(\mathcal{A}_{\textrm{prin}})$, see Section~\ref{subsec:ordVarieties}.

\section{Generalized cluster scattering diagrams}
\label{sec:genScatDiag}

In this section, we describe the construction of cluster scattering diagrams for reciprocal generalized cluster algebras and prove important properties of such diagrams. Definitions of generalized fixed data, generalized torus seed data, generalized cluster varieties, and other fundamental objects are given in Section~\ref{subsec:genBasicDefinitions}. The construction of generalized cluster scattering diagrams is then given in Section~\ref{subsec:genScatDiagConstruction}. The restriction to reciprocal generalized cluster algebras will be necessary in the proof of Theorem~\ref{theorem:mutDiagConsistent}, which establishes the mutation invariance of generalized cluster scattering diagrams. Subsequently, Section~\ref{subsec:genMutInvariance} defines mutation at the diagram level and verifies that our generalized cluster scattering diagrams are mutation invariant. Section~\ref{subsec:genChamberStructure} extends the description of the chamber structure to generalized cluster scattering diagrams, thereby showing that the Fock-Goncharov cluster chambers are the maximal cones of a simplicial fan.
In Section~\ref{subsec:genPrincipalCoef}, we extend the definitions from Section~\ref{subsec:genBasicDefinitions} for the principal coefficient case. Finally, in Section~\ref{subsec:genAScat}, we describe how to construct a space $\mathcal{A}_{\textrm{scat}}$ from a generalized cluster scattering diagram and then how to identify $\mathcal{A}_{\textrm{scat}}$ with the $\mathcal{A}$-variety.

\subsection{Fixed data, seed data, and generalized cluster varieties}
\label{subsec:genBasicDefinitions}

We begin by updating some definitions for the generalized setting. First, we update the definition of fixed data to include data from the exchange degree matrix $[r_{ij}]$.

\begin{definition} \label{def:gen fixed}
    The following data is referred to as \emph{generalized fixed data}, denoted $\Gamma$:
    \begin{itemize}
        \item A lattice $N$ called the \emph{cocharacter lattice} with skew-symmetric bilinear form ${\{ \cdot, \cdot \}: N \times N \rightarrow \mathbb{Q}}$.
        \item A saturated sublattice $N_{\textrm{uf}} \subseteq N$ called the \emph{unfrozen sublattice}.
        \item An index set $I$ with $| I | = \textrm{rank}(N)$ and subset $I_{\textrm{uf}} \subseteq I$ such that ${| I_{\textrm{uf}}| = \textrm{rank}(N_{\textrm{uf}})}$.
        \item A set of positive integers $\{ d_i \}_{i \in I}$ such that $\textrm{gcd}(d_i) = 1$.
        \item A sublattice $N^{\circ} \subseteq N$ of finite index such that ${\{ N_{\textrm{uf}}, N^{\circ} \} \subseteq \mathbb{Z}}$ and ${\{ N, N_{\textrm{uf}} \cap N^{\circ} \} \subseteq \mathbb{Z}}$.
        \item A lattice $M = \textrm{Hom}(N,\mathbb{Z})$ called the \emph{character lattice} and sublattice ${M^{\circ} = \textrm{Hom}(N^{\circ},\mathbb{Z})}$.
        \item A set of positive integers $\{r_i \}_{i \in I}$.
        \item A collection $\{a_{i,j} \}_{i \in I_{\textrm{uf}},j \in [r_i-1]}$ of formal variables.
    \end{itemize}
    The adjective `fixed' refers to the fact that this data is fixed under mutation.
\end{definition}

Note that the exchange polynomial coefficients $\{ a_{i,j} \}_{i \in I_{\textrm{uf}}, j \in [r_i-1]}$ are formal variables, rather than elements of $\Bbbk$. As such, we must work over the ground ring $R = \Bbbk [a_{i,j}]$, where $\Bbbk$ is a field of characteristic zero, rather than over $\Bbbk$ as in the ordinary case. In doing so, we follow the work of \cite{B-FM-M-NC} on cluster varieties with coefficients.  In particular, this will be necessary in the proof of Proposition \ref{prop:universalLaurentPolynomial}, but is not necessary to construct cluster scattering diagrams and theta functions or to express generalized cluster monomials as theta functions.

We also establish the notion of a \emph{generalized torus seed}, also denoted $\bfs$.

\begin{definition}
\label{def:genTorusSeed}
Given a set of generalized fixed data, we can define associated \emph{generalized torus seed data} $\bfs = \left\{ (e_i,(a_{i,j})) \right\}_{i \in I_{\textrm{uf}}, j \in [r_i-1]}$ such that the collection $\{ e_i\}_{i \in I_{\textrm{uf}}}$ satisfies the conditions for ordinary torus seed data and each $(a_{i,j})$ is a tuple of formal variables taken from the collection specified in the fixed data.
\end{definition}

Analogous to the ordinary case, this defines a dual basis $\{ e_i^* \}_{i \in I}$ for $M$ and $\{f_i = d_i^{-1}e_i^* \}_{i \in I}$ for $M^{\circ}$. Note that when $r_i = 1$ for all $i$, our definitions reduce to the definitions for an ordinary torus seed.

We will confine our attention to the subclass of \emph{reciprocal} generalized cluster algebras:

\begin{definition}
A generalized torus seed $\bfs$ is called \emph{reciprocal} if its scalar tuples $(a_{i,j})$ satisfy the reciprocity condition $a_{i,j} = a_{i,r_i-j}$. We refer to the associated algebra as a \emph{reciprocal} generalized cluster algebra.
\end{definition}

Note that in the definition \cite{CS} of a generalized cluster algebra, the mutations of the exchange polynomial coefficients $a_{k,s}$ are of the form
\[a_{k,s}' = a_{k, r_k-s}, \]
and therefore mutation does not introduce new exchange polynomial coefficients. 
This allows us to include the collection of all exchange polynomial coefficients $\{ a_{i,j} \}$ in the fixed data and work over the ground ring $R = \Bbbk [a_{i,j}]$ when working with generalized cluster algebras.
Because we work only with reciprocal generalized cluster algebras, whose exchange polynomials are fixed under mutation, we do not necessarily need to include the exchange polynomial coefficients $a_{i,j}$ in the seed data. However, we include the (ordered) exchange polynomial coefficients as part of that data to allow for future extension to arbitrary generalized cluster algebras.

\begin{example}
\label{ex:genFixedData}
The generalized cluster algebra
\[ \mathcal{A}\left(\mathbf{x},\mathbf{y},\begin{bmatrix} 0 & 1 \\ -1 & 0 \end{bmatrix},\begin{bmatrix} 3 & 0 \\ 0 & 1 \end{bmatrix}, ((1,a,a,1),(1,1)) \right) \]
has generalized fixed data $\Gamma$ with a rank 2 lattice $N=N^{\circ} $, the dual lattice $M = M^{\circ}$, $d=(1,1)$, $r=(3,1)$, ${I = I_{\textrm{uf}} = \{ 1, 2 \}}$, and skew-symmetric bilinear form $\{ \cdot, \cdot \} : N \times N\rightarrow \mathbb{Z}$ specified by the exchange matrix. One possible choice of generalized torus seed data is 
\[ {\bfs = \{ (e_1 = (1,0),(1,a,a,1)),(e_2 = (0,1),(1,1)) \}}. \]
\end{example}

\begin{definition}
\label{def:basisMutGen}
    Given generalized torus seed data $\bfs$ and some $k \in I_{\textrm{uf}}$, a \emph{mutation in direction $k$} of the generalized torus seed data is defined by the following transformations of basis vectors and exchange polynomial coefficients:
    \begin{align*}
        e_i' &:= \begin{cases}
                    e_i + r_k[\epsilon_{ik}]_+e_k & i \neq k \\
                    -e_k & i = k
                \end{cases} \\
        f_i' &:= \begin{cases}
                    -f_k + r_k\sum_{j \in I_{\textrm{uf}}} [-\epsilon_{kj}]_{+}f_j & i = k \\
                    f_i & i \neq k
                \end{cases} \\
        a_{k,s}' &:= a_{k,r_k-s}
    \end{align*}
    The basis mutation induces the following mutation of the matrix $[\epsilon_{ij}]$:
    \begin{align*}
        \epsilon_{ij}' &:= \{e_i',e_j' \}d_j
                        = \begin{cases}
                            -\epsilon_{ij} & k = i \textrm{ or } k = j \\
                            \epsilon_{ij} & k \neq i,j \textrm{ and } \epsilon_{ik}\epsilon_{kj} \leq 0 \\
                            \epsilon_{ij} + r_k|\epsilon_{ik}|\epsilon_{kj} & k \neq i,j \textrm{ and } \epsilon_{ik}\epsilon_{kj} \geq 0
                        \end{cases}
    \end{align*}
\end{definition}

Note that when all $r_k = 1$, the formulas in Definition \ref{def:basisMutGen} coincide with those in Definition \ref{def:basisMut}.  Given generalized torus seed data $\bfs$, we can then define associated algebraic tori $\mathcal{A}_{\bfs}$ and $\mathcal{X}_{\bfs}$. 
As we are working over the the ground ring $R = \Bbbk[a_{i,j}]$, we will consider the tori over the ring $R = \Bbbk[a_{i,j}]$.
More precisely, for a lattice $L$, let
\[ T_L(R) := \textrm{Spec}\left(R[L^{*}] \otimes_{\Bbbk} R \right) = T_L \times_{\Bbbk} \textrm{Spec}(R). \]

This notation then allows us to state the following definition:

\begin{definition}
\label{def:seedTori}
A choice of generalized torus seed data $\bfs$ defines the tori:
\begin{align*}
    \mathcal{X}_{\bfs} &= T_{M}(R) = T_M \times_{\Bbbk} \textrm{Spec}(R) = \textrm{Spec}\left(\Bbbk[N] \right) \times_{\Bbbk} \textrm{Spec}(R) \\
    \mathcal{A}_{\bfs} &= T_{N^{\circ}}(R) = T_{N^{\circ}} \times_{\Bbbk} \textrm{Spec}(R) = \textrm{Spec}\left(\Bbbk[M^{\circ}] \right) \times_{\Bbbk} \textrm{Spec}(R)
\end{align*}
\end{definition}

Just as in the ordinary case, there are several common notational conventions for the coordinates of these algebraic tori. We will use $y_1, \dots, y_n$ for the coordinates of $\mathcal{X}_{\bfs}$ and $x_1, \dots, x_n$ for the coordinates of $\mathcal{A}_{\bfs}$ in order to be consistent with the prevailing notation for ordinary and generalized cluster algebras.

\begin{definition}
\label{def:genBiratMutMaps}
We define birational maps ${\mu_k: \mathcal{X}_\bfs \rightarrow \mathcal{X}_{\mu_k(\bfs)}}$ and ${\mu_k: \mathcal{A}_\bfs \rightarrow \mathcal{A}_{\mu_k(\bfs)}}$ via the pull-back of functions
\begin{align}
    \label{eq:A_pullback_gen}
    \mu_k^*z^m &= z^m \left( 1 + a_{k,1}z^{v_k} + \cdots + a_{k,r_k-1}z^{(r_k-1)v_k} + z^{r_kv_k} \right)^{-\langle d_ke_k,m \rangle} \\
    \label{eq:X_pullback_gen}
    \mu_k^*z^n &= z^n\left(1 + a_{k,1}z^{e_k} + \cdots + a_{k,r_k-1}z^{(r_k-1)e_k} + z^{r_ke_k} \right)^{-[n,e_k]}
\end{align}
for $n \in N$ and $m \in M^\circ$.
\end{definition}

\begin{remark}
The exchange relations given in Definition~\ref{def:genMutation} can be recovered from equations~\eqref{eq:A_pullback_gen} and \eqref{eq:X_pullback_gen} by setting $m=f_i$ and $n = e_i$. Consider the mutation of $x_i = z^{f_i}$ and $y_i = z^{e_i}$ in direction $k$. If $i = k$, then
\begin{align*}
    \mu_k^*(y_k') = \mu_k^*\left(z^{e_k'} \right) = \mu_k^*\left( z^{-e_k} \right) = z^{-e_k} \left(1 + \cdots + z^{r_ke_k} \right)^{-[-e_k,e_k]} = z^{-e_k} = y_k^{-1}
\end{align*}
and
\begin{align*}
    \mu_k^*(x_k') &= \mu_k^*\left( z^{f_k'} \right) \\
    &= \mu_k^*\left( z^{-f_k + r_k\sum_{j \in I_{\textrm{uf}}} [-\epsilon_{kj}]_{+}f_j} \right) \\
    &= z^{-f_k + r_k\sum_{j \in I_{\textrm{uf}}} [-\epsilon_{kj}]_{+}f_j}\left(1 + a_{k,1}z^{v_k} + \cdots + a_{k,r_k-1}z^{(r_k-1)v_k} + z^{r_kv_k} \right)^{-\langle d_ke_k, -f_k + r_k\sum_{j \in I_{\textrm{uf}}} [-\epsilon_{kj}]_{+}f_j \rangle} \\
    &= z^{-f_k}\left( \prod_{j \in I_{\textrm{uf}}} z^{[-\epsilon_{kj}]_{+}f_j} \right)^{r_k}\left(1 + a_{k,1}z^{v_k} + \cdots + a_{k,r_k-1}z^{(r_k-1)v_k} + z^{r_kv_k} \right)^{\langle d_ke_k,f_k \rangle} \\
    & \qquad \cdot \left( \prod_{j \in I_{\textrm{uf}}}  \left(1 + a_{k,1}z^{v_k} + \cdots + a_{k,r_k-1}z^{(r_k-1)v_k} + z^{r_kv_k} \right)^{-\langle d_ke_k,[\epsilon_{kj}]_{+}f_j \rangle}\right) \\
    &= z^{-f_k}\left( \prod_{j \in I_{\textrm{uf}}} z^{[-\epsilon_{kj}]_{+}f_j} \right)^{r_k}\left(1 + a_{k,1}z^{v_k} + \cdots + a_{k,r_k-1}z^{(r_k-1)v_k} + z^{r_kv_k} \right) \\
    &= z^{-f_k}\left( \prod_{j \in I_{\textrm{uf}}} z^{[-\epsilon_{kj}]_{+}f_j} \right)^{r_k} \left(1 + a_{k,1}\left(\prod_{j \in I_{\textrm{uf}}} z^{\epsilon_{kj}f_j} \right) + \cdots + \left(\prod_{j \in I_{\textrm{uf}}} z^{\epsilon_{kj}f_j} \right)^{r_k} \right) \\
    &= x_k^{-1} \left( \prod_{j \in I_{\textrm{uf}}} x_j^{[-b_{kj}]_{+}} \right)^{r_k} \left(1 + a_{k,1} \left( \prod_{j \in I_{\textrm{uf}}} x_j^{b_{kj}} \right) + \cdots + \left( \prod_{j \in I_{\textrm{uf}}} x_j^{b_{kj}} \right)^{r_k} \right)
\end{align*}
If $i \neq k$, then
\begin{align*}
    \mu_k^*(y_i') &= \mu_k^*\left(z^{e_i'} \right) \\
    &= \mu_k^*\left( z^{e_ i+ r_k[\epsilon_{ik}]_{+}e_k} \right) \\
    &= z^{e_ i+ r_k[\epsilon_{ik}]_{+}e_k}\left( 1 + a_{k,1}z^{e_k} + \cdots + a_{k,r_k-1}z^{(r_k-1)e_k} + z^{r_ke_k}\right)^{-[e_ i+ r_k[\epsilon_{ik}]_{+}e_k, e_k]} \\
    &= z^{e_i}z^{r_k[\epsilon_{ik}]_{+}e_k}\left( 1 + a_{k,1}z^{e_k} + \cdots + a_{k,r_k-1}z^{(r_k-1)e_k} + z^{r_ke_k}\right)^{-[e_i,e_k]} \\
    &= z^{e_i}\left(z^{[\epsilon_{ik}]_{+}e_k} \right)^{r_k}\left( 1 + a_{k,1}z^{e_k} + \cdots + a_{k,r_k-1}z^{(r_k-1)e_k} + z^{r_ke_k}\right)^{-\epsilon_{ik}} \\
    &= y_i \left( y_k^{[b_{ik}]_{+}} \right)^{r_k}\left(1 + a_{k,1}y_k + \cdots + a_{k,r_k-1}y_k^{r_k-1} + y_k^{r_k} \right)^{-b_{ik}}
\end{align*}
and
\begin{align*}
    \mu_k^*(x_i') = \mu_k^*\left(z^{f_i'} \right) = \mu\left(z^{f_i} \right) = z^{f_i}\left( 1 + a_{k,1}z^{v_k} + \cdots + z^{r_kv_k} \right)^{-\langle d_ke_k,f_k \rangle} = z^{f_i} = x_i
\end{align*}
\end{remark}

\begin{remark} \label{rk:toclarify}
Let us examine why the $r_i-1$ scalars are included as extra data in fixed data. 
We note that in Definition \ref{def:genMutation}, which is the definition of generalized cluster algebras from \cite[Definition 2.2]{Nakanishi}, Nakanishi \cite[Equation (2.1)]{Nakanishi} emphasised that his $d_i$ (which is our $r_i$) are not the scalars used to make the exchange matrix skew-symmetric.
While equations \eqref{eq:A_pullback_gen} and \eqref{eq:X_pullback_gen} appear to be two different type of mutations, we should always bear in mind that the $y$ variables come from the cluster algebras with principal coefficients, i.e. one can obtain the $\cX$ mutation in equation \eqref{eq:X_pullback_gen} from equation \eqref{eq:A_pullback_gen} as the $\cA$ case with principal coefficients. 
Although the notation is developed in later sections, one can jump ahead to Remark \ref{rk:however} to see how we can obtain the $\cX$ mutations from the $\Aprin$ mutations. 
One can further consult  
Example \ref{eg:whywecare} to see the effect that varying the $r_i$ values can have even when the resulting exchange matrices associated to the companion algebras (Section \ref{sec:companionDiagrams}) are `the same'.
\end{remark}

We are now working over polynomial ring $R = \Bbbk[a_{i,j}]$ instead of a characteristic zero field $\Bbbk$. Proposition 2.4 in \cite{GHK} can be generalized to this setting by working through 
\cite{Gro60}. We will state the version in \cite{B-FM-M-NC} for this setting. 

\begin{prop}\cite[Lemma 3.10]{B-FM-M-NC}
\label{prop:torusGluing}
Let $\{S_i \}$ be a collection of integral, separate schemes of finite type over a locally Noetherian ring $R$, with birational maps $f_{ij} : S_i \rightarrow S_j$ for all $i,j$, with $f_{ii} = \textrm{Id}$ and $f_{jk} \circ f_{ij} = f_{ik}$ as rational maps. Let $U_{ij} \subset S_i$ be the largest open subscheme such that $f_{ij}: U_{ij} \rightarrow f_{ij}(U_{ij})$ is an isomorphism. Then there exists a scheme
\[ S:= \bigcup_{i} S_i \]
which is obtained by gluing the $S_i$ along the open sets $U_{ij}$ via the maps $f_{ij}$.
\end{prop}

As in the ordinary setting, let $\mathfrak{T}_v$ be a directed infinite rooted tree where each vertex has $|I_{\textrm{uf}}|$ outgoing edges, labeled by the elements of $I_{\textrm{uf}}$.  An edge with label $k \in I_{\textrm{uf}}$ now corresponds to generalized mutation in direction $k$.  Hence, any simple path beginning at vertex $v$ determines a sequence of mutations and an associated generalized torus seed $\bfs_{w}$ for each vertex $w$ of $\mathfrak{T}_v$. We attach copies of $\mathcal{X}_{\bfs_w}$ and $\mathcal{A}_{\bfs_w}$ to each vertex $w$.

\begin{definition}
\label{def:genClusterVarieties}
The generalized \emph{$\mathcal{A}$ cluster variety} is the scheme
\[ \mathcal{A} := \bigcup_{\bfs \in \mathfrak{T}} \mathcal{A}_{\bfs} \]
obtained by using Proposition~\ref{prop:torusGluing} to glue the collection of algebraic tori $\{ \mathcal{A}_{\bfs} \}_{\bfs \in \mathfrak{T}}$ according to the birational maps $\mu_k: \mathcal{A}_{\bfs} \rightarrow \mathcal{A}_{\mu_k(\bfs)}$ specified in Definition~\ref{def:genBiratMutMaps}. Analogously, the generalized \emph{$\mathcal{X}$ cluster variety} is defined to be the scheme
\[ \mathcal{X} := \bigcup_{\bfs \in \mathfrak{T}} \mathcal{X}_{\bfs} \]
obtained by gluing the collection $\{ \mathcal{X}_{\bfs} \}_{\bfs \in \mathfrak{T}}$ according to the birational maps ${\mu_k: \mathcal{X}_{\bfs} \rightarrow \mathcal{X}_{\mu_k(\bfs)}}$. 
\end{definition}
For readability, we will often refer to the these schemes simply as the $\mathcal{A}$-variety and $\mathcal{X}$-variety when it is clear from context that they are generalized cluster varieties rather than ordinary cluster varieties.

The \emph{upper generalized cluster algebra} $up(V)$ associated to a generalized cluster variety $V$ is defined as ${\textrm{up}(V) := \Gamma(V,\mathcal{O}_{V})}.$ Equivalently, it is defined in \cite{GSV-I} as
\[ \textrm{up}\left( \mathcal{A} \right) := \bigcap_{\textrm{clusters} \{ x_1, \dots, x_n \} \textrm{ of } \mathcal{A}} R[x_1^{\pm 1}, \dots, x_n^{\pm 1}] \subset \mathcal{F}\]
The \emph{generalized cluster algebra} $\textrm{gen}(V)$ associated to a generalized cluster variety $V$ is the subalgebra of $\textrm{up}(V)$ generated by the set of global monomials on $V$.

Intuitively, we expect that the construction of the  generalized $\mathcal{A}$-variety and $\mathcal{X}$-variety should not depend on the choice of seed - that is, mutation equivalent seeds $\bfs$ and $\bfs'$ should yield isomorphic schemes. The two smaller commutative diagrams in the following proposition show that structures of the tori given in Definition~\ref{def:seedTori} are compatible with generalized torus seed data mutation. This induces a similar compatibility for $\mathcal{A}$ and $\mathcal{X}$.

\begin{prop}
Let $K = \textrm{ker}(p_2^*)$ and $K^{\circ} = K \cap N^{\circ}$. For a given generalized torus seed $\bfs$ and mutation direction $k \in [n]$, the following diagrams are commutative:
\begin{center}
    \begin{tikzcd}
    T_{K^\circ} \arrow[r] \arrow[d, "="]& \mathcal{A}_\bfs \arrow[r,"p"]  \arrow[d,"\mu_k"] & \mathcal{X}_{\bfs} \arrow[r] \arrow[d,"\mu_k"] & T_{K^*} \arrow[d,"="] \\
    T_{K^\circ} \arrow[r] & \mathcal{A}_{\mu_k}(\bfs) \arrow[r, "p"] & \mathcal{X}_{\mu_k(\bfs)} \arrow[r] & T_{K^*}
    \end{tikzcd}
    
    \vspace{5mm}
    \begin{tikzcd}
    T_{(N/N_{\textrm{uf}})^*} \arrow[r] \arrow[d,"="] & \mathcal{X}_{\bfs} \arrow[d,"\mu_k"] & \mathcal{A}_{\bfs} \arrow[d,"\mu_k"] \arrow[r] & T_{N^\circ / (N_{\textrm{uf}}\cap N^\circ)} \arrow[d,"="] \\
     T_{(N/N_{\textrm{uf}})^*} \arrow[r]& \mathcal{X}_{\mu_k(\bfs)} & \mathcal{A}_{\mu_k(\bfs)} \arrow[r] & T_{N^\circ / (N_{\textrm{uf}}\cap N^\circ)}
    \end{tikzcd}
\end{center}
\end{prop}

\begin{proof}
There are several unlabeled maps in the above commutative diagrams. Those maps come from the following structures, as described in \cite{GHK}:
\begin{enumerate}
    \item The inclusion $K \subseteq N$ induces a map $\mathcal{X}_{\bfs} \rightarrow T_{K^*}$.
    \item The inclusion $K^{\circ} \rightarrow N^{\circ}$ induces a map $T_{K^{\circ}} \rightarrow \mathcal{A}_{\bfs}$.
    \item Let $N_{\textrm{uf}}^{\perp} := \{m \in M^{\circ} : \langle m, n \rangle = 0 \textrm{ for all } n \in N_{\textrm{uf}} \}$. Then the inclusion $N_{\textrm{uf}}^{\perp} \subseteq M^{\circ}$ induces a map $\mathcal{A}_{\bfs} \rightarrow T_{N^{\circ} / (N_{\textrm{uf}} \cap N^{\circ})}$.
    \item The choice of the map $p^*: N \rightarrow M^{\circ}$ defines a map $p: \mathcal{A}_{\bfs} \rightarrow \mathcal{X}_{\bfs}$. The map $p^*:N\rightarrow M^{\circ}$ induces maps $p^*: K \rightarrow N_{\textrm{uf}}^{\perp}$ and $p^*:N/N_{\textrm{uf}} \rightarrow (K^{\circ})^{*}$ which define maps $p:T_{N/(N_{\textrm{uf}}\cap N^{\circ})} \rightarrow T_{K^{*}}$ and $p:T_{K^{\circ}}\rightarrow T_{(N/N_{\textrm{uf}})^{*}}$.
\end{enumerate}
Using the definitions of these maps, $p$, and $\mu_k$, it is straightforward to check the commutativity of each square.
\end{proof}

\subsection{Generalized cluster scattering diagrams}
\label{subsec:genScatDiagConstruction}

As in the ordinary case, we will be interested in a particular scattering diagram which is defined by the generalized fixed and torus seed data. To define this diagram, which we refer to as the \emph{generalized cluster scattering diagram}, we begin by modifying the definition of a wall to reflect the fact that we are now working over the ground ring $R = \Bbbk[a_{i,j}]$ rather than over $\Bbbk$.

As before, we assume that the map $p_1^*: N_{\textrm{uf}} \rightarrow M^{\circ}$ is injective. This \emph{Injectivity Assumption} ensures that we are able to choose a convex top-dimensional cone $\sigma \subset M_{\mathbb{R}}$ with associated monoid $P := \sigma \cap M^{\circ}$ such that $p_1^*(e_i) \in J := P\backslash P^{\times}$ for all $i \in I_{\textrm{uf}}$. If $p_1^*$ is not injective, then $P$ may fail to be full-dimensional. Although injectivity is not guaranteed for an arbitrary set of fixed data, it is guaranteed in the principal coefficient case, which we will discuss in Section~\ref{subsec:genPrincipalCoef}. For this reason, the results in Section~\ref{sec:genThetaBasis} are proved via the principal coefficient case and then extended to arbitrary generalized cluster varieties.

As in Section~\ref{subsec:ordDiagrams}, set
\[N^+ := N_{\bfs}^+ := \left\{ \left. \sum_{i \in I_{\textrm{uf}}} a_ie_i \right| a_i \geq 0, \sum a_i > 0 \right\} \]
and choose a linear function $d: N \rightarrow \mathbb{Z}$ such that $d(n) > 0$ for $n \in N^+$.

\begin{definition}
\label{def:walls}
    A \emph{wall} in $M_\mathbb{R}$ is a pair $(\mathfrak{d},f_\mathfrak{d}) \in (N^+,\widehat{R[P]})$ such that for some primitive $n_0 \in N^{+}$,
    \begin{enumerate}
        \item $f_\mathfrak{d} \in \widehat{R[P]}$ has the form $1 + \sum_{j=1}^{\infty} c_j z^{jp_1^*(n_0)}$ with $c_j \in R$.
        \item $\mathfrak{d} \subset n_0^{\perp} \subset M_\mathbb{R}$ is a $({\mathrm{rank}~M -1})$-dimensional convex rational polyhedral cone.
    \end{enumerate}
    We refer to $\mathfrak{d} \subset M_\mathbb{R}$ as the \emph{support} of the wall $(\mathfrak{d},f_\mathfrak{d})$.
\end{definition}

Let $\mathfrak{m}$ now denote the ideal in $\widehat{R[P]}$ which consists of formal power series with constant term zero. 
The definitions of a scattering diagram, wall-crossing automorphism, and path-ordered product must then either be updated to reflect the change in ground ring or read with the understanding that $\mathfrak{m}$ now denotes an ideal in $\widehat{R[P]}$ rather than in $\widehat{\Bbbk[P]}$. We give the statements of these definitions in the generalized setting below, for the sake of completeness.

\begin{definition}
\label{def:genScatteringDiagram}
A \emph{scattering diagram} $\mathfrak{D}$ for $N^+$ and $\bfs$ is a set of walls $\{ (\mathfrak{d},f_{\mathfrak{d}}) \}$ such that for every degree $k > 0$, there are a finite number of walls $(\mathfrak{d},f_{\mathfrak{d}}) \in \mathfrak{D}$ with $f_{\mathfrak{d}} \neq 1 \mod \mathfrak{m}^{k+1}$.
\end{definition}

\begin{definition}
\label{def:genWallcrossing}
For $n_0 \in N^{+}$, let $m_0 := p_1^*(n_0)$ and $f = 1 + \sum_{k = 1}^{\infty} c_kz^{km_0} \in \widehat{R[P]}$. Then $\mathfrak{p}_{f} \in \widehat{R[P]}$ denotes the automorphism
\[ \mathfrak{p}_f(z^m) = z^mf^{\langle n_0',m \rangle} \]
where $n_0'$ generates the monoid $\mathbb{R}_{\geq 0}n_0 \cap N^{\circ}$.
\end{definition}

\begin{definition} \label{def:genPathProduct}
Let $\gamma: [0,1] \rightarrow M_\mathbb{R} \backslash \textrm{Sing}(\mathfrak{D})$ be a smooth immersion which crosses walls transversely and whose endpoints are not in the support of $\mathfrak{D}$. Let $0 < t_1 \leq t_2 \leq \cdots \leq t_s < 1$ be a sequence such that at time $t_i$ the path $\gamma$ crosses the wall $\mathfrak{d}_i$ such that $f_i \neq 1~\mod \mathfrak{m}^{k+1}$. Definition~\ref{def:1.6-GHKK} ensures that this is a finite sequence. For each $i \in \{ 1, \dots, s \}$, set $\epsilon_i := -\textrm{sgn}(\langle n_i, \gamma'(t_i) \rangle)$ where $n_i \in N^{+}$ is the primitive vector normal to $\mathfrak{d}_i$. For each degree $k > 0$, define
\[ \mathfrak{p}_{\gamma,\mathfrak{D}}^k := \mathfrak{p}_{f_{\mathfrak{d}_{t_s}}}^{\epsilon_s} \circ \cdots \circ \mathfrak{p}_{f_{\mathfrak{d}_{t_1}}}^{\epsilon_1},  \]
where $\mathfrak{p}_{f_{\mathfrak{d}_{t_i}}}$ is defined as in Definition \ref{def:genWallcrossing}. Then,
\[ \mathfrak{p}_{\gamma,\mathfrak{D}} := \lim_{k \rightarrow \infty} \mathfrak{p}_{\gamma,\mathfrak{D}}^k. \]
We refer to $ \mathfrak{p}_{\gamma,\mathfrak{D}}$ as a \emph{path-ordered product}.
\end{definition}

Generalized cluster scattering diagrams have the same notions of equivalence and uniqueness as ordinary scattering diagrams.
To define them, we first define the initial scattering diagrams for generalized cluster algebras.

\begin{definition}
\label{def:initialScatDiagram}
Let $v_i = p_1^*(e_i)$ for $i \in I_{\textrm{uf}}$. Then we define
\[ \mathfrak{D}_{\textrm{in},\textbf{s}} := \{ (e_i^{\perp},1+a_{i,1}z^{v_i} + a_{i,2}z^{2v_i} + \cdots + a_{i,r_i-1}z^{(r_i - 1)v_i} + z^{r_iv_i}) \}_{i \in I_{\textrm{uf}}} \]
\end{definition}

Similar to \cite[Theorem 1.12]{GHKK}, we obtain: 

\begin{theorem} 
\label{thm:gen_consistent} Given a generalized torus seed $\bfs$, there exists a consistent scattering diagram $\mathfrak{D}_{\bfs}$ such that $ \mathfrak{D}_{\textrm{in},\bfs} \subset \mathfrak{D}_{\bfs}$ and   $\mathfrak{D}_{\bfs} \backslash \mathfrak{D}_{\textrm{in},\bfs}$ consists only of walls $\mathfrak{d} \subset n_0^{\perp}$ with $p_1^*(n_0) \not\in \mathfrak{d}$. The scattering diagram $\mathfrak{D}_\bfs$ is unique up to equivalence.
\end{theorem}

\begin{proof}
The proof given in Section 1.2 and Appendix C of \cite{GHKK} for ordinary cluster scattering diagrams holds in our generalized setting. That proof is a special case of results from \cite{GS} and \cite{KS-chapter} and holds in our setting because it does not require that the wall-crossing automorphisms are strictly binomial.
\end{proof}

The generalized cluster scattering diagram $\mathfrak{D}_{\bfs}$ is the unique (up to equivalence) consistent scattering diagram obtained by adding walls to $\mathfrak{D}_{\textrm{in},\bfs}$.

\begin{example}
The generalized cluster algebra from Example \ref{ex:genFixedData} has birational maps $\mu_1,\mu_2$ defined by the pullbacks 
\begin{align*}
    \mu_{\cX,1}^*z^n &= z^n(1+az^{(1,0)}+az^{(2,0)}+z^{(3,0)})^{-[n,(1,0))]}, \\
    \mu_{\cA,1}^*z^m &= z^m(1+az^{(0,1)}+az^{(0,2)}+z^{(0,3)})^{-\langle (1,0)),m \rangle}, \\
    \mu_{\cX,2}^*z^n &= z^n(1+z^{(0,1)})^{-[n,(0,1)]}, \\ \mu_{\cA,2}^*z^m &= z^m(1+z^{(-1,0)})^{-\langle (0,1),m \rangle}.
\end{align*}
Since the $p^*$ map is injective in this case,  the initial $\cA$ scattering diagram is of the form
\[ \mathfrak{D}_{\textrm{in},\bfs} = \{ ((0,1)^{\perp},1+z^{(-1,0)}),((1,0)^{\perp},1+az^{(0,1)}+az^{(0,2)}+z^{(0,3)}) \}, \]
which can be drawn as
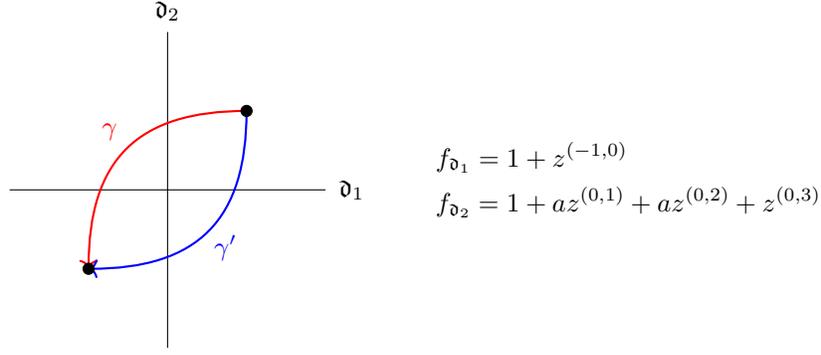
\begin{center}
    \begin{minipage}{0.3\textwidth}
    \begin{tikzpicture}[scale=1.4]
        \draw (-1.5,0) to (1.5,0);
        \draw (0,-1.5) to (0,1.5);
        
        \node at (1.75,0) {$\mathfrak{d}_1$};
        \node at (0,1.7) {$\mathfrak{d}_2$};
        
        \draw[red,->,out=180,in=90, thick,looseness=1.25] (0.75,0.75) to (-0.75,-0.75);
        \draw[blue,->,out=-90,in=0,thick, looseness=1.25] (0.75,0.75) to (-0.75,-0.75);
        
        \node[red] at (-0.55,0.55) {$\gamma$};
        \node[blue] at (0.55,-0.55) {$\gamma'$};
        
        \filldraw (0.75,0.75) circle (1.5pt);
        \filldraw (-0.75,-0.75) circle (1.5pt);
    \end{tikzpicture}
    \end{minipage}\qquad
    \begin{minipage}{0.3\textwidth}
        \vspace{20mm}
        \begin{align*}
            f_{\mathfrak{d}_1} &= 1 + z^{(-1,0)} \\
            f_{\mathfrak{d}_2} &= 1 + az^{(0,1)} + az^{(0,2)} + z^{(0,3)}
        \end{align*}
        \vspace{15mm}
    \end{minipage}
\end{center}
Consider the paths $\gamma$ (drawn in red, on the left) and $\gamma'$ (drawn in blue, on the right). We can demonstrate that the diagram $\mathfrak{D}_{\textrm{in},\bfs}$ is not consistent by computing $\mathfrak{p}_{\gamma,\mathfrak{D}_{\textrm{in},\bfs}}$ and $\mathfrak{p}_{\gamma',\mathfrak{D}_{\textrm{in},\bfs}}$. We compute $\mathfrak{p}_{\gamma,\mathfrak{D}_{\textrm{in},\bfs}}$ as
\begin{align*}
    z^{(1,1)} &\xmapsto{\mathfrak{d}_2} z^{(1,1)}\left(1 + az^{(0,1)} + az^{(0,2)} + z^{(0,3)} \right)^{\langle (1,1),(1,0) \rangle} \\
    &= z^{(1,1)}\left(1 + az^{(0,1)} + az^{(0,2)} + z^{(0,3)} \right) \\
    &\xmapsto{\mathfrak{d}_1} z^{(1,1)} \left(1 + z^{(-1,0)}\right)^{\langle (1,1),(0,1) \rangle}
    \cdot \\
    &\qquad \left(
    \begin{array}{c}
    1 + az^{(0,1)}\left(1 + z^{(-1,0)}\right)^{\langle (0,1),(0,1) \rangle} + az^{(0,2)}\left(1 + z^{(-1,0)}\right)^{\langle (0,2),(0,1) \rangle} \\
    + z^{(0,3)}\left(1 + z^{(-1,0)}\right)^{\langle (0,3),(0,1) \rangle}
    \end{array} \right) \\
    &= z^{(1,1)}\left(1 + z^{(-1,0)}\right) \left( 1 + az^{(0,1)}\left(1 + z^{(-1,0)}\right) + az^{(0,2)}\left(1 + z^{(-1,0)}\right)^2 + z^{(0,3)}\left(1 + z^{(-1,0)}\right)^3\right)
\end{align*}
Similarly, we compute $\mathfrak{p}_{\gamma',\mathfrak{D}_{\textrm{in},\bfs}}$ as
\begin{align*}
    z^{(1,1)} &\xmapsto{\mathfrak{d}_1} z^{(1,1)}\left(1+z^{(-1,0)} \right)^{\langle (1,1),(0,1) \rangle} \\
    &= z^{(1,1)}\left(1+z^{(-1,0)}\right) \\
    &\xmapsto{\mathfrak{d}_2} z^{(1,1)}\left( 1 + az^{(0,1)} + az^{(0,2)} + z^{(0,3)}\right)^{\langle (1,1),(1,0) \rangle}\left(1 + z^{(-1,0)}\left( 1 + az^{(0,1)} + az^{(0,2)} + z^{(0,3)}\right)^{\langle (-1,0),(1,0) \rangle} \right) \\
    &= z^{(1,1)}\left( 1 + az^{(0,1)} + az^{(0,2)} + z^{(0,3)}\right) \left(1 + \frac{z^{(-1,0)}}{\left( 1 + az^{(0,1)} + az^{(0,2)} + z^{(0,3)}\right)} \right) \\
    &= z^{(1,1)}\left( 1 + az^{(0,1)} + az^{(0,2)} + z^{(0,3)} + z^{(-1,0)}\right)
\end{align*}
Observe that $\mathfrak{p}_{\gamma,\mathfrak{D}_{\textrm{in},\bfs}} \neq \mathfrak{p}_{\gamma',\mathfrak{D}_{\textrm{in},\bfs}}$. Hence, $\mathfrak{D}_{\textrm{in},\bfs}$ is by definition not consistent. Making the diagram consistent requires adding four walls:
\begin{align*}
    \mathfrak{d}_3 &= \left(\mathbb{R}_{\geq 0}(1,-3), 1 + z^{(-1,3)} \right), \\
    \mathfrak{d}_4 &= \left( \mathbb{R}_{\geq 0}(1,-2), 1 + az^{(-1,2)} + az^{(-2,4)} + z^{(-3,6)} \right), \\
    \mathfrak{d}_5 &= \left( \mathbb{R}_{\geq 0}(2,-3), 1 + z^{(-2,3)} \right), \\
    \mathfrak{d}_6 &= \left( \mathbb{R}_{\geq 0}(1,-1), 1 + az^{(-1,1)} + az^{(-2,2)} + z^{(-3,3)} \right).
\end{align*}
The consistent diagram is shown in Example~\ref{ex:gen_T_k}.
\end{example}

\begin{example} \label{eg:whywecare}
Consider the family of generalized cluster algebras
\[ \mathcal{A}_{\mathbf{b},\mathbf{r}} = \left( (x_1,x_2),(y_1,y_2), \begin{bmatrix} 0 & b_2 \\ -b_1 & 0 \end{bmatrix}, \begin{bmatrix} r_1 & 0 \\ 0 & r_2 \end{bmatrix}, (\textbf{a}_1,\textbf{a}_2) \right), \]
where $r_1b_1 = r_2b_2 = 2$. Each generalized cluster algebra has $I = I_{\textrm{uf}} = \{ 1, 2 \}$ and skew-symmetric bilinear form $\{ \cdot, \cdot \} : N \times N \rightarrow \mathbb{Q}$ given by $\{ e_i, e_j \} = \delta_{ij}$. The rest of the generalized fixed data associated to each such generalized cluster algebra is summarized in the following table:

\begin{center}
\begin{tabular}{|c|c|c|c|c|}
     \hline & $\mathcal{A}_{(2,2),(1,1)}$ & $\mathcal{A}_{(2,1),(1,2)}$ & $\mathcal{A}_{(1,2),(2,1)}$ & $\mathcal{A}_{(1,1),(2,2)}$  \\ \hline \hline
     $\mathbf{b}$ & $(2,2)$ & $(2,1)$ & $(1,2)$ & $(1,1)$ \\ \hline
     $\mathbf{r}$ & $(1,1)$ & $(1,2)$ & $(2,1)$ & $(2,2)$ \\ \hline
     $N$ & $\textrm{span}\{ e_1,e_2 \}$ & $\textrm{span}\{ e_1,e_2 \}$  & $\textrm{span}\{ e_1,e_2 \}$  & $\textrm{span}\{ e_1,e_2 \}$ \\ \hline
     $N^{\circ}$ & $\textrm{span}\{ 2e_1,2e_2 \}$ & $\textrm{span}\{ 2e_1,e_2 \}$ & $\textrm{span}\{ e_1,2e_2 \}$ & $\textrm{span}\{ e_1,e_2 \}$ \\ \hline
     $M$ & $\textrm{span}\{ e_1^*,e_2^* \}$ & $\textrm{span}\{ e_1^*,e_2^* \}$ & $\textrm{span}\{ e_1^*,e_2^* \}$ & $\textrm{span}\{ e_1^*,e_2^* \}$ \\ \hline
     $M^{\circ}$ & $\textrm{span}\left\{ \frac{1}{2}e_1^*,\frac{1}{2}e_2^* \right\}$ & $\textrm{span}\left\{ \frac{1}{2}e_1^*,e_2^* \right\}$ & $\textrm{span}\left\{ e_1^*,\frac{1}{2}e_2^* \right\}$ & $\textrm{span}\left\{ e_1^*,e_2^* \right\}$ \\ \hline
     $\{ a_{i,j} \}$ & $\emptyset$ & $\{ a_{2,1} = a \}$ & $\{ a_{1,1} =  a \}$ & $\{ a_{1,1} = a, a_{2,1} = b \}$ \\ \hline
\end{tabular}
\end{center}
Let $e_1 = (1,0)$ and $e_2 = (0,1)$. One natural choice of a set of generalized torus seeds for this family of generalized cluster algebras is:
\begin{align*}
    \mathbf{s}_{\mathcal{A}_{(2,2),(1,1)}} &:= \left( (e_1,(1,1)),(e_2,(1,1))  \right), \\
    \mathbf{s}_{\mathcal{A}_{(2,1),(1,2)}} &:= \left( (e_1,(1,1)),(e_2,(1,a,1))  \right), \\
    \mathbf{s}_{\mathcal{A}_{(1,2),(2,1)}} &:= \left( (e_1,(1,a,1)),(e_2,(1,1))  \right), \\
    \mathbf{s}_{\mathcal{A}_{(1,1),(2,2)}} &:= \left( (e_1,(1,a,1)),(e_2,(1,b,1))  \right).
\end{align*}
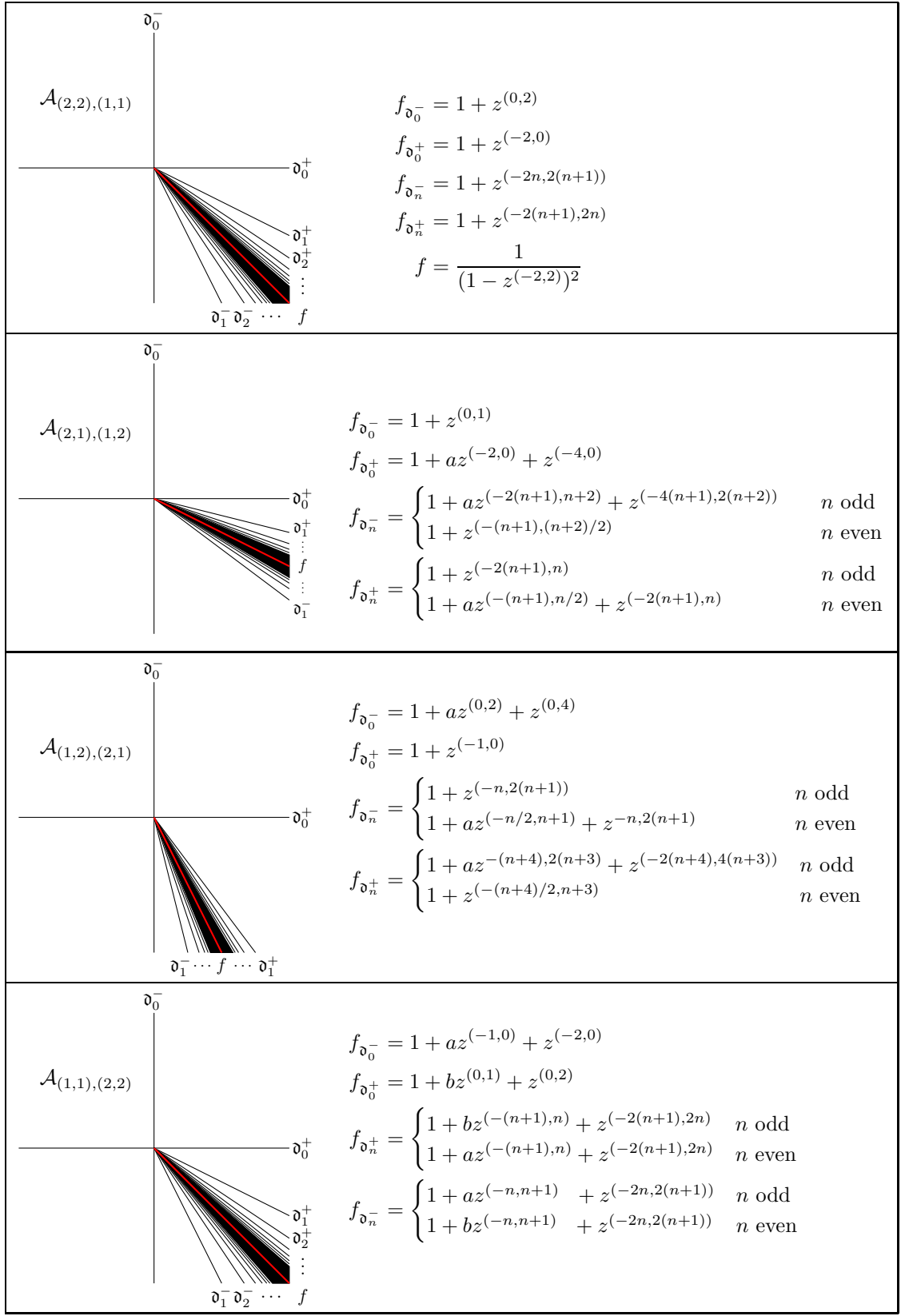
\begin{figure}
    \begin{tabularx}{0.9\textwidth}[t]{|XX|}
         \hline 
         \begin{minipage}{\textwidth}
        \begin{tikzpicture}[scale=0.75]
        \node at (-1.5,1.5) {$\mathcal{A}_{(2,2),(1,1)}$};
        
        \draw (0,-3) to (0,3);
        \draw (-3,0) to (3,0);
        
        \draw (0,0) to (3,-1.5);
        \draw (0,0) to (3,-2);
        \draw (0,0) to (3,-2.25);
        \draw (0,0) to (3,-2.4);
        \draw (0,0) to (3,-2.5);
        \draw (0,0) to (3,-2.57);
        \draw (0,0) to (3,-2.63);
        
        \draw (0,0) to (1.5,-3);
        \draw (0,0) to (2,-3);
        \draw (0,0) to (2.25,-3);
        \draw (0,0) to (2.4,-3);
        \draw (0,0) to (2.5,-3);
        \draw (0,0) to (2.57,-3);
        \draw (0,0) to (2.63,-3);
        
        \filldraw (0,0) to (3,-2.63) to (3,-3) to (2.63,-3) to (0,0);
        
        \draw[thick,red] (0,0) to (3,-3);
        
        \node[scale=0.85] at (0,3.3) {$\mathfrak{d}_{0}^{-}$};
        \node[scale=0.85] at (3.3,0) {$\mathfrak{d}_{0}^{+}$};
        
        \node[scale=0.85] at (3.3,-1.5) {$\mathfrak{d}_{1}^{+}$};
        \node[scale=0.85] at (3.3,-2) {$\mathfrak{d}_{2}^{+}$};
        \node[scale=0.85] at (3.3,-2.5) {$\vdots$};
        
        \node[scale=0.85] at (1.5,-3.3) {$\mathfrak{d}_{1}^{-}$};
        \node[scale=0.85] at (2,-3.3) {$\mathfrak{d}_{2}^{-}$};
        \node[scale=0.85] at (2.6,-3.3) {$\hdots$};
        
        \node[scale=0.85] at (3.3,-3.3) {$f$};
    \end{tikzpicture}\end{minipage} &
    \vspace{-6em}
    {\begin{align*}
        f_{\mathfrak{d}_0^{-}} &= 1 + z^{(0,2)} \\
        f_{\mathfrak{d}_0^{+}} &= 1 + z^{(-2,0)} \\
        f_{\mathfrak{d}_n^{-}} &= 1 + z^{(-2n,2(n+1))} \\
        f_{\mathfrak{d}_n^{+}} &= 1 + z^{(-2(n+1),2n)} \\
        f &= \frac{1}{(1-z^{(-2,2)})^2} 
    \end{align*}}  \\ \hline
        \begin{minipage}{\textwidth}
        \begin{tikzpicture}[scale=0.75]
        \node at (-1.5,1.5) {$\mathcal{A}_{(2,1),(1,2)}$};
        
        \draw (0,-3) to (0,3);
        \draw (-3,0) to (3,0);
        
        \draw (0,0) to (3,-0.75);
        \draw (0,0) to (3,-1);
        \draw (0,0) to (3,-1.125);
        \draw (0,0) to (3,-1.2);
        
        \draw (0,0) to (3,-2.25);
        \draw (0,0) to (3,-2);
        \draw (0,0) to (3,-1.875);
        \draw (0,0) to (3,-1.8);
        
        \filldraw (0,0) to (3,-1.25) to (3,-1.75) to (0,0);
        
        \draw[thick,red] (0,0) to (3,-1.5);
        
        \node[scale=0.85] at (0,3.3) {$\mathfrak{d}_{0}^{-}$};
        \node[scale=0.85] at (3.3,0) {$\mathfrak{d}_{0}^{+}$};
        
        \node[scale=0.75] at (3.3,-0.65) {$\mathfrak{d}_1^{+}$};
        \node[scale=0.55] at (3.3,-1) {$\vdots$};
        
        \node[scale=0.55] at (3.3,-1.9) {$\vdots$};
        \node[scale=0.75] at (3.3,-2.45) {$\mathfrak{d}_1^{-}$};
        
        \node[scale=0.75] at (3.3,-1.5) {$f$};
    \end{tikzpicture}\end{minipage} &
    \vspace{-6em}    
    {\begin{align*}
        f_{\mathfrak{d}_0^{-}} &= 1 + z^{(0,1)} \\
        f_{\mathfrak{d}_{0}^{+}} &= 1 + az^{(-2,0)} + z^{(-4,0)} \\
        f_{\mathfrak{d}_n^{-}} &= \begin{cases} 1 + az^{(-2(n+1),n+2)} + z^{(-4(n+1),2(n+2))} & ~~~ n \textrm{ odd} \\ 1 + z^{(-(n+1),(n+2)/2)} & ~~~ n  
        \textrm{ even} \end{cases} \\
        f_{\mathfrak{d}_n^{+}} &= \begin{cases} 1 + z^{(-2(n+1),n)} & \qquad~~~~~n \textrm{ odd} \\ 1 + az^{(-(n+1),n/2)} + z^{(-2(n+1),n)} & \qquad~~~~~n \textrm{ even} \end{cases}
    \end{align*}}  \\ \hline
    \begin{minipage}{\textwidth}
        \begin{tikzpicture}[scale=0.75]
        \node at (-1.5,1.5) {$\mathcal{A}_{(1,2),(2,1)}$};
        
        \draw (0,-3) to (0,3);
        \draw (-3,0) to (3,0);
        
        \draw (0,0) to (2.25,-3);
        \draw (0,0) to (2,-3);
        \draw (0,0) to (1.875,-3);
        \draw (0,0) to (1.8,-3);
        
        \draw (0,0) to (0.75,-3);
        \draw (0,0) to (1,-3);
        \draw (0,0) to (1.125,-3);
        \draw (0,0) to (1.2,-3);
        
        \filldraw (0,0) to (1.75,-3) to (1.25,-3) to (0,0);

        \draw[thick,red] (0,0) to (1.5,-3);
        
        \node[scale=0.85] at (0,3.3) {$\mathfrak{d}_{0}^{-}$};
        \node[scale=0.85] at (3.3,0) {$\mathfrak{d}_{0}^{+}$};
        
        \node[scale=0.85] at (2.55,-3.3) {$\mathfrak{d}_1^{+}$};
        \node[scale=0.75] at (2,-3.3) {$\dots$};
        
        \node[scale=0.85] at (0.6,-3.3) {$\mathfrak{d}_1^{-}$};
        \node[scale=0.75] at (1.1,-3.3) {$\dots$};
        
        \node[scale=0.85] at (1.5,-3.3) {$f$};
        
    \end{tikzpicture}\end{minipage} &
    \vspace{-8em}    
    {\begin{align*}
        f_{\mathfrak{d}_0^{-}} &= 1 + az^{(0,2)} + z^{(0,4)} \\
        f_{\mathfrak{d}_0^{+}} &= 1 + z^{(-1,0)} \\
        f_{\mathfrak{d}_n^{-}} &= \begin{cases} 1 + z^{(-n,2(n+1))} & \qquad~~~~~ n \textrm{ odd} \\ 1 + az^{(-n/2,n+1)} + z^{-n,2(n+1)} & \qquad~~~~~ n \textrm{ even} \end{cases} \\
        f_{\mathfrak{d}_n^{+}} &= \begin{cases} 1 + az^{-(n+4),2(n+3)} + z^{(-2(n+4),4(n+3))} & n \textrm{ odd} \\ 1 + z^{(-(n+4)/2,n+3)} & n \textrm{ even} \end{cases}
    \end{align*}} \\ \hline
             \begin{minipage}{\textwidth}
        \begin{tikzpicture}[scale=0.75]
        \node at (-1.5,1.5) {$\mathcal{A}_{(1,1),(2,2)}$};
        
        \draw (0,-3) to (0,3);
        \draw (-3,0) to (3,0);
        
        \draw (0,0) to (3,-1.5);
        \draw (0,0) to (3,-2);
        \draw (0,0) to (3,-2.25);
        \draw (0,0) to (3,-2.4);
        \draw (0,0) to (3,-2.5);
        \draw (0,0) to (3,-2.57);
        \draw (0,0) to (3,-2.63);
        
        \draw (0,0) to (1.5,-3);
        \draw (0,0) to (2,-3);
        \draw (0,0) to (2.25,-3);
        \draw (0,0) to (2.4,-3);
        \draw (0,0) to (2.5,-3);
        \draw (0,0) to (2.57,-3);
        \draw (0,0) to (2.63,-3);
        
        \filldraw (0,0) to (3,-2.63) to (3,-3) to (2.63,-3) to (0,0);
        
        \draw[thick,red] (0,0) to (3,-3);
        
        \node[scale=0.85] at (0,3.3) {$\mathfrak{d}_{0}^{-}$};
        \node[scale=0.85] at (3.3,0) {$\mathfrak{d}_{0}^{+}$};
        
        \node[scale=0.85] at (3.3,-1.5) {$\mathfrak{d}_{1}^{+}$};
        \node[scale=0.85] at (3.3,-2) {$\mathfrak{d}_{2}^{+}$};
        \node[scale=0.85] at (3.3,-2.5) {$\vdots$};
        
        \node[scale=0.85] at (1.5,-3.3) {$\mathfrak{d}_{1}^{-}$};
        \node[scale=0.85] at (2,-3.3) {$\mathfrak{d}_{2}^{-}$};
        \node[scale=0.85] at (2.6,-3.3) {$\hdots$};
        
        \node[scale=0.85] at (3.3,-3.3) {$f$};
    \end{tikzpicture}\end{minipage} &
    \vspace{-8em}    
    {\begin{align*}
        f_{\mathfrak{d}_{0}^{-}} &= 1 + az^{(-1,0)} + z^{(-2,0)} \\
    f_{\mathfrak{d}_{0}^{+}} &= 1 + bz^{(0,1)} + z^{(0,2)} \\
    f_{\mathfrak{d}_{n}^{+}} &= \begin{cases} 1 + bz^{(-(n+1),n)} + z^{(-2(n+1),2n)} & n \textrm{ odd} \\ 1 + az^{(-(n+1),n)} + z^{(-2(n+1),2n)} & n \textrm{ even} \end{cases} \\
    f_{\mathfrak{d}_{n}^{-}} &= \begin{cases} 1 + az^{(-n,n+1)} ~~ + z^{(-2n,2(n+1))} & n \textrm{ odd} \\ 
    1 + bz^{(-n,n+1)} ~~ + z^{(-2n,2(n+1))} & n \textrm{ even } \end{cases}
    \end{align*}}\\ \hline
    \end{tabularx}
    \caption{The known ordinary cluster scattering diagram for $\mathcal{A}_{(2,2),(1,1)}$ and partially computed generalized cluster scattering diagrams for $\cA_{(1,2),(2,1)}$, $\cA_{(2,1),(1,2)}$, and $\cA_{(1,1),(2,2)}$ discussed in Example~\ref{eg:whywecare}.}
    \label{fig:table_r_vs_d_example}
\end{figure}
For this set of generalized torus seeds, we have
\begin{align*}
    \mathfrak{D}_{\textrm{in},\mathcal{A}_{(2,2),(1,1)}} &= \left\{ \left((1,0)^{\perp}, 1 + z^{(0,2)}\right), \left((0,1)^{\perp}, 1 + z^{(-2,0)}\right) \right\}, \\
    \mathfrak{D}_{\textrm{in},\mathcal{A}_{(2,1),(1,2)}} &= \left\{ \left((1,0)^{\perp}, 1 + z^{(0,1)}\right), \left((0,1)^{\perp}, 1 + az^{(-2,0)} + z^{(-4,0)}\right)  \right\}, \\
    \mathfrak{D}_{\textrm{in},\mathcal{A}_{(1,2),(2,1)}} &= \left\{ \left((1,0)^{\perp}, 1 + az^{(0,2)} + z^{(0,4)}\right), \left((0,1)^{\perp}, 1 + z^{(-1,0)}\right) \right\}, \\
    \mathfrak{D}_{\textrm{in},\mathcal{A}_{(1,1),(2,2)}} &= \left\{ \left((1,0)^{\perp}, 1 + az^{(0,1)} + z^{(0,2)}\right), \left((0,1)^{\perp}, 1 + bz^{(-1,0)} + z^{(-2,0)}\right) \right\}.
\end{align*}
The table in Figure~\ref{fig:table_r_vs_d_example} shows the known cluster scattering diagram for $\mathcal{A}_{(2,2),(1,1)}$ and partially computed cluster scattering diagrams for $\mathcal{A}_{(1,2),(2,1)}$, $\mathcal{A}_{(2,1),(1,2)}$, and  $\mathcal{A}_{(1,1),(2,2)}$. 
While Theorem~\ref{thm:gen_consistent} guarantees the existence of consistent generalized cluster scattering diagrams, the wall structures of the diagrams are not explicitly laid out. 

Observe that when $a = b = 0$, the generalized cluster scattering diagram for $\mathcal{A}_{(1,1),(2,2)}$ reduces to the known cluster scattering diagram for $\mathcal{A}_{(2,2),(1,1)}$, where the wall-crossing automorphisms are known. See \cite{reineke2010poisson} for the $\cX$ setting, and \cite[Figure 2]{cheung2019theta} for the $\cA$ setting. For  $\mathcal{A}_{(1,2),(2,1)}$, $\mathcal{A}_{(2,1),(1,2)}$, and $\mathcal{A}_{(1,1),(2,2)}$, the incoming walls can be computed by definition and the outgoing walls in the cluster complex follow from the generalized mutation rule. It is difficult, however, to determine what wall-crossing automorphism $f$ should be attached to the limiting ray to make the diagram consistent.
\end{example}

\subsection{Mutation invariance of generalized cluster scattering diagrams}
\label{subsec:genMutInvariance}

We must slightly tweak the mutation of ordinary scattering diagrams. We use the same definitions of $\mathcal{H}_{k,+}$ and $\mathcal{H}_{k,-}$, but modify the definition of $T_k$ and the procedure for applying $T_k$ as follows:

\begin{definition}
\label{def:T_k}
We define the piecewise linear transformation $T_k: M^\circ \rightarrow M^\circ$ as
\[ T_k(m) := \begin{cases}
                m + r_kv_k\langle d_ke_k,m \rangle & m \in \mathcal{H}_{k,+} \\
                m & m \in \mathcal{H}_{k,-}
            \end{cases}\]
As before, we sometimes use the shorthand $T_{k,-}$ and $T_{k,+}$ to refer to $T_k$ in, respectively, the regions $\mathcal{H}_{k,+}$ and $\mathcal{H}_{k,-}$.
\end{definition}

\begin{definition}
\label{def:diagramMutProcedure}
The scattering diagram $T_k(\mathfrak{D}_\bfs)$ is obtained from $\mathfrak{D}_{\bfs}$ via the following procedure:
\begin{enumerate}
    \item For each wall $(\mathfrak{d},f_{\mathfrak{d}})$ in $\mathfrak{D}_{\bfs}$ other than $\mathfrak{d}_k := (e_k^{\perp},1 + a_1z^{v_k} + \cdots + a_{r_k-1}z^{(r_k-1)v_k} + z^{r_kv_k})$, there are either one or two corresponding walls in $T_k(\mathfrak{D}_\bfs)$. If $\textrm{dim}(\mathfrak{d} \cap \mathcal{H}_{k,-}) \geq \textrm{rank}(M) - 1$, then add to $T_k(\mathfrak{D}_\textbf{s})$ the wall $(T_k(\mathfrak{d} \cap \mathcal{H}_{k,-}),T_{k,-}(f_{\mathfrak{d}}))$ where the notation $T_{k,\pm}(f_{\mathfrak{d}})$ indicates the formal power series obtained by applying $T_{k,\pm}$ to the exponent of each term of $f_{\mathfrak{d}}$. If $\textrm{dim}(\mathfrak{d} \cap \mathcal{H}_{k,+}) \geq \textrm{rank}(M) - 1$, add the wall $(T_k(\mathfrak{d} \cap \mathcal{H}_{k,+}),T_{k,+}(f_{\mathfrak{d}}))$.
    \item The wall $\mathfrak{d}_k$ in $\mathfrak{D}_\bfs$ becomes the wall $\mathfrak{d}_k' = (e_k^{\perp},1 + a_{1}z^{-v_k} + \cdots + a_{r_k-1}z^{-(r_k-1)v_k} + z^{-r_kv_k})$ in $T_k(\mathfrak{D}_\bfs)$.
\end{enumerate}
\end{definition}

\begin{example}
\label{ex:gen_T_k}
\noindent Consider the generalized cluster algebra
\[ \mathcal{A}\left(\mathbf{x},\mathbf{y},\begin{bmatrix} 0 & 1 \\ -1 & 0 \end{bmatrix},\begin{bmatrix} 3 & 0 \\ 0 & 1 \end{bmatrix}, ((1,a,a,1),(1,1)) \right)\]
with seed data $\bfs = (((1,0),(1,a,a,1)),((0,1),(1,1)))$. 
Note that the injectivity assumption is satisfied in this example. 
For this algebra, we have $r_1 = 3$, $r_2 = 1$, and $d_1 = d_2 = 1$. By definition, this means that
\begin{align*}
    \epsilon_{12} &= \{ e_1, e_2 \}d_2 = 1 \\
    \epsilon_{21} &= \{ e_2, e_1 \}d_1 = -1 \\
    v_1 &= p_1^*((1,0)) = (0,1) \\
    v_2 &= p_1^*((0,1)) = (-1,0)
\end{align*}
and the initial cluster scattering diagram is
\begin{align*}
    \mathfrak{D}_{\textrm{in},\bfs} &= \left\{ ((1,0),1+z^{(-1,0)}),((0,1),1+az^{(0,1)}+az^{(0,2)}+z^{(0,3)}) \right\}.
\end{align*}
Adding walls to make $\mathfrak{D}_{\textrm{in},\bfs}$ consistent, we obtain the following:
\vspace{2mm}
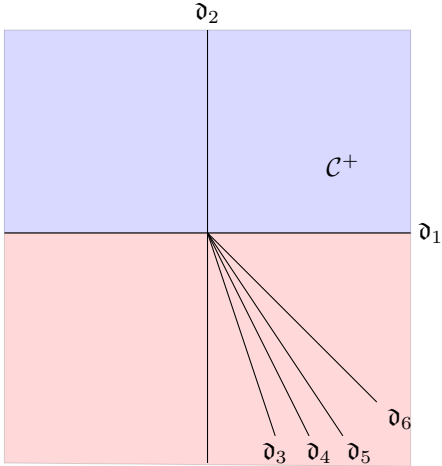
\begin{center}
    \begin{minipage}{0.44\textwidth}
	\begin{tikzpicture}[scale=0.9]
	    \draw [fill=blue, opacity=0.15] (-3,0) to (3,0) to (3,3) to (-3,3) to (-3,0);
	    \draw [fill=red, opacity=0.15] (-3,0) to (3,0) to (3,-3.45) to (-3,-3.4) to (-3,0);
	
		\draw (-3,0) to (3,0);
		\draw (0,-3.4) to (0,3);
		\draw (0,0) to (2.5,-2.5);
		\draw (0,0) to (2,-3);
		\draw (0,0) to (1.5,-3);
		\draw (0,0) to (1,-3);
				
		\node at (3.3,0) {$\mathfrak{d}_1$};
		\node at (0,3.25) {$\mathfrak{d}_2$};
		\node at (2.85,-2.75) {$\mathfrak{d}_6$};
		\node at (2.25,-3.25) {$\mathfrak{d}_5$};
		\node at (1.65,-3.25) {$\mathfrak{d}_4$};
		\node at (1,-3.25) {$\mathfrak{d}_3$};
		
		\node at (2,1) {$\mathcal{C}^+$};
	\end{tikzpicture}
	\end{minipage}
	\qquad
	\begin{minipage}{0.44\textwidth}
	\begin{align*}
        f_{\mathfrak{d}_1} &= 1+z^{(-1,0)} \\
        f_{\mathfrak{d}_2} &= 1+az^{(0,1)}+az^{(0,2)}+z^{(0,3)} \\
        f_{\mathfrak{d}_3} &= 1+z^{(-1,3)} \\
        f_{\mathfrak{d}_4} &= 1 + az^{(-1,2)}+az^{(-2,4)}+z^{(-3,6)} \\
        f_{\mathfrak{d}_5} &= 1 + z^{(-2,3)} \\
        f_{\mathfrak{d}_6} &= 1 + az^{(-1,1)}+az^{(-2,2)}+z^{(-3,3)}
    \end{align*}
	\end{minipage}
\end{center}
\noindent By definition, we have the half-planes
\begin{align*}
    \mathcal{H}_{2,+} &= \left\{ (0,y): y > 0 \right\} \\
    \mathcal{H}_{2,-} &= \left\{ (0,y): y < 0 \right\}
\end{align*}
which are shown on $\mathfrak{D}_{\bfs}$ in blue, for $\mathcal{H}_{2,+}$, and red, for $\mathcal{H}_{2,-}$. To mutate in direction $k=2$, we will use the linear transformation
\begin{align*}
T_2(m) &= \begin{cases}
			m + (-1,0)\langle (0,1),m\rangle & m \in \mathcal{H}_{2,+} \\
			m & m \in \mathcal{H}_{2,-}
		\end{cases}
\end{align*}
Because $T_2$ fixes the walls in $\mathcal{H}_{2,-}$, the only walls that change under $T_2$ are $\mathfrak{d}_1$ and $\mathfrak{d}_2 \cap \mathbb{R}_{\geq 0}(0,1)$. Because the support of $\mathfrak{d}_1$ is $e_2^\perp = (1,0)$, it is  transformed via the procedure outlined in (2) of Definition \ref{def:diagramMutProcedure} and becomes
\[ \mathfrak{d}_1' = (e_2^\perp, 1 + z^{(1,0)}) \]
To determine the image of $\mathfrak{d}_2 \cap \mathbb{R}_{\geq 0}(0,1)$, we compute
\begin{align*}
    T_2((0,1)) &= (0,1) + (-1,0)\langle (0,1),(0,1) \rangle = (-1,1)
\end{align*}
Because $T_2$ is a linear transformation, we know that $T_2((0,2)) = (-2,2)$ and $T_2((0,3)) = (-3,3)$. As such,
\[ T_2(\mathfrak{d}_2 \cap \mathbb{R}_{\geq 0}(0,1)) = (\mathbb{R}_{\geq 0}(-1,1),1+a^{(-1,1)}+az^{(-2,2)}+z^{(-3,3)}) \]
and we draw $T_2(\mathfrak{D}_{\bfs})$ as
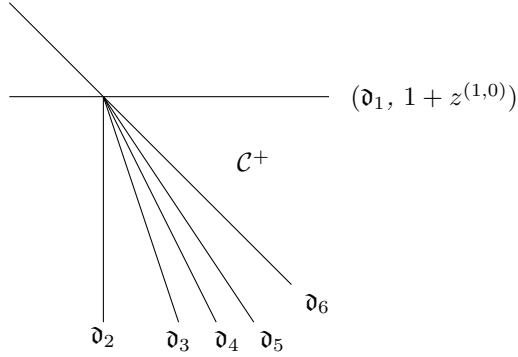
\begin{center}
	\begin{tikzpicture}
		\draw (-1.25,0) to (3,0);
		\draw (0,-3) to (0,0);
		\draw (-1.25,1.25) to (2.5,-2.5);
		\draw (0,0) to (2,-3);
		\draw (0,0) to (1.5,-3);
		\draw (0,0) to (1,-3);
				
		\node at (4.4,0) {$(\mathfrak{d}_1$, $1+z^{(1,0)})$};
		\node at (0,-3.2) {$\mathfrak{d}_2$};
		\node at (2.85,-2.75) {$\mathfrak{d}_6$};
		\node at (2.25,-3.25) {$\mathfrak{d}_5$};
		\node at (1.65,-3.25) {$\mathfrak{d}_4$};
		\node at (1,-3.25) {$\mathfrak{d}_3$};
		
		\node at (2,-0.85) {$\mathcal{C}^+$};
	\end{tikzpicture}
\end{center}
where $f_{\mathfrak{d}_2}, f_{\mathfrak{d}_3}, f_{\mathfrak{d}_4}, f_{\mathfrak{d}_5}$, and $f_{\mathfrak{d}_6}$ are the same automorphisms as in $\mathfrak{D}_{\bfs}$. We can also compute the new basis vectors $e_1'$ and $e_2'$ using Definition \ref{def:basisMutGen}:
\begin{align*}
    e_1' &= e_1 + r_2[\epsilon_{12}]_+e_2 \\
         &= (1,0) + (0,1) \\
         &= (1,1) \\
    e_2' &= -e_2 = (0,-1)
\end{align*}
Because $\mathcal{A}$ has exchange polynomials with reciprocal coefficients, the exchange polynomial coefficients are fixed under mutation. So we have \[ \mu_2(\bfs) = (((1,1),(1,a,a,1)),((0,-1),(1,1))) \]
Recall that by definition, the basis vectors $e_1 = (1,0)$ and $e_2 = (0,1)$ of the original torus seed $\bfs$ form a basis for the lattice $N$. Observe that the vectors $e_1' = (1,1)$ and $e_2' =  (0,-1)$ obtained via torus seed mutation are another choice of basis for $N$.
\end{example}
Each cluster mutation $\mu_k$ can be defined by a triple $(n,m,r) \in N \times M \times \mathbb{Z}_{\geq 0}$ with $\langle n,m \rangle = 0$. Emulating the notation of \cite{GHK}, we denote this mutation as $\mu_{(n,m,r)}$. It is defined by the pullback 
\[ \mu^*_{(n,m,r)}\left(z^{m'} \right) = z^{m'} \cdot \left( 1 + a_{1}z^{m} + \cdots + a_{r-1}z^{(r-1)m} + z^{rm} \right)^{\langle n, m' \rangle}, \]
where $a_1, \dots, a_{r-1}$ are scalars and $r \in \mathbb{Z}_{\geq 0}$. 

Recall from Section~\ref{subsec:ordDiagrams} that we refer to the max-plus tropicalization of $\mathbb{Z}$ as the \emph{Fock-Gonacharov tropicalization} and denote it as $\mathbb{Z}^T$.
Let $\mu: T_N \rightarrow T_N$ be a positive birational map. Then $\mu^T: N \rightarrow N$
denotes the induced map
$T_N(\mathbb{Z}^T) \rightarrow T_N(\mathbb{Z}^T)$.

\begin{prop}[Analogue of Proposition 2.4 of \cite{GHKK}]
\label{prop:T_k_tropicalization}
The map ${T_k: M^\circ \rightarrow M^\circ}$ given in Definition~\ref{def:T_k} is the Fock-Goncharov tropicalization of the  map ${\mu_{(v_k,d_ke_k,r_k)}}$.
\end{prop}

\begin{proof}
The map $\mu_{(d_ke_k,v_k,r_k)}: T_{M^{\circ}} \rightarrow T_{M^{\circ}}$ is defined by the pullback
\begin{align*} \mu_{(v_k,d_ke_k,r_k)}^*\left(z^{m} \right) &= z^m \left( 1+ a_{1}z^{v_k} + \cdots + a_{r_k-1}z^{(r_k-1)v_k} + z^{r_kv_k} \right)^{\langle d_ke_k, m \rangle}
\end{align*}
By definition, $\mu_{\langle d_ke_k,v_k,r_k \rangle}$ has Fock-Goncharov tropicalization
\begin{align*}
    \mu_{d_ke_k,v_k,r_k}^T : N &\rightarrow N \\
    x &\mapsto x + r_k[\langle d_ke_k, x \rangle]_{+}v_k
\end{align*}
Observe that when $x \in \mathcal{H}_{k,-}$, then $\langle d_ke_k,x \rangle \leq 0$ and the above map reduces to $x \mapsto x$. When $x \in \mathcal{H}_{k,+}$, then $\langle d_ke_k, \rangle \geq 0$ and the map reduces to $x \mapsto x + r_kv_k\langle d_ke_k,x \rangle$. Hence, the tropicalization agrees exactly with our definition of $T_k$, as desired.
\end{proof}

\begin{theorem}\cite[Theorem 1.24]{GHKK}
\label{theorem:mutDiagConsistent}
If the injectivity assumption holds, then $T_k(\mathfrak{D}_{\bfs})$ is a consistent scattering diagram for $N^+_{\mu_k(\bfs)}$. Moreover, the diagrams $\mathfrak{D}_{\mu_k(\bfs)}$ and $T_k(\mathfrak{D}_{\bfs})$ are equivalent.
\end{theorem}

We need to show that $T_k(\mathfrak{D}_\bfs)$ is a scattering diagram for $\mathfrak{g}_{\mu_k(\bfs)}$ and $N^+_{\mu_k(\bfs)}$. As in the ordinary case, the major technical hurdle in doing so is the fact that the wall-crossing automorphisms of $\mathfrak{D}_\bfs$ and $\mathfrak{D}_{\mu_k(\bfs)}$ live in different completed monoid rings. Those in $\mathfrak{D}_\bfs$ live in $\widehat{R[P]}$, where $P$ is the monoid generated by $\{ v_i \}_{i \in I_{\textrm{uf}}}$. Those in $\mathfrak{D}_{\mu_k(\bfs)}$ live, instead, in $\widehat{R[P']}$, where $P'$ is the monoid generated by $\{ v_i' \}_{i \in I_{\textrm{uf}}}$.

To overcome this difficulty, we define an additional monoid $\overline{P}$ which contains both $P$ and $P'$. Let $\sigma \subseteq M^\circ$ be a top-dimensional cone which contains the vectors $\{v_i \}_{i \in I_{\textrm{uf}}}$ and $-v_k$, such that $\sigma \cap (-\sigma) = \mathbb{R}v_k$. For a fixed choice of $\sigma$, let $\overline{P} := \sigma \cap M^\circ$ and $J = \overline{P} \backslash (\overline{P} \cap \mathbb{R}v_k) = \overline{P} \backslash \overline{P}^{\times}$.

Even after choosing an appropriate monoid $\overline{P}$, we still have to deal with the fact that the wall-crossing automorphism associated to the wall
\[ \mathfrak{d}_k = \left(e_k^\perp, 1 + a_{k,1}z^{v_k} + \cdots + a_{k,r_k-1}z^{(r_k-1)v_k} + z^{v_k} \right) =: (e_k^\perp, f_k) \]
is an automorphism of the localization $\widehat{R[\overline{P}]}_{f_k}$ rather than the ring $\widehat{R[\overline{P}]}$, where the completions are with respect to the ideal $J$. Let $\mathfrak{p}_{\mathfrak{d}_k} \in \widehat{R[\overline{P}]}_{f_k}$ denote the automorphism associated with crossing $\mathfrak{d}_k$ from $\mathcal{H}_{k,-}$ into $\mathcal{H}_{k,+}$. By definition,
\[ \mathfrak{p}_{\mathfrak{d}_k}(z^m) = z^m(1+a_{k,1}z^{v_k} + \cdots + a_{k,r_k-1}z^{(r_k-1)v_k} + z^{v_k})^{-\langle d_ke_k,m \rangle}. \]
We can then define
\[ N_\bfs^{+,k} := \left\{ \left. \sum_{i \in I_{\textrm{uf}}} a_ie_i ~\right|~ a_i \in \mathbb{Z}_{\geq 0} \textrm{ for } i \neq k, a_k \in \mathbb{Z}, \textrm{ and} \sum_{i \in I_{\textrm{uf}} \backslash \{ k \}} a_i > 0 \right\}. \]
Because $\bfs' = (\bfs \backslash \{ v_k \}) \cup \{ - v_k \}$, the conditions of this definition mean that $N_\bfs^{+,k} = N_{\bfs'}^{+,k}$. As such, we can use the abbreviated notation $N^{+,k}$ without introducing any ambiguity.

To allow us to work in $\overline{P}$, we need to slightly modify the definition of a scattering diagram:

\begin{definition}
\label{def:extendedScatDiag}
Given the monoid $\overline{P}$ and ideal $J$, a wall is a pair $(\mathfrak{d},f_\mathfrak{d})$ such that for some primitive $n_0 \in N^{+,k}$,
    \begin{enumerate}
        \item $f_\mathfrak{d} \in \widehat{R[\overline{P}]}$ has the form $1 + \sum_{k=1}^\infty c_kz^{k p^*(n_0)}$ and is congruent to $1$ mod $J$,
        \item  $\mathfrak{d} \subset n_0^\perp \subset M_\mathbb{R}$ is a $(\textrm{rank}(N)-1)$-dimensional convex rational polyhedral cone.
    \end{enumerate}
For a seed $\bfs$, the slab is ${\mathfrak{d}_k = (e_k^\perp,1+a_{k,1}z^{v_k} + \cdots + a_{k,r_k-1}z^{(r_k-1)v_k} + z^{v_k})}$. Because $v_k \in \overline{P}^\times$, the slab does not qualify as a wall under the above definition. So we extend the definition of a \emph{scattering diagram} $\mathfrak{D}$ such that:
\begin{enumerate}
    \item $\mathfrak{D}$ contains a collection of walls and potentially the slab $\mathfrak{d}_k$, and
    \item for $k > 0$, we have $f_\mathfrak{d} \equiv 1 \textrm{mod}~J^k$ for all but finitely many walls of $\mathfrak{D}$.
\end{enumerate}
\end{definition}

In this modified scattering diagram, crossing a wall or slab $(\mathfrak{d},f_\mathfrak{d})$ induces an automorphism $\mathfrak{p}^{\pm 1}_{f_\mathfrak{d}} \in \widehat{R[\overline{P}]}_{f_k}$. Note that the localization at $f_k$ is only really required when crossing $\mathfrak{d}_k$, as otherwise $f_\mathfrak{d}$ lives in $\widehat{R[\overline{P}]}$.

The proof of Theorem \ref{theorem:mutDiagConsistent} requires the following result:

\begin{theorem}[Analogue of Theorem 1.28 of \cite{GHKK}]
\label{theorem:scatDiagramUniqueness}
There exists a scattering diagram $\overline{\mathfrak{D}}_\bfs$ such that
\begin{itemize}
    \item $\overline{\mathfrak{D}_\bfs} \supseteq \mathfrak{D}_{\textrm{in},\bfs}$,
    \item $\overline{\mathfrak{D}_\bfs} \backslash \mathfrak{D}_{\textrm{in},\bfs}$ consists of only outgoing walls,
    \item and the path-ordered product $\mathfrak{p}_{\gamma,\mathfrak{D}} \in \widehat{R[\overline{P}]_{f_k}}$ depends only on the endpoints of $\gamma$.
\end{itemize}
Such $\overline{\mathfrak{D}}_\bfs$ is unique up to equivalence. Further, because $\overline{\mathfrak{D}}_\bfs$ is also a scattering diagram for $N^+_\bfs$, it is equivalent to $\mathfrak{D}_\bfs$. Moreover, this implies that the only wall contained in $e_k^\perp$ is the slab $\mathfrak{d}_k$
\end{theorem}
The proof given in \cite{GHKK} in the ordinary setting also holds in our generalized setting. Because that proof is quite lengthy, we do not reproduce it here.

We will also need the following definition:

\begin{definition}
A codimension two convex rational polyhedral cone $\mathfrak{j}$ is a \emph{joint} of the scattering diagram $\mathfrak{D}$ if either every wall $\mathfrak{d} \subseteq n^\perp$ that contains $\mathfrak{j}$ has direction $-p^*(n) = -\{n, \cdot \}$ tangent to $\mathfrak{j}$ or direction not tangent to $\mathfrak{j}$. In the first case, where every wall is tangent to $\mathfrak{j}$, we call the joint \emph{parallel}. In the second case, we call the joint \emph{perpendicular}.
\end{definition}

We're now prepared to prove Theorem \ref{theorem:mutDiagConsistent}:

\begin{proof}
Let $\bfs = \{ e_i \}_{i \in I}$ be a fixed choice of generalized torus seed and $\bfs' := \mu_k(\bfs) = \{e_i' \}_{i \in I}$. From Theorem \ref{theorem:scatDiagramUniqueness}, we know that the scattering diagrams for $\bfs$ and $\bfs'$ are unique up to equivalence and therefore we can choose representative scattering diagrams $\mathfrak{D}_\bfs$ and $\mathfrak{D}_{\bfs'}$.

Notice that if $z^m \in J^i$ for some $i > 0$, then $z^{T_{k,\pm}(m)} \in J^i$. As such, $T_k(\mathfrak{D}_\bfs)$ is also a scattering diagram for the seed $\bfs'$ in the slightly extended sense of Definition \ref{def:extendedScatDiag}. In order to use Theorem \ref{theorem:scatDiagramUniqueness} to show that $\mathfrak{D}_{\bfs'}$ and $T_k(\mathfrak{D}_\bfs)$ are equivalent, we need to (1) verify that $T_k(\mathfrak{D}_\bfs)$ is consistent and (2) show that both diagrams are equivalent to diagrams with the same set of slabs and incoming walls.

We can begin by showing that $T_k(\mathfrak{D}_\bfs)$ is consistent. To do so, we need to show that $\mathfrak{p}_{\gamma,T_k(\mathfrak{D}_\bfs)} = \textrm{id}$ for any loop $\gamma$ for which the path-ordered product is defined. Because $\mathfrak{D}_\bfs$ is consistent and so by definition $\mathfrak{p}_{\gamma,\mathfrak{D}} = \textrm{id}$ whenever the path-ordered product is defined, one strategy is to show that $\mathfrak{p}_{\gamma,T_k(\mathfrak{D}_\bfs)} =  \mathfrak{p}_{\gamma,\mathfrak{D}_\bfs}$ and therefore $\mathfrak{p}_{\gamma,T_k(\mathfrak{D}_\bfs)} = \textrm{id}$. In areas of $\mathfrak{D}_\bfs$ where $T_k$ is linear, the consistency of a loop in $T_k(\mathfrak{D}_\bfs)$ is an immediate consequence of linearity since each wall is crossed either not at all or in both possible directions.

Therefore, we need only be concerned about when $\gamma$ is a loop around a joint $\mathfrak{j}$ of $\mathfrak{D}_\bfs$ which is contained in the slab $\mathfrak{d}_k$. Given such $\gamma$, we can subdivide it as $\gamma = \gamma_1\gamma_2\gamma_3\gamma_4$ where $\gamma_1$ crosses $\mathfrak{d}_k$, $\gamma_2 \subseteq \mathcal{H}_{k,+}$ contains all the crossings of walls in $\mathfrak{D}_\bfs$ which contain $\mathfrak{j}$ and lie in $\mathcal{H}_{k,+}$, $\gamma_3$ crosses $\mathfrak{d}_4$, and $\gamma_4$ contains all the crossings of walls in $\mathfrak{D}_\bfs$ that contain $\mathfrak{j}$ and lie in $\mathcal{H}_{k,-}$. We can also assume that it has a basepoint $Q$ in $\mathcal{H}_{k,-}$.

One example of a possible subdivision of $\gamma$ is shown below:

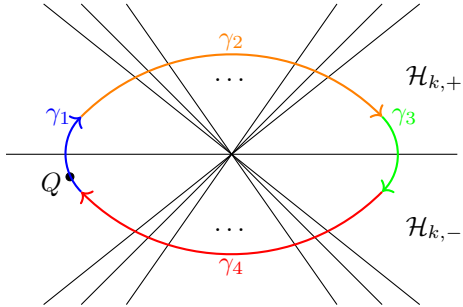
\begin{center}
    \begin{tikzpicture}
        \draw (-3,0) to (3,0);
        \draw (0,0) to (-2.5,2);
        \draw (0,0) to (2.5,2);
        \draw (0,0) to (-2,2);
        \draw (0,0) to (2,2);
        \draw (0,0) to (-1.4,2);
        \draw (0,0) to (1.4,2);
        \draw (0,0) to (-2.5,-2);
        \draw (0,0) to (2.5,-2);
        \draw (0,0) to (-2,-2);
        \draw (0,0) to (2,-2);
        \draw (0,0) to (-1.4,-2);
        \draw (0,0) to (1.4,-2);
        
        \node at (0,1) {$\dots$};
        \node at (0,-1) {$\dots$};
        
        \filldraw (-2.15,-0.3) circle (1.5pt);
        \node at (-2.4,-0.4) {$Q$};
        
        \draw[color=blue,thick,out=135,in=-135,->] (-2,-0.5) to (-2,0.5);
        \draw[color=orange,thick,out=45,in=135,->] (-2,0.5) to (2,0.5);
        \draw[color=green,thick,out=-45,in=45,->] (2,0.5) to (2,-0.5);
        \draw[color=red,thick,in=-45,out=-135,->] (2,-0.5) to (-2,-0.5);
        
        \node[orange] at (0,1.5) {$\gamma_2$};
        \node[red] at (0,-1.5) {$\gamma_4$};
        \node[blue] at (-2.3,0.5) {$\gamma_1$};
        \node[green] at (2.3,0.5) {$\gamma_3$};
        
        \node at (2.7,1) {$\mathcal{H}_{k,+}$};
        \node at (2.7,-1) {$\mathcal{H}_{k,-}$};
    \end{tikzpicture}
\end{center}

Let $\mathfrak{p}_{\mathfrak{d}_k}$ denote the wall-crossing automorphism for crossing $\mathfrak{d}_k$ from $\mathcal{H}_{k,-}$ into $\mathcal{H}_{k,+}$. Similarly, let $\mathfrak{p}_{\mathfrak{d}_k}$ denote crossing $\mathfrak{d}_k'$ from $\mathcal{H}_{k,-}$ into $\mathcal{H}_{k,+}$. Explicitly, we have
\begin{align*}
    \mathfrak{p}_{\mathfrak{d}_k}(z^m) &= z^m\left(1 + a_{k,1}z^{v_k} + \cdots + a_{k,r_k-1}z^{(r_k-1)v_k} + z^{r_kv_k} \right)^{-\langle d_ke_k, m \rangle} \\
    \mathfrak{p}_{\mathfrak{d}_k'}(z^m) &= z^m\left(1 + a_{k,1}z^{-v_k} + \cdots + a_{k,r_k-1}z^{-(r_k-1)v_k} + z^{-r_kv_k} \right)^{-\langle d_ke_k, m \rangle}
\end{align*}
Because $\mathfrak{d}_k$ is the only wall contained in $e_k^\perp$, we know that $\mathfrak{p}_{\gamma_1,\mathfrak{D}_\bfs} = \mathfrak{p}_{\mathfrak{d}_k}$ and $\mathfrak{p}_{\gamma_3,\mathfrak{D}_\bfs} = \mathfrak{p}_{\mathfrak{d}_k}^{-1}$. Let $\alpha: \Bbbk[M^\circ] \rightarrow \Bbbk[M^\circ]$ be the automorphism $\alpha(z^m) = z^{m + r_kv_k \langle d_ke_k, m \rangle}$ induced by $T_{k,+}$. Let $v_i' = p^*(e_i')$. Because $e_k' = -e_k$, observe that the slab for $\bfs'$ is
\begin{align*}
    \mathfrak{d}_k' &= \left((e_k')^\perp,1+a_{k,1}z^{v_k'}+ \cdots + a_{k,r_k-1}z^{(r_k-1)v_k'} + z^{r_kv_k'}\right) \\
    &= \left(e_k^\perp,1+a_{k,1}z^{-v_k}+ \cdots + a_{k,r_k-1}z^{-(r_k-1)v_k} + z^{-r_kv_k}\right),
\end{align*}
We can then observe the following relationships:
\begin{align*}
    \mathfrak{p}_{\gamma_1,T_k(\mathfrak{D}_\bfs)} &= \mathfrak{p}_{\mathfrak{d}_k'} \\
    \mathfrak{p}_{\gamma_2,T_k(\mathfrak{D}_\bfs)} &= \alpha \circ \mathfrak{p}_{\gamma_2,\mathfrak{D}_\bfs} \circ \alpha^{-1} \\
    \mathfrak{p}_{\gamma_3,T_k(\mathfrak{D}_\bfs)} &= \mathfrak{p}_{\mathfrak{d}_k'}^{-1} \\
    \mathfrak{p}_{\gamma_4,T_k(\mathfrak{D}_\bfs)} &= \mathfrak{p}_{\gamma_4,\mathfrak{D}_\bfs}
\end{align*}
So we have
\begin{align*}
    \mathfrak{p}_{\gamma,\mathfrak{D}_\bfs} &= 
    \mathfrak{p}_{\gamma_4,\mathfrak{D}_\bfs} \circ \mathfrak{p}_{\gamma_3,\mathfrak{D}_\bfs} \circ \mathfrak{p}_{\gamma_2,\mathfrak{D}_\bfs} \circ \mathfrak{p}_{\gamma_1,\mathfrak{D}_\bfs} \\
    &= \mathfrak{p}_{\gamma_4,\mathfrak{D}_\bfs} \circ \mathfrak{p}_{\mathfrak{d}_k}^{-1} \circ \mathfrak{p}_{\gamma_2,\mathfrak{D}_\bfs} \circ \mathfrak{p}_{\mathfrak{d}_k}, \\
    \mathfrak{p}_{\gamma,T_k(\mathfrak{D}_\bfs)} &= \mathfrak{p}_{\gamma_4,T_k(\mathfrak{D}_\bfs)} \circ \mathfrak{p}_{\gamma_3,T_k(\mathfrak{D}_\bfs)} \circ \mathfrak{p}_{\gamma_2,T_k(\mathfrak{D}_\bfs)} \circ \mathfrak{p}_{\gamma_1,T_k(\mathfrak{D}_\bfs)} \\
    &= \mathfrak{p}_{\gamma_4,\mathfrak{D}_\bfs} \circ \mathfrak{p}_{\mathfrak{d}_k'}^{-1} \circ \alpha \circ p_{\gamma_2,\mathfrak{D}_\bfs} \circ \alpha^{-1} \circ \mathfrak{p}_{\mathfrak{d}_k'},
    \end{align*}
and showing that $\mathfrak{p}_{\gamma,\mathfrak{D}_\bfs} = \mathfrak{p}_{\gamma,T_k(\mathfrak{D}_\bfs)}$ reduces to showing that $\alpha^{-1} \circ \mathfrak{p}_{\mathfrak{d}_k'} = \mathfrak{p}_{\mathfrak{d}_k}$. Using the fact that $a_{k,i} = a_{k,r_k-i}$, observe that
\begin{align*}
    \alpha^{-1}\left(\mathfrak{p}_{\mathfrak{d}_k'}(z^m) \right) &= \alpha^{-1}\left( z^m\left(1 + a_{k,1}z^{-v_k} + \cdots + a_{k,r_k-1}z^{-(r_k-1)v_k} + z^{-r_kv_k} \right)^{-\langle d_ke_k, m \rangle} \right) \\
    &= z^{m-r_kv_k\langle d_ke_k,m \rangle}\left(1 + a_{k,1}z^{-v_k} + \cdots + a_{k,r_k-1}z^{-(r_k-1)v_k} + z^{-r_kv_k} \right)^{-\langle d_ke_k, m \rangle} \\
    &= z^m \left(z^{r_kv_k}\left(1 + a_{k,1}z^{-v_k} + \cdots + a_{k,r_k-1}z^{-(r_k-1)v_k} + z^{-r_kv_k} \right)\right)^{-\langle d_ke_k, m \rangle} \\
    &= z^m \left(z^{r_kv_k} + a_{k,1}z^{(r_k-1)v_k} + \cdots + a_{k,r_k-1}z^{v_k} + 1 \right)^{-\langle d_ke_k, m \rangle} \\
    &= z^m \left(z^{r_kv_k} + a_{k,r_k-1}z^{(r_k-1)v_k} + \cdots + a_{k,1}z^{v_k} + 1 \right)^{-\langle d_ke_k, m \rangle} \\
    &= \mathfrak{p}_{\mathfrak{d}_k}\left(z^m \right),
\end{align*}
as desired. As such, we have that $\mathfrak{p}_{\gamma,\mathfrak{D}_\bfs} = \mathfrak{p}_{\gamma,T_k(\mathfrak{D}_\bfs)}$ and therefore $\mathfrak{p}_{\gamma,T_k(\mathfrak{D}_\bfs)} = \textrm{id}$ and $T_k(\mathfrak{D}_\bfs)$ is consistent.

Next, we want to show that $T_k(\mathfrak{D}_\bfs)$ and $\mathfrak{D}_{\bfs'}$ have, up to equivalence, the same set of slabs and incoming walls. Recall that $\mathfrak{D}_{\textrm{in},\bfs'}$ contains only the slab and incoming walls of $\mathfrak{D}_{\bfs'}$, so it will suffice to show that the incoming walls and slab of $T_k(\mathfrak{D}_{\bfs})$ appear in $\mathfrak{D}_{\textrm{in},\bfs}$.

First, observe that if $\mathfrak{d} \subseteq n^\perp$ is an outgoing wall in $\mathfrak{D}_\bfs$, then it is mapped to an outgoing wall in $T_k(\mathfrak{D}_\bfs)$. This follows from the definition - recall that $\mathfrak{d}$ is outgoing if $p_1^*(n) \not\in \mathfrak{d}$. Because $T_k$ is injective, having $p_1^*(n) \not\in \mathfrak{d}$ implies $T_k(p_1^*(n)) \not\in T_k(\mathfrak{d})$. Hence, we need only consider the slab and incoming walls of $T_k(\mathfrak{D}_{\bfs})$. Equivalently, we consider the walls of $T_k(\mathfrak{D}_{\textrm{in},\bfs})$.

Recall that the slab for $\bfs'$ is $\mathfrak{d}_k' = \left(e_k^\perp,1+a_{k,1}z^{-v_k}+ \cdots + a_{k,r_k-1}z^{-(r_k-1)v_k} + z^{-r_kv_k}\right)$, which appears in both $\mathfrak{D}_{\textrm{in},\bfs'}$ and $T_k(\mathfrak{D}_{\textrm{in},\bfs})$ by definition. Next, we consider the walls $\mathfrak{d}_i = (e_i^\perp,{1+a_{i,1}z^{v_i}+ \cdots + z^{r_iv_i}})$ for $i \neq k$. To do so, we need to divide our argument into three cases based on whether $\langle v_i, e_k \rangle$ is positive, zero, or negative. Because $\mathfrak{d}_i$ is an incoming wall, it will necessarily lie in both $\mathcal{H}_{k,+}$ and $\mathcal{H}_{k,-}$.

\vspace{3mm}
\noindent \textbf{Case 1:} If $\langle e_k,v_i \rangle = 0$, then the two halves of $\mathfrak{d}_i \in \mathfrak{D}_{\textrm{in},\bfs}$ are mapped to the walls 
\begin{align*}
    &((e_i^\perp \cap \mathcal{H}_{k,+}),{1+a_{i,1}z^{T_{i,+}(v_i)}+ \cdots + a_{i,r_i-1}z^{T_{k,+}((r_i-1)v_i)} + z^{T_{k,+}(r_iv_i)}}), \\
    &((e_i^\perp \cap \mathcal{H}_{k,-}),{1+a_{i,1}z^{T_{k,-}(v_i)}+ \cdots + a_{i,r_i-1}z^{T_{k,-}((r_i-1)v_i)} + z^{T_{k,-}(r_iv_i)}})
\end{align*}
whose union is the wall
\[ \left((e_i)^\perp, 1+a_{i,1}z^{v_i}+ \cdots + a_{i,r_i-1}z^{(r_i-1)v_i} + z^{r_iv_i} \right) \]
because having $\langle v_i, e_k \rangle = 0$ means that $v_i' = T_{k,\pm}(v_i) = v_i$. Because $e_i' = e_i$, the above wall in $T_k(\mathfrak{D}_{\textrm{in},\bfs})$ is the same as the wall
\[  \left((e_i')^\perp, 1+a_{i,1}z^{v_i'}+ \cdots + a_{i,r_i-1}z^{(r_i-1)v_i'} + z^{r_iv_i'} \right),  \]
which we know by definition appears in $\mathfrak{D}_{\textrm{in},{\bfs'}}$.

\vspace{3mm}
\noindent \textbf{Case 2:} Suppose $\langle e_k,v_i \rangle > 0$. We must consider where $\mathfrak{d}_i \cap \mathcal{H}_{k,+}$ is mapped by $T_k$. This portion of $\mathfrak{d}_i$ becomes the wall
\[\mathfrak{d}_{i,+}' := \left( T_k(\mathcal{H}_{k,+} \cap e_i^\perp), 1+a_{i,1}z^{T_{i,+}(v_i)}+ \cdots + a_{i,r_i-1}z^{T_{k,+}((r_i-1)v_i)} + z^{T_{k,+}(r_iv_i)} \right) \]
in $T_k(\mathfrak{D}_\bfs)$. To see that $\mathfrak{d}_{i,+}'$ is incoming in $T_k(\mathfrak{D})_{\textrm{in},\bfs}$, observe that ${p_1^*(e_i) = v_i \in (\mathcal{H}_{k,+} \cap e_i^\perp)}$ and therefore $T_k(p_1^*(e_i)) = T_k(v_i) \in \mathfrak{d}_{i,+}'$. To argue that $\mathfrak{d}_{i,+}'$ also appears as an incoming wall in $\mathfrak{D}_{\bfs'}$, we need to show that $T_k(\mathcal{H}_{k,+} \cap e_i^\perp) \subseteq (e_i')^\perp$ and that $T_{k,+}(v_i) = v_i'$. Observe that for $m \in \mathcal{H}_{k,+} \cap e_i^\perp,$
\begin{align*}
    \langle e_i', T_k(m) \rangle &= \langle e_i + r_k[\epsilon_{ik}]_+e_k, m + r_kv_k \langle d_k e_k, m \rangle \rangle \\
    &= \langle e_i, m \rangle + \langle e_i, r_kv_k \langle d_ke_k, m \rangle \rangle + \langle r_k[\epsilon_{ik}]_+e_k, m \rangle + \langle r_k[\epsilon_{ik}]_+e_k, r_kv_k \langle d_ke_k,m \rangle \rangle \\
    &= r_k \langle d_ke_k,m \rangle \langle e_i, v_k \rangle + r_k[\epsilon_{ik}]_+\langle e_k,m \rangle \\
    &= r_k \langle d_ke_k,m \rangle \langle e_i, p_1^*(e_k) \rangle + r_kd_k \{e_i, e_k \}\langle e_k,m \rangle \\
    &= r_kd_k \{e_k,e_i \}\langle e_k,m \rangle + r_kd_k \{e_i, e_k \}\langle e_k,m \rangle \\
    &= r_kd_k \langle e_k, m \rangle \left( \{ e_k, e_i \} + \{ e_i, e_k \} \right) \\
    &= 0
\end{align*}
and therefore $T_k(m) \in (e_i')^\perp$. Next, observe that
\begin{align*}
    T_{k,+}(v_i) &= v_i + r_kv_k \langle d_ke_k, v_i \rangle \\
    &= v_i + r_kd_kv_k \langle e_k, p_1^*(e_i) \rangle \\
    &= v_i + r_kd_kv_k \{e_i, e_k \} \\
    &= v_i + r_k \epsilon_{ik} v_k \\
    &= p_1^*(e_i) + r_k\epsilon_{ik} p_1^*(e_k) \\
    &= \{ e_i, \cdot \} + r_k\epsilon_{ik} \{ e_k, \cdot \} \\
    &= \{e_i + r_k\epsilon_{ik}e_k, \cdot \} \\
    &= \{ e_i', \cdot \} \\
    &= p_1^*(e_i') \\
    &= v_i'
\end{align*}
As such, $\mathfrak{d}_i \cap \mathcal{H}_{k,+} \in \mathfrak{D}_\bfs$ is mapped to the wall
\[ \mathfrak{d}_i' = \left(T_k(\mathcal{H}_{k,+} \cap e_i^\perp), 1 + a_{i,1}z^{v_i'} + \cdots + a_{i,r_i - 1}z^{(r_i-1)v_i'} + z^{v_i'} \right) \in T_k(\mathfrak{D}_{\textrm{in},\bfs)}, \]
which is half of the wall
\[ \left((e_i')^\perp, 1 + a_{i,1}z^{v_i'} + \cdots + a_{i,r_i - 1}z^{(r_i-1)v_i'} + z^{v_i'} \right) \in \mathfrak{D}_{\textrm{in},\bfs'}.\]

\vspace{3mm}
\noindent \textbf{Case 3:} Finally, let $\langle e_k,v_i \rangle < 0$. The half of $\mathfrak{d}_i$ with support $\mathfrak{d}_i \cap \mathcal{H}_{k,-}$ is mapped by $T_k$ to
\[ \mathfrak{d}_{i,-}' := \left(T_k(\mathcal{H}_{k,-} \cap e_i^\perp), 1 + a_{i,1}z^{T_{k,-}(v_i)} + \cdots + a_{i,r_i-1}z^{T_{k,-}((r_i-1)v_i)} + z^{T_{k,-}(r_iv_i)} \right) \]
Since $T_{k,-}(m) = m$ for $m \in \mathcal{H}_{k,-} \cap e_i^{\perp}$ and $T_{k,-}(v_i) = v_i$, we have
\[ \mathfrak{d}_{i,-}' = \left(\mathcal{H}_{k,-} \cap e_i^\perp, 1 + a_{i,1}z^{v_i} + \cdots + a_{i,r_i-1}z^{(r_i-1)v_i} + z^{(r_iv_i} \right). \]
Because $\langle e_k, v_i \rangle = \{ e_i, e_k \} < 0$, we know that $\epsilon_{ik} = d_k\{e_i,e_k \} <0$. Therefore,
\begin{align*}
    e_i' &= e_i + r_k[\epsilon_{ik}]_+e_k = e_i \\
    v_i' &= p_1^*(e_i') = p_1^*(e_i) = v_i
\end{align*}
and so $\mathfrak{d}_{i,-}'$ is simply half of the wall \[ ((e_i')^\perp, 1 + a_{i,1}z^{v_i'} + \cdots + a_{i,r_i-1}z^{(r_i-1)v_i'} + z^{r_iv_i'} ) \in \mathfrak{D}_{\textrm{in},\bfs'}.\]

\vspace{2mm}
Hence, we see that after dividing some of the walls of $\mathfrak{D}_{\textrm{in},\bfs'}$ into two halves, the diagrams $T_k(\mathfrak{D}_{\textrm{in},\bfs})$ and $\mathfrak{D}_{\textrm{in},\bfs'}$ have the same set set of incoming walls. Therefore, up to the same halving of walls, the diagrams $T_k(\mathfrak{D}_\bfs)$ and $\mathfrak{D}_{\bfs'}$ also have the same set of incoming walls.
\end{proof}

\subsection{Chamber Structure}
\label{subsec:genChamberStructure}

Analogously to the ordinary case, the generalized $T_k$ map defined in Section~\ref{subsec:genMutInvariance} determines a chamber structure on the generalized cluster scattering diagram $\mathfrak{D}_{\bfs}$. Because this structure arises in the same manner as in the ordinary case, the exposition in this section will largely focus on stating properties and results that are necessary for proofs in later sections. For more details about the ordinary case, we refer the reader to either Section~\ref{subsec:ordMutInvar} or \cite{GHKK}.

Let $\mathfrak{T}_{v}$ be the infinite rooted tree defined in Section~\ref{subsec:genBasicDefinitions}, $\mathbf{s}_v$ be the generalized torus seed associated to vertex $v$, and $w \neq v$ be an arbitrary vertex in $\mathfrak{T}_{v}$. Then the sequence of edge labels $k_1, \dots, k_{\ell}$ on a simple path between $v$ and $w$ determine a map $T_{w} = T_{k_{\ell}} \circ \cdots \circ T_{k_1} : M_{\mathbb{R}} \rightarrow M_\mathbb{R}$ where each $T_{k_i}$ is defined with respect to the basis vector $e_{k_i}$ in the mutated generalized torus seed $\mu_{k_{i-1}} \circ \cdots \circ \mu_{k_1}(\bfs_v)$ rather than the original generalized torus seed $\bfs_v$. It follows from Theorem~\ref{theorem:mutDiagConsistent} that $T_{w}(\mathfrak{D}_{\bfs}) = \mathfrak{D}_{\bfs_{w}}$, where $\bfs_{w}$ denotes a generalized torus seed associated to the vertex $w$.

Let $\Sigma$ be a set of generalized fixed data that satisfies the injectivity assumption and $\bfs$ be a choice of associated generalized torus seed. By construction, $\mathfrak{D}_{\bfs}$ must include a collection of incoming walls with support $\{ e_k^{\perp} \}_{k \in I_{\textrm{uf}}}$. As in the ordinary case, we define
\begin{align*}
    \mathcal{C}_{\bfs}^{+} &:= \{ m \in M_{\mathbb{R}} : \langle e_i, m \rangle \geq 0 \textrm{ for all } i \in I_{\textrm{uf}} \}, \\
    \mathcal{C}_{\bfs}^{-} &:= \{ m \in M_{\mathbb{R}} : \langle e_i, m \rangle \leq 0 \textrm{ for all } i \in I_{\textrm{uf}} \},
\end{align*}
and refer to $\mathcal{C}_{\bfs}^{+}$ as the \emph{positive chamber}. The chambers $\mathcal{C}_{\bfs}^{\pm}$ are closures of connected components of $M_{\mathbb{R}} \backslash \textrm{Supp}(\mathfrak{D}_{\bfs})$. Let $\mathcal{C}_{\mu_{k}(\bfs)}^{\pm}$ denote the chambers where either all $\langle e_i', m \rangle \geq 0$ or $\langle e_i', m \rangle \leq 0$, respectively. Then $\mathcal{C}_{\mu_k(\bfs)}^{\pm}$ is similarly the closure of a connected component of $M_{\mathbb{R}} \backslash \textrm{Supp}(\mathfrak{D}_{\mu_k(\bfs)})$. Hence, $T_k^{-1}(\mathcal{C}_{\mu_{k}(\bfs)}^{\pm})$ is the closure of a connected component of $M_{\mathbb{R}} \backslash \textrm{Supp}(\mathfrak{D}_{\bfs})$ which shares a codimension one face, given by $e_k^{\perp}$, with $\mathcal{C}_{\bfs}^{\pm}$. This extends to generalized torus seeds related to $\bfs$ by longer mutation sequences. Let $w$ be a vertex of $\mathfrak{T}_{\bfs}$ that is reachable from the root vertex via a simple path of arbitrary length. It follows from Theorem~\ref{theorem:mutDiagConsistent} and the previous paragraph that $\mathcal{C}_{w}^{\pm} := T_w^{-1}\left( \mathcal{C}_{\bfs_w}^{\pm} \right)$ is a closure of a connected component of $M_{\mathbb{R}} \backslash \textrm{Supp}(\mathfrak{D}_{\bfs})$. It is important to note, however, that  the collection of cones $C_{v}^{\pm}$ will not always form a dense subset of $M_{\mathbb{R}}$.

Let $\mathcal{C}_{w}^{\pm}$ denote the chamber of $\textrm{Supp}(\mathfrak{D}_{\bfs})$ which corresponds to the vertex $w \in \mathfrak{T}_{\bfs}$ and $\Delta_{\bfs}^{\pm}$ denote the collection of chambers $\mathcal{C}^{\pm}_{w}$ as $w$ runs over the vertices of $\mathfrak{T}_{\bfs}$. As before, we refer to elements of $\Delta_{\bfs}^{+}$ as \emph{cluster chambers}. This chamber structure coincides with the natural generalization of the Fock-Goncharov cluster complex.

\begin{definition}[Analogue of Definition 2.14 of \cite{FG}]
Fix a set of generalized fixed data $\Sigma$ and an associated generalized torus seed $\bfs$. Then for a generalized torus seed ${\bfs' = \{ (e_{i}', (a_{i,j}')) \}}$ which is reachable via a mutation sequence from $\bfs$, the \emph{Fock-Goncharov cluster chamber associated to $\bfs'$} is the subset $\{ x \in \mathcal{A}^{\vee}(\mathbb{R}^T) : (z^{e_i'})^T(x) \leq 0 \textrm{ for all } i \in I_{\textrm{uf}} \},$
which is identified with the subset $ \{ x \in \mathcal{A}^{\vee}(\mathbb{R}^t) : (z^{e_i'})^t(x) \leq 0 \textrm{ for all } i \in I_{\textrm{uf}} \} \}$
via the canonical sign-change map $i : \mathcal{A}^{\vee}(\mathbb{R}^T) \rightarrow \mathcal{A}^{\vee}(\mathbb{R}^t)$.
\end{definition}

\begin{lemma}[Analogue of Lemma 2.10 of \cite{GHKK}]
\label{lemma:FG_clusterChamber}
Let $\Sigma$ be a set of generalized fixed data which satisfies the injectivity assumption and $\bfs$ be an accompanying choice of generalized torus seed. Let $\bfs' = \{ (e_{i}', (a_{i,j}')) \}$ be a distinct generalized torus seed which is reachable via some mutation sequence from $\bfs$. Then the positive chamber $\mathcal{C}_{\bfs'}^{+} \subset M^{\circ}_{\mathbb{R},\bfs'} = \mathcal{A}^{\vee}(\mathbb{R}^T)$ (which can be identified with $\mathcal{A}^{\vee}(\mathbb{R}^t)$ via the sign-change map $i$) is the Fock-Goncharov cluster chamber associated to $\bfs'$. Therefore, the Fock-Goncharov cluster chambers are the maximal cones of a simplicial fan and $\Delta^{+}$ is identified with $\Delta^{+}_{\bfs}$ for every choice of generalized torus seed $\bfs$ which gives an identification of $\mathcal{A}^{\vee}(\mathbb{R}^T)$ with $M_{\mathbb{R},\bfs}^{\circ}$.
\end{lemma}

\begin{proof}
The proof given by \cite{GHKK} holds in the generalized setting, because we showed in Proposition~\ref{prop:T_k_tropicalization} that our modified $T_k$ map is the Fock-Goncharov tropicalization of the generalized mutation map $\mu_{(v_k,d_ke_k,r_k)}$.
\end{proof}

It follows from the previous proposition, when the injectivity assumption holds, that:

\begin{theorem}[Analogue of Theorem 2.13 of \cite{GHKK}]
\label{theorem:GHKK_2-13}
For any set of initial data, the Fock-Goncharov cluster chambers in $\mathcal{A}^{\vee}(\mathbb{R}^T)$ are the maximal cones of a simplicial fan.
\end{theorem}

\subsection{Principal coefficients}
\label{subsec:genPrincipalCoef}

The injectivity assumption is not satisfied in most cases. 
We can still handle those cases when we include principal coefficients in the setting. We will now rephrase the construction for generalized cluster algebras with principal coefficients.

\begin{definition}
Given generalized fixed data $\Gamma$, the generalized fixed data $\Gamma_{\textrm{prin}}$ for the principal coefficient case is defined in the same way as for ordinary cluster algebras, with the additional requirement that $\widetilde{r} = (r,r)$, i.e. that $\widetilde{r}$ consists of two copies of $r$ with $\widetilde{r}_i = r_i$ for $i \in I$, and we now include the collection $\{ a_{i,j} \}_{i \in I_{\textrm{uf}},j \in [r_i -1]}$.
\end{definition}

\begin{definition}
Given a generalized torus seed $\bfs$, the generalized torus seed with principal coefficients is defined as
\[ \bfs_{\textrm{prin}} := \widetilde{\bfs} = \{ ((e_i,0),\mathbf{a}_i),((0,f_i),\mathbf{a}_i) \}_{i \in I} \]
\end{definition}

We can then use these updated definitions to define the cluster varieties with principal coefficients. Recall that we work over the ring $R = \Bbbk[a_{i,j}]$.

\begin{definition}
Given a generalized torus seed $\bfs$, we define the associated algebraic tori
\begin{align*}
    \mathcal{X}_{\bfs_{\textrm{prin}}} &:= T_{\widetilde{M}}(R) = \textrm{Spec}~\Bbbk[\widetilde{N}] \times_{\Bbbk} \textrm{Spec}(R), \\
    \mathcal{A}_{\bfs_{\textrm{prin}}} &:= T_{\widetilde{N}^{\circ}}(R) = \textrm{Spec}~\Bbbk[\widetilde{M}^{\circ}] \times_{\Bbbk} \textrm{Spec}(R).
\end{align*}
The generalized $\mathcal{A}$ cluster variety with principal coefficients and generalized $\mathcal{X}$ cluster variety with principal coefficients are then defined as in the ordinary case.
\end{definition}

As before, the generalized $\mathcal{A}$-variety is given by the fiber $\mathcal{A}_e$ and the generalized $\mathcal{X}$-variety is given by the quotient $\mathcal{A}_{\textrm{prin}}/T_{N^{\circ}}$.

\begin{prop}[Analogue of Proposition B.2 of \cite{GHKK}]
\label{prop:GHKK_B2}
Given a set of generalized fixed data $\Gamma$:
\begin{enumerate}
    \item There is the following commutative diagram, where the dotted arrows are only present if there are no frozen variables:
    \begin{center}
    \begin{tikzcd}
        \mathcal{A}_t \arrow[r] \arrow[d] & \Aprin \arrow[d, "\pi"] \arrow[r,"\widetilde{p}"] \arrow[rr, bend left, "p"] & \mathcal{X} \arrow[d, "\lambda"] & \mathcal{X}_{\textrm{prin}} \arrow[l,"\rho"] \arrow[d,"w"] & \mathcal{A} \arrow[l,"\xi"] \arrow[d] \\
        t \arrow[r] & T_{M} \arrow[r, dashed] & T_{K^*} & T_{M} \arrow[l,dashed] & e \arrow[l]
    \end{tikzcd}
    \end{center}
    where $t$ is any point in $T_M$, $e \in T_{M}$ is the identity, and $p$ is an isomorphism which is canonical if there are no frozen variables.
    \item There is a torus action of $T_{N^\circ}$ on $\Aprin$; $T_{K^{\circ}}$ on $\mathcal{A}$; $T_{N_{\textrm{uf}}^{\perp}}$ on $\mathcal{X}$; and $T_{\widetilde{K}^{\circ}}$ on $\Aprin$, where $\widetilde{K}^{\circ}$ is the kernel of the map $N^{\circ} \oplus M \rightarrow N_{\textrm{uf}}^{*}$ given by $(n,m) \mapsto p_2^*(n) - m$. The action of $T_{N^{\circ}}$ and $T_{\widetilde{K}^{\circ}}$ on $T_{M}$ is such that the map $\pi: \Aprin \rightarrow T_{M}$ is $T_{N^{\circ}}$-equivariant and $T_{\widetilde{K}^{\circ}}$-equivariant. The map $\widetilde{p}: \Aprin \rightarrow \mathcal{X} = \Aprin / T_{N^{\circ}}$ is a $T_{N^{\circ}}$-torsor. Furthermore, there is a map $T_{\widetilde{K}^{\circ}} \rightarrow T_{N_{\textrm{uf}}^{\perp}}$ such that $\widetilde{p}$ is compatible with the actions of these tori on, respectively, $\Aprin$ and $\mathcal{X}$. Hence, $\tau: \Aprin \rightarrow \mathcal{X}/T_{N_{\textrm{uf}}^{\perp}}$ is a $T_{\widetilde{K}^{\circ}}$-torsor.
\end{enumerate}
\end{prop}

\begin{proof}
Because the map definitions remain the same, the proof given in \cite{GHKK} holds in the generalized setting.
\end{proof}

\begin{remark} \label{rk:however}
In Section \ref{subsec:genAprintoAX}, we will define theta functions in both the $\cA$ and $\cX$ cases. Here, we briefly preview the content of Lemma~\ref{lemma:GHKK_7-10}. 
From the discussion above, $\Aprin$ is a $T_{N^{\circ}}$-torsor over $\cX$.
Hence, the global functions on $\cX$ are given by $T_{N^{\circ}}$-invariant global monomials $\vartheta_m$ on $\Aprin$, where $m $ lies in the cluster complex of $\Aprin$ and in the slice $w^{-1}(0)$ for $w: (m,n) \mapsto m-p^*(n)$. 
This leads to the slicing that we will illustrate in Example~\ref{ex:genG2_principalCoeff}. We can further understand the discussion in terms of mutation maps. In the $\Aprin$ case, the $\cA$ mutation in direction $k$ given in Equation~\eqref{eq:A_pullback_gen} is
\begin{align*}
        \mu_k^*z^{(m,n)} &= z^{(m,n)} \left( 1 + a_{k,1}z^{(v_k, e_k)} + \cdots + a_{k,r_k-1}z^{(r_k-1) \cdot (v_k, e_k)} + z^{r_k(v_k, e_k)} \right)^{-\langle (d_ke_k, 0),(m,n) \rangle} ,
\end{align*}
for $(m,n) \in M^{\circ} \oplus N$. 
Having $(m,n) \in w^{-1}(0)$ implies that $(m,n) = (p^*(n),n)$. This yields the mutation transformation
\begin{align*}
        \mu_k^*z^{(p^*(n),n)} &= z^{(p^*(n),n)} \left( 1 + a_{k,1}z^{(v_k, e_k)} + \cdots + a_{k,r_k-1}z^{(r_k-1) \cdot (v_k, e_k)} + z^{r_k(v_k, e_k)} \right)^{-\langle (d_ke_k, 0),(p^*(n),n) \rangle}. 
\end{align*}
By the change of variable $z^{(p^*(n),n)} \mapsto z^n$, we obtain
\begin{align*}
        \mu_k^*z^{n} &= z^{n} \left( 1 + a_{k,1}z^{ e_k} + \cdots + a_{k,r_k-1}z^{(r_k-1)  e_k} + z^{r_k e_k} \right)^{-\langle d_ke_k, p^*(n) \rangle}
        \\
    &= z^n \left( 1 + a_{k,1}z^{e_k} + \cdots + a_{k,r_k-1}z^{(r_k-1)e_k} + z^{r_ke_k} \right)^{-\{n,d_ke_k \}} \\
    &= z^n \left( 1 + a_{k,1}z^{e_k} + \cdots + a_{k,r_k-1}z^{(r_k-1)e_k} + z^{r_ke_k} \right)^{- [n,e_k]},
\end{align*}
which is precisely the $\cX$ mutation from Equation~\eqref{eq:X_pullback_gen}.
We further emphasize that the  change of variables $z^{(p_1^*(n), n)} \mapsto z^n$
exactly matches the treatment in Fomin-Zelevinsky \cite{FZ-IV} of cluster algebras with principal coefficients, see Equation~\eqref{eq:changeofvar}.
\end{remark}

Next, we update the definition of the initial generalized cluster scattering diagram.

\begin{definition} Given a generalized seed $\bfs$, let $\widetilde{v_i} := (v_i,e_i) = (p_1^*(e_i),e_i)$. The corresponding initial $\Aprin$ scattering diagram,  according to Definition \ref{def:initialScatDiagram}, is of the form 
\[ \mathfrak{D}_{\textrm{in},\bfs}^{\mathcal{A}_{\textrm{prin}}} = \left\{ \left( (e_i,0)^\perp, 1 + a_{i,1}z^{\widetilde{v_1}} + \cdots + a_{i,r_i - 1}z^{(r_i-1)\widetilde{v_i}} + z^{r_i\widetilde{v_i}} \right) \right\} \]
\end{definition}

The following example examines the scattering diagram $\mathfrak{D}_{\bfs}^{\mathcal{A}_{\textrm{prin}}}$ for a generalized cluster algebra with principal coefficients.
Beginning with this scattering diagram $\mathfrak{D}_{\bfs}^{\mathcal{A}_{\textrm{prin}}}$ for the case with principal coefficients, we outline how to obtain the scattering diagrams for the $\cA$ and the $\cX$ cases. In this particular 2-dimensional example, the $\cA$ scattering diagram is well-defined; however, this is not true in general as the injectivity assumption fails.  Formally, we can only define the $\cA$ and $\cX$ theta functions in scattering diagrams with principal coefficients, as discussed in Section 
\ref{subsec:genAprintoAX}. 
One can then construct $\cA$ and $\cX$ diagrams which give us the $\cA$ and $\cX$ theta functions. The following example illustrates this idea. 

\begin{example}
\label{ex:genG2_principalCoeff}
Consider the generalized cluster algebra with  $B= \begin{bmatrix} 0 & -1 \\ 1 & 0 \end{bmatrix}$, $d = (1,1)$, and $[r_{ij}] = \begin{bmatrix} 3 & 0 \\ 0 & 1 \end{bmatrix}$, as in Example~\ref{ex:genFixedData}.
 The generalized fixed data $\Gamma_{\textrm{prin}}$ has index set $\widetilde{I} = I \sqcup I$, $\widetilde{I}_{\textrm{uf}} = \{ 1, 2 \}$, $\widetilde{d} = (1,1,1,1)$, $\widetilde{r} = (3,1,3,1)$ and lattices $ \widetilde{N} = N \oplus M^{\circ}$, $\widetilde{N}^{\circ} = N^{\circ} \oplus M$, $\widetilde{M} = M \oplus N^{\circ}$, and $\widetilde{M}^{\circ} = M^{\circ} \oplus N$, all of which have basis $\{ (1,0,0,0),(0,1,0,0),(0,0,1,0),(0,0,0,1) \}$. By definition, we have
 \[ \mathfrak{D}^{\mathcal{A}_{\textrm{prin}}}_{\textrm{in},\bfs} = \left\{
 \begin{array}{c} \widetilde{\mathfrak{d}}_1 = \left((0,1,0,0)^{\perp}, 1 + z^{(-1,0,0,1)} \right), \\ \widetilde{\mathfrak{d}}_2 = \left((1,0,0,0)^{\perp}, 1 + az^{(0,1,1,0)} + az^{(0,2,2,0)} + z^{(0,3,3,0)}\right)
 \end{array}
 \right\} \]
which can be completed, using the definition of consistency, to give a generalized cluster scattering diagram for $\mathcal{A}_{\textrm{prin}}$, denoted $\mathfrak{D}^{\mathcal{A}_{\textrm{prin}}}_{\bfs}$. Because this diagram is four dimensional, the figure below shows the projection of $\mathcal{D}_{\bfs}^{\mathcal{A}_{\textrm{prin}}}$ onto $M^{\circ}$ with labels that indicate the corresponding walls and wall-crossing automorphisms in $\mathcal{A}_{\textrm{prin}}$. To illustrate this projection, consider the wall $\widetilde{\mathfrak{d}}_1 \in \mathfrak{D}_{\mathbf{s}}^{\mathcal{A}_{\textrm{prin}}}$ with support $(0,1,0,0)^{\perp} \subset \mathbb{R}^4$, i.e. the three-dimensional hyperplane $\textrm{span} \{ (1,0,0,0), (0,0,1,0), (0,0,0,1) \}$. When we project from $\widetilde{M}^{\circ}_{\mathbb{R}}$ onto $M^{\circ}_{\mathbb{R}}$, the wall $\widetilde{\mathfrak{d}}_1$ is projected onto $\mathbb{R}(1,0) \subset \mathbb{R}^2$, i.e. the one-dimensional hyperplane $\textrm{span} \{ (1,0) \}$. Similarly, $\widetilde{\mathfrak{d}}_2$ is projected onto $\mathbb{R}(0,1)$, i.e. the one-dimensional hyperplane $\textrm{span} \{ (0,1) \}$.
\begin{center}
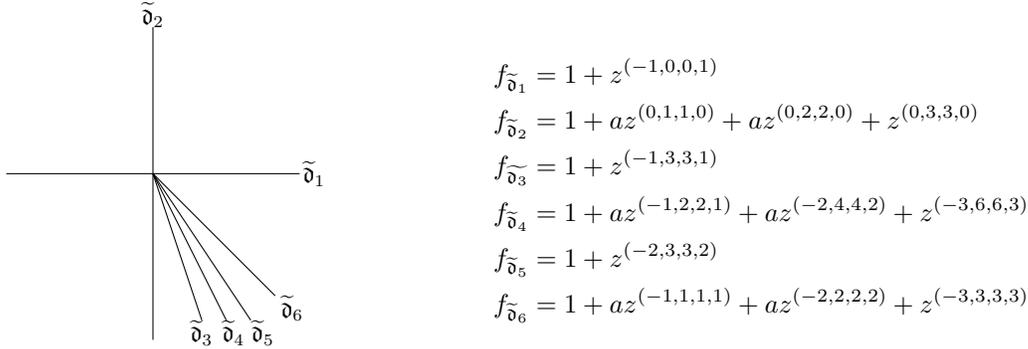

    \begin{minipage}{0.2\textwidth}
 	\begin{tikzpicture}[scale=0.65]
		\draw (-3,0) to (3,0);
		\draw (0,-3.4) to (0,3);
		\draw (0,0) to (2.5,-2.5);
		\draw (0,0) to (2,-3);
		\draw (0,0) to (1.5,-3);
		\draw (0,0) to (1,-3);
				
		\node[scale=0.9] at (3.3,0) {$\widetilde{\mathfrak{d}}_1$};
		\node[scale=0.9] at (0,3.25) {$\widetilde{\mathfrak{d}}_2$};
		\node[scale=0.9] at (2.85,-2.75) {$\widetilde{\mathfrak{d}}_6$};
		\node[scale=0.9] at (2.25,-3.25) {$\widetilde{\mathfrak{d}}_5$};
		\node[scale=0.9] at (1.65,-3.25) {$\widetilde{\mathfrak{d}}_4$};
		\node[scale=0.9] at (1,-3.25) {$\widetilde{\mathfrak{d}}_3$};
	\end{tikzpicture} 
	\end{minipage}
	\qquad\qquad\qquad
	\begin{minipage}{0.54\textwidth}
	\begin{align*}
        f_{\widetilde{\mathfrak{d}}_1} &= 1+z^{(-1,0,0,1)} \\
        f_{\widetilde{\mathfrak{d}}_2} &= 1+az^{(0,1,1,0)}+az^{(0,2,2,0)}+z^{(0,3,3,0)} \\
        f_{\widetilde{\mathfrak{d}_3}} &= 1+z^{(-1,3,3,1)} \\
        f_{\widetilde{\mathfrak{d}}_4} &= 1 + az^{(-1,2,2,1)}+az^{(-2,4,4,2)}+z^{(-3,6,6,3)} \\
        f_{\widetilde{\mathfrak{d}}_5} &= 1 + z^{(-2,3,3,2)} \\
        f_{\widetilde{\mathfrak{d}}_6} &= 1 + az^{(-1,1,1,1)}+az^{(-2,2,2,2)}+z^{(-3,3,3,3)}
    \end{align*}
	\end{minipage}
		\captionof{figure}{2-dimensional projection of the scattering diagram $\mathfrak{D}^{\mathcal{A}_{\textrm{prin}}}_{\bfs}$}
   \label{fig:egdiag}
\end{center}

The $\cA$ scattering diagram can be obtained from the projection $M^{\circ} \oplus N \rightarrow M$, $(m,n) \mapsto m$. 
The walls of the $\cA $ diagram will be the same as in Figure \ref{fig:egdiag} while the wall functions are of the form $1+z^m$ rather than $1+z^{(m,n)}$ listed in Figure \ref{fig:egdiag}. 

Next, we can obtain the scattering diagram for $\mathcal{X}_{\bfs}$ from $\mathfrak{D}^{\mathcal{A}_{\textrm{prin}}}_{\bfs}$. 
A full discussion of this construction is laid out in  \cite[Section 2.2.1]{C-M-NC}. 
As in the earlier remark, the idea is to take the slicing
\[ w^{-1}(0)=
\{ m \in M^{\circ} : m = p^*(n), n \in N \}
\]
of $\mathfrak{D}_{\bfs}^{\mathcal{A}_{\textrm{prin}}}$. As we later point out in Section \ref{subsec:genAprintoAX},  restrictions of path-ordered products to $w^{-1}(0)$ are well defined for paths $\gamma$ that lie entirely in the slice. For example, consider the wall $\widetilde{\mathfrak{d}}_3$ and a path $\widetilde{\gamma}$ in $\mathfrak{D}_{\bfs}^{\mathcal{A}_{\textrm{prin}}}$ as demonstrated in the projection below:
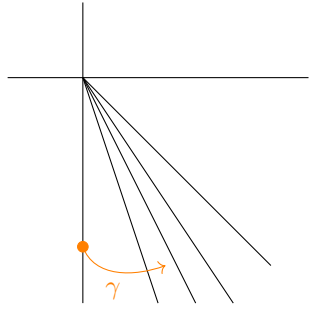
\begin{center}
    \begin{tikzpicture}
		\draw (-1,0) to (3,0);
		\draw (0,-3) to (0,1);
		\draw (0,0) to (2.5,-2.5);
		\draw (0,0) to (2,-3);
		\draw (0,0) to (1.5,-3);
		\draw (0,0) to (1,-3);
        
        \filldraw[orange] (0,-2.25) circle (2pt);
        \draw[orange,out=-70,in=-160,->] (0,-2.25) to (1.1,-2.5);
        \draw[orange] (0.4,-2.6) node[below]{$\gamma$};
    \end{tikzpicture}
\end{center}
By definition, for $(0,-1,-1,0)$ lies in the slice, we have
\begin{align*}
    \mathfrak{p}_{\widetilde{\gamma}}\left( z^{(0,-1,-1,0)} \right) &= z^{(0,-1,-1,0)}\left(1 + z^{(-1,3,3,1)} \right)^{\langle (0,-1,-1,0),(-3,-1,0,0) \rangle} \\
    &= z^{(0,-1,-1,0)}\left(1+z^{(-1,3,3,1)} \right).
\end{align*}
By change of variables $z^{(p^*(n), n)} \mapsto z^n$, 
the path-ordered product becomes $\mathfrak{p}_{\gamma}(z^{(-1,0)}) = z^{(0,-1)}(1+z^{(3,1)})$ 
, from which we can read off the wall-crossing automorphism for $\mathfrak{d}_3$ as $f_{\mathfrak{d}_3} = 1 + z^{(3,1)}$. Similar computations for the remaining walls allow us to draw $\mathfrak{D}_{\bfs}^{\mathcal{X}}$, as follows.
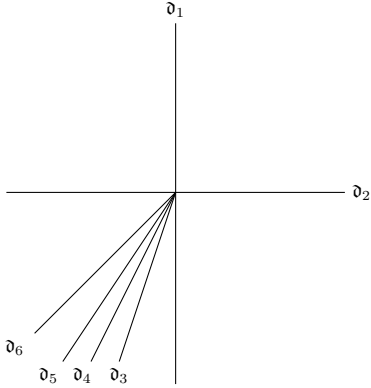
\begin{center}
    	    \begin{minipage}{0.3\textwidth}
	\begin{tikzpicture}[scale=0.75]
		\draw (-3,0) to (3,0);
		\draw (0,-3.4) to (0,3);
		\draw (0,0) to (-2.5,-2.5);
		\draw (0,0) to (-2,-3);
		\draw (0,0) to (-1.5,-3);
		\draw (0,0) to (-1,-3);
				
		\node[scale=0.75] at (3.3,0) {$\mathfrak{d}_2$};
		\node[scale=0.75] at (0,3.25) {$\mathfrak{d}_1$};
		\node[scale=0.75] at (-2.85,-2.75) {$\mathfrak{d}_6$};
		\node[scale=0.75] at (-2.25,-3.25) {$\mathfrak{d}_5$};
		\node[scale=0.75] at (-1.65,-3.25) {$\mathfrak{d}_4$};
		\node[scale=0.75] at (-1,-3.25) {$\mathfrak{d}_3$};
		
	\end{tikzpicture}
	\end{minipage}
	\qquad\qquad
	\begin{minipage}{0.44\textwidth}
	\begin{align*}
        f_{\mathfrak{d}_1} &= 1+z^{(0,1)} \\
        f_{\mathfrak{d}_2} &= 1+az^{(1,0)}+az^{(2,0)}+z^{(3,0)} \\
        f_{\mathfrak{d}_3} &= 1+z^{(3,1)} \\
        f_{\mathfrak{d}_4} &= 1 + az^{(2,1)}+az^{(4,2)}+z^{(6,3)} \\
        f_{\mathfrak{d}_5} &= 1 + z^{(3,2)} \\
        f_{\mathfrak{d}_6} &= 1 + az^{(1,1)}+az^{(2,2)}+z^{(3,3)}
    \end{align*}
	\end{minipage}
\end{center}
Alternately, note that the wall $\widetilde{\mathfrak{d}}_k$ in $\mathfrak{D}_{\bfs}^{\mathcal{A}_{\textrm{prin}}}$ has support containing either $\mathbb{R} \cdot (p^*(n),n)$ or $\mathbb{R}_{\geq 0} \cdot (p^*(n),n)$.
Hence, in this 2-dimensional setting the wall  $\mathfrak{d}_k$ in $\mathfrak{D}^{\mathcal{X}_{\bfs}}$ has support given by either $\mathbb{R} \cdot n$ or $\mathbb{R}_{\geq 0} \cdot n$ for $n \in N$.

\end{example}

\subsection{Building $\Ascat$ from a generalized cluster scattering diagram}
\label{subsec:genAScat}

In this section, we parallel the exposition in Section 4 of \cite{GHKK}, which describes how to build the space $\Ascat$ from an ordinary scattering diagram and then identifies $\Ascat$ with the $\A$ variety. We review relevant portions of their constructions and statements, pointing out where modifications are needed  to extend the results to generalized cluster algebras with reciprocal coefficients. 

Let $\Gamma$ be a set of generalized initial data such that the diagram $\mathfrak{D}_\bfs$ yields a
cluster chamber structure $\Delta_\bfs^+$.  We will often want to discuss multiple copies of the lattices $N, M, N^\circ$, and $M^\circ$ which arise from different choices of seed $\bfs$. To allow us to distinguish between these copies, we index both the seeds and lattices by either the vertices $v$ of $\mathfrak{T}_v$ or chambers $\sigma$ of $\Delta_\bfs^+$.  For example, the seed $\bfs_v$ gives rise to the diagram $\mathfrak{D}_{\bfs_v}$ on the lattice $M^\circ_{\mathbb{R},\bfs_v}$. The chambers in $\mathfrak{D}_{\bfs_v}$ give the Fock-Goncharov cluster complex $\Delta^+$ under the identification $M^\circ_{\mathbb{R},\bfs_v} = \mathcal{A}^\vee(\mathbb{R}^T)$. Because the space $\mathcal{A}^\vee(\mathbb{R}^T)$ is independent of the choice of the initial seed $\bfs$, there is a canonical bijection between the cluster chambers of $\mathfrak{D}_{\bfs_v}$ and $\mathfrak{D}_{\bfs_{v'}}$ as a consequence of this identification.

\begin{definition} \cite[Construction 4.1]{GHKK}
\label{def:Ascat_atlas}
Given a seed $\bfs$, we want to construct a space, $\mathcal{A}_{\textrm{scat},\bfs}$ using the chambers $\sigma \in \Delta_\bfs^+$. For  distinct $\sigma, \sigma' \in \Delta_\bfs^+$, there exists a path $\gamma$ from $\sigma'$ to $\sigma$. This path gives rise to an automorphism $\mathfrak{p}_{\gamma,\mathfrak{D}_\bfs}: \widehat{R[P]} \rightarrow \widehat{R[P]}$ which depends only on the choice of $\sigma$ and $\sigma'$ and is independent of choice of path. 

For each chamber in $\Delta_\bfs^+$, attach a copy of the torus $T_{N^\circ,\sigma} := T_{N^\circ}$. If $\gamma$ is chosen such that it lies in the support of the cluster complex, then the wall-crossing automorphisms attached to walls crossed by $\gamma$ are birational maps of $T_{N^\circ}$. Therefore the path-ordered product $\mathfrak{p}_{\gamma,\mathfrak{D}_\bfs}$ can be viewed as a well-defined map of fields of fractions $\mathfrak{p}_{\gamma,\mathfrak{D}_\bfs}: R(M^\circ) \rightarrow R(M^\circ)$ which induces a positive birational map $\mathfrak{p}_{\sigma,\sigma'}:T_{N^\circ,\sigma} \rightarrow T_{N^\circ,\sigma'}$.

The space $\mathcal{A}_{\textrm{scat},\bfs}$ is constructed by gluing the collection of tori $\{ T_{N^\circ,\sigma} \}_{\sigma \in \Delta_\bfs^+}$ using the birational maps $\{ \mathfrak{p}_{\sigma,\sigma'} \}_{\sigma,\sigma' \in \Delta_{\bfs}^+}$ according to the method described in Proposition 2.4 of \cite{GHK}.
\end{definition}

\begin{prop}[Analog of Proposition 4.3 of \cite{GHKK}]
\label{prop:A_scatIso}
Let $\bfs$ be a seed, $v$ be the root of $\mathfrak{T}_\bfs$, and $v'$ be any other vertex of $\mathfrak{T}_\bfs$. Let $\mu_{v',v}^T : M^\circ_{v'} \rightarrow M_v^\circ$ be the Fock-Goncharov tropicalization of $\mu_{v',v}:T_{M^\circ,v'} \rightarrow T_{M^\circ,v}$. The restriction $\left.\mu_{v',v}^T\right|_{\sigma'} : M^\circ_{\sigma'} \rightarrow M^\circ_{\sigma}$ to each cluster chamber $\sigma'$ of $\Delta^+_{\bfs_{v'}}$ gives a linear isomorphism from $\sigma'$ to the corresponding cluster chamber $\sigma := \mu^T_{v',v}(\sigma')$ in $\Delta^+_\bfs$ and induces an isomorphism
\[ T_{v',\sigma} : T_{N^\circ,\sigma} \rightarrow T_{N^\circ,\sigma'}. \]
When $\sigma$ ranges across all the cluster chambers of $\Delta^+_{\bfs_v}$, the isomorphisms $T_{v',\sigma}$ glue together to yield an isomorphism between $\mathcal{A}_{\textrm{scat},\bfs_v}$ and $\mathcal{A}_{\textrm{scat},\bfs_{v'}}$.
\end{prop}

\begin{proof}
We follow the structure of the proof in \cite{GHKK}. In general, $v$ and $v'$ are related by a composition of mutations and $\mu_{v',v}$ is the inverse of that composition. Proving this statement for arbitrary $v$ and $v'$, then, essentially consists of proving the statement for the special case where $v$ and $v'$ are related by a single mutation. In this case, $\mu_{v',v}^T = T_k^{-1}$ and the isomorphism $T_{v',\sigma} : T_{N^\circ,\sigma} \rightarrow T_{N^\circ,T_k(\sigma)}$ is induced by the restriction $\left.T_k^{-1}\right|_{T_k(\sigma)}$. 

To show that gluing together these isomorphisms for all $\sigma \in \Delta_{\bfs_v}^+$ gives an isomorphism between $\mathcal{A}_{\textrm{scat},\bfs_v}$ and  $\mathcal{A}_{\textrm{scat},\bfs_{v'}}$, we need to show commutativity of the diagram

\begin{center}
\begin{tikzcd}
    T_{N^\circ,\sigma} \arrow[r,"T_{v',\sigma}"] \arrow[d,"\mathfrak{p}_{\sigma,\tilde{\sigma}}"']& T_{N^\circ,\sigma'} \arrow[d,"\mathfrak{p}_{\sigma',\tilde{\sigma}'}"] \\
    T_{N^\circ,\tilde{\sigma}} \arrow[r,"T_{v',\tilde{\sigma}}"']& T_{N^\circ,\tilde{\sigma}'}
\end{tikzcd}
\end{center}
where $\sigma$ and $\tilde{\sigma}$ are chambers in $\Delta_{\bfs_v}^+$ and $\sigma' = T_k(\sigma)$ and $\tilde{\sigma}' = T_k(\tilde{\sigma})$ are chambers in $\Delta_{\bfs_{v'}}$. Note that the map $\mathfrak{p}_{\sigma,\tilde{\sigma}}$ indicates a wall-crossing in $\mathfrak{D}_\bfs$ and $\mathfrak{p}_{\sigma',\tilde{\sigma}'}$ indicates a wall crossing in $\mathfrak{D}_{\bfs'}$.

If $\sigma$ and $\tilde{\sigma}$ both fall in $\mathcal{H}_{k,-}$, then commutativity is immediate because $T_k$ fixes the wall-crossing automorphism on the wall between $\sigma$ and $\tilde{\sigma}$. If both chambers fall in $\mathcal{H}_{k,+}$, then commutativity follows from Theorem \ref{theorem:mutDiagConsistent} because by definition the path-ordered product between a given pair of points is equal in equivalent diagrams. Hence, the important case to consider is when $\sigma$ and $\tilde{\sigma}$ are on opposite sides of the wall with support $e_k^\perp$.

Without loss of generality, we can assume that $\sigma$ is the chamber in $\mathcal{H}_{k,+}$, where $e_k$ is non-negative. We know that the only wall in $\mathfrak{D}_\bfs$ contained in $e_k^\perp$ is the slab $\mathfrak{d}_k = (e_k^\perp, 1 + a_{k,1}z^{v_k} + \cdots + a_{k,r_k-1}z^{(r_k-1)v_k} + z^{r_kv_k})$. In $\mathfrak{D}_{\bfs'}$, the slab contained in $e_k^\perp$ is  ${\mathfrak{d}_k' = (e_k^\perp, 1 + a_{k,1}z^{-v_k} + \cdots + a_{k,r_k-1}z^{-(r_k-1)v_k} + z^{-r_kv_k})}$. As such, the only way for $\sigma$ and $\tilde{\sigma}$ to be on opposite sides of $e_k^\perp$ is for the wall between them to be $\mathfrak{d}_k$ in $\mathfrak{D}_{\bfs}$ and $\mathfrak{d}_k'$ in $\mathfrak{D}_{\bfs'}$. Pictorially, we can envision:
\begin{center}
    \begin{tikzpicture}
    \draw (0,0) to (3,0);
    \node at (5,0) {$1 + a_{k,1}z^{v_k} + \cdots + z^{r_kv_k}$};

    \filldraw [blue] (1.5,-0.75) circle (1pt);
    \filldraw [blue] (1.5,0.75) circle (1pt);
    \draw[blue,out=120,in=-120,->] (1.5,-0.75) to (1.5,0.75);
    \node[blue] at (1.9,0.3) {$\mathfrak{p}_{\sigma,\tilde{\sigma}}$};
    
    \node at (0,0.5) {$\tilde{\sigma} \in \mathcal{H}_{k,-}$};
    \node at (0,-0.5) {$\sigma \in \mathcal{H}_{k,+}$};
    \end{tikzpicture}
\end{center}
in $\mathfrak{D}_\bfs$ and similarly
\begin{center}
    \begin{tikzpicture}
    \draw (0,0) to (3,0);
    \node at (5,0) {$1 + a_{k,1}z^{-v_k} + \cdots + z^{-r_kv_k}$};

    \filldraw [blue] (1.5,-0.75) circle (1pt);
    \filldraw [blue] (1.5,0.75) circle (1pt);
    \draw[blue,out=120,in=-120,->] (1.5,-0.75) to (1.5,0.75);
    \node[blue] at (1.9,0.3) {$\mathfrak{p}_{\sigma',\tilde{\sigma}}'$};
    
    \node at (0,0.5) {$\tilde{\sigma}' \in \mathcal{H}_{k,+}$};
    \node at (0,-0.5) {$\sigma' \in \mathcal{H}_{k,-}$};
    \end{tikzpicture}
\end{center}
in $\mathfrak{D}_{\bfs'}$. We can then compute that
\begin{align*}
    T^*_{v',\sigma}\left(\mathfrak{p}^*_{\sigma',\tilde{\sigma}'}(z^m) \right) &= T_{v',\sigma}^*\left( z^m (1 + a_{k,1}z^{-v_k} + \cdots + a_{k,r_k-1}z^{-(r_k-1)v_k} + z^{-r_kv_k})^{-\langle d_ke_k,m\rangle} \right) \\
    &= T_{v',\sigma}^*\left(z^m \right) T_{v',\sigma}^* \left( (1 + a_{k,1}z^{-v_k} + \cdots + a_{k,r_k-1}z^{-(r_k-1)v_k} + z^{-r_kv_k})^{-\langle d_ke_k,m\rangle}\right) \\
    &= \mu_{v',v}^T\left(z^m \right) \mu_{v',v}^T \left( (1 + a_{k,1}z^{-v_k} + \cdots + a_{k,r_k-1}z^{-(r_k-1)v_k} + z^{-r_kv_k})^{-\langle d_ke_k,m\rangle}\right) \\
    &= T_k^{-1}\left(z^m \right) T_k^{-1} \left( (1 + a_{k,1}z^{-v_k} + \cdots + a_{k,r_k-1}z^{-(r_k-1)v_k} + z^{-r_kv_k})^{-\langle d_ke_k,m\rangle}\right) \\
    &= z^{m - r_kv_k \langle d_ke_k, m \rangle} (1 + a_{k,1}z^{-v_k} + \cdots + a_{k,r_k-1}z^{-(r_k-1)v_k} + z^{-r_kv_k})^{-\langle d_ke_k,m\rangle} \\
    &= z^m \left( z^{-r_kv_k} \right)^{-\langle d_ke_k,m \rangle} (1 + a_{k,1}z^{-v_k} + \cdots + a_{k,r_k-1}z^{-(r_k-1)v_k} + z^{-r_kv_k})^{-\langle d_ke_k,m\rangle} \\
    &= z^m (z^{r_kv_k} + a_{k,1}z^{(r_k-1)v_k} + \cdots + a_{k,r_k-1}z^{v_k} + 1)^{-\langle d_ke_k,m\rangle} \\
    &= z^m \left(1 + a_{k,1}z^{v_k} + \cdots + a_{k,r_k-1}z^{(r_k-1)v_k} + z^{r_kv_k} \right)^{-\langle d_ke_k,m \rangle} \\
    &= \mathfrak{p}^*_{\sigma,\tilde{\sigma}} \left( z^m \right)
\end{align*}
Because $\mathcal{H}_{k,+}$ and $\mathcal{H}_{k,-}$ are reversed in $\mathfrak{D}_\bfs$ and $\mathfrak{D}_{\bfs'}$, our assumption that ${m \in \mathcal{H}_{k,+}}$ in $\mathfrak{D}_{\bfs}$ means that $m \in \mathcal{H}_{k,-}$ in $\mathfrak{D}_{\bfs'}$. As such, $z^m = T^*_{v,\tilde{\sigma}}(z^m)$ and we have $T^*_{v',\sigma}\left(\mathfrak{p}^*_{\sigma',\tilde{\sigma}'}(z^m) \right) = \mathfrak{p}^*_{\sigma,\tilde{\sigma}} \left( T^*_{v,\tilde{\sigma}}(z^m) \right)$ and the desired commutativity holds. Note that this computation relies on the reciprocity condition $a_{k,i} = a_{k,r_i - i}$.
\end{proof}

\begin{theorem}[Analogue of Theorem 4.4 of \cite{GHKK}]
\label{theorem:AscatCommuteMutation}
For a given generalized torus seed $\bfs$, let $v$ denote the root of $\mathfrak{T}_\bfs$ and $v'$ be 
 another arbitrary
vertex in $\mathfrak{T}_\bfs$. Let $\phi_{v,v'}^*: M_{v'}^\circ \rightarrow M_{v'}^\circ$ be the linear map $\left.\mu_{v,v'}^T\right|_{\mathcal{C}^+_{v' \in \bfs}}$ and $\phi_{v,v'}:T_{N^\circ,v'} \rightarrow T_{N^\circ,v'}$ be the map between the associated tori. Then the collection $\{ \phi_{v,v'} \}_{v'}$ glue to give an isomorphism
\[ \mathcal{A}_\bfs := \bigcup_{v'} T_{N^\circ} \rightarrow \mathcal{A}_{\textrm{scat},\bfs} := \bigcup_{v'} T_{N^\circ,v'}\]
and the diagram
\begin{center}
    \begin{tikzcd}
        \mathcal{A}_\bfs \arrow[r] \arrow[d] & \mathcal{A}_{\textrm{scat},\bfs} \arrow[d] \\
        \mathcal{A}_{\bfs_{v'}} \arrow[r] & \mathcal{A}_{\textrm{scat}, \bfs_{v'}}
    \end{tikzcd}
\end{center}
commutes, where the horizontal maps are the isomorphisms that were just defined, the right-hand vertical map is the isomorphism described in Proposition \ref{prop:A_scatIso} and the left-hand map is the natural open immersion $\mathcal{A}_{\bfs} \hookrightarrow \mathcal{A}_{\bfs_{v'}}$.
\end{theorem}

\begin{proof}
The proof of this theorem is identical to that of Theorem 4.4 from \cite{GHKK}; previous propositions check that despite the differences in wall-crossing automorphisms, we have the necessary commutativity of diagrams.
\end{proof}

This allows us to identify the rings of regular functions on $\mathcal{A}_{\textrm{scat}}$ and $\mathcal{A}_{\bfs}$.

\begin{definition}\cite[Definition 4.8]{GHKK}
Let $\Gamma$ be a generalized fixed data and $\bfs$ be an associated initial generalized torus seed. Let $\bfs_w = (e_1',\dots,e_n')$ be a generalized torus seed with dual basis $\{ (e_i')^* \}_i$ and $f_i' = d_i^{-1}(e_i')^*$. A \emph{cluster monomial} in $\bfs_w$ is a monomial in $T_{N^\circ,w} \subset \mathcal{A}$ of the form $z^m$ with $m = \sum_{i=1}^n a_if_i'$ for $a_i \in \mathbb{Z}_{\geq 0}$. We refer to any regular function which is a cluster monomial in some seed of $\mathcal{A}$ as a \emph{cluster monomial on $\mathcal{A}$}.
\end{definition}

\section{Broken lines and theta functions}
\label{subsec:genBrokenLines}
Broken lines have a similar  definition in the generalized setting as in the ordinary setting, with the caveat that we now work over the ground ring $R = \Bbbk[a_{i,j}]$ and so the monomials attached to the domains of linearity of a broken line lie in $R[M^{\circ}]$ rather than $\Bbbk[M^{\circ}]$.

\begin{definition}
Let $\mathfrak{D}$ be a generalized cluster scattering diagram, $m_0$ be a point in $M^\circ \backslash \{ 0 \}$, and $Q$ be a point in $M_\mathbb{R} \backslash \textrm{Supp}(\mathfrak{D})$. A \emph{broken line} with endpoint $Q$ and initial slope $m_0$ is a piecewise linear path $\gamma: (-\infty,0] \rightarrow M_\mathbb{R} \backslash \textrm{Sing}(\mathfrak{D})$ with finitely many domains of linearity. Each domain of linearity, $L$, has an associated monomial $c_Lz^{m_L} \in R[M^\circ]$ such that the following conditions are satisfied:
\begin{enumerate}
    \item $\gamma(0) = Q$
    \item If $L$ is the first domain of linearity of $\gamma$, then $c_Lz^{m_L} = z^{m_0}$.
    \item Within the domain of linearity $L$, the broken line has slope $-m_L$. In other words, $\gamma'(t) = -m_L$ on $L$. 
    \item Let $t$ be a point at which $\gamma$ is non-linear and is passing from one domain of linearity, $L$, to another, $L'$, and define
    \[ \mathfrak{D}_t = \{ (\mathfrak{d},f_\mathfrak{d}) \in \mathfrak{D} : \gamma(t) \in \mathfrak{d}. \} \]
    Then the power series $\mathfrak{p}_{\left.\gamma\right|_{(t-\epsilon,t+\epsilon)},\mathfrak{D}_t}(c_Lz^{m_L})$ contains the term $c_{L'}z^{m_{L'}}$. As in Definition \ref{def:genWallcrossing}, here the power series $\mathfrak{p}_{\left.\gamma\right|_{(t-\epsilon,t+\epsilon)},\mathfrak{D}_t}(c_{L}z^{m_L})$ is over the ground ring $R$ and may contain many more terms than in the ordinary case.
\end{enumerate}
\end{definition}
The definition of a theta function in terms of broken lines remains the same, except that we are now working over the ground ring $R$. 

\begin{definition}
\label{def:genThetaFunction}
Suppose $\mathfrak{D}$ is a generalized cluster scattering diagram and consider points $m_0 \in M^\circ \backslash \{ 0 \}$ and $Q \in M_\mathbb{R} \backslash \textrm{Supp}(\mathfrak{D})$. For a broken line $\gamma$ with initial exponent $m_0$ and endpoint $Q$, we define $I(\gamma) = m_0$, $b(\gamma) = Q$, and $\textrm{Mono}(\gamma) = c(\gamma)z^{F(\gamma)}$ where $\textrm{Mono}(\gamma)$ is the monomial attached to the final domain of linearity of $\gamma$. We then define
\[ \vartheta_{Q,m_0} := \sum_\gamma \textrm{Mono}(\gamma) \]
where the summation ranges over all broken lines $\gamma$ with initial exponent $m_0$ and endpoint $Q$. When $m_0 = 0$, then for any endpoint $Q$ we define $\vartheta_{Q,0} = 1$.
\end{definition}

We show that just as in the ordinary case, theta functions defined via generalized cluster scattering diagrams satisfy several crucial properties, such as the following: the collection of them include cluster monomials, agreement of theta functions and path-ordered products for cluster monomials, and Laurentness.

\begin{theorem}\cite[Theorem 3.5]{GHKK}
\label{theorem:genthetafuncomposition}
Let $\mathfrak{D}$ be a consistent scattering diagram, $m_0$ be a point in $M \backslash \{ 0 \}$, and consider a pair of points $Q$ and $Q'$ in $M_\mathbb{R} \backslash \textrm{Supp}(\mathfrak{D})$ such that $Q$ and $Q'$ are linearly independent over $\mathbb{Q}$. Then for any path $\gamma$ with endpoints $Q$ and $Q'$ for which $\mathfrak{p}_{\gamma,\mathfrak{D}}$ is defined, we have
\[ \vartheta_{Q',m_0} = \mathfrak{p}_{\gamma,\mathfrak{D}}(\vartheta_{Q,m_0}) \]
\end{theorem}

\begin{proof}
As in the ordinary setting, this is a special case of the results of Section 4 of \cite{CPS}. Those results do not assume that the wall-crossing automorphisms are binomials and are therefore also applicable to our setting.
\end{proof}

In Section~\ref{subsec:genMutInvariance}, we discussed the mutation invariance of generalized cluster scattering diagrams. It is also important that the theta functions exhibit this mutation invariance. Recall that the positive chamber of a cluster scattering diagram corresponds to a choice of initial torus seed $\bfs$ for the associated generalized cluster algebra. If the cluster scattering diagrams $\mathfrak{D}_{\bfs}$ and $\mathfrak{D}_{\bfs'}$ are related by a single application of the map $T_k$, then the initial torus seeds are related by a single mutation, i.e. $\bfs' = \mu_k(\bfs)$. The following proposition exhibits a bijection between the sets of broken lines and theta functions defined on  $\mathfrak{D}_{\bfs}$ and $\mathfrak{D}_{\bfs'}$.
\begin{prop}[Analog of Proposition 3.6 of \cite{GHKK}]
\label{prop:brokenLineBijection}
The transformation $T_k$ gives a bijection between broken lines with endpoint $Q$ and initial slope $m_0$ in $\mathfrak{D}_\bfs$ and broken lines with endpoint $T_k(Q)$ and initial slope $T_k(m_0)$ in $\mathfrak{D}_{\mu_k(\bfs)}$. In particular,
\begin{align*}
    \vartheta^{\mu_k(\bfs)}_{T_k(Q),T_k(m_0)} &= \begin{cases}
            T_{k,+}\left( \vartheta^\bfs_{Q,m_0} \right) & Q \in \mathcal{H}_{k,+} \\
            T_{k,-}\left( \vartheta^\bfs_{Q,m_0} \right) & Q \in \mathcal{H}_{k,-}
        \end{cases}
\end{align*}
where $T_{k,\pm}$ acts linearly on the exponents in $\vartheta^\bfs_{Q,m_0}$.
\end{prop}

\begin{proof}
We follow the structure of the proof of Proposition 3.6 from \cite{GHKK}. 

Let $\gamma$ be a broken line in a scattering diagram $\mathfrak{D}_\bfs$ and $T_k(\gamma)$ denote the composite map $T_k \circ \gamma: (-\infty,0] \rightarrow M_\mathbb{R}$. If any domain of linearity of $\gamma$ is in both $\mathcal{H}_{k,+}$ and $\mathcal{H}_{k,-}$, we can subdivide that domain of linearity at the point where it crosses between $\mathcal{H}_{k,+}$ and $\mathcal{H}_{k,-}$. As such, we can assume for any domain of linearity $L$ that $\gamma(L)$ falls either entirely inside $\mathcal{H}_{k,+}$ or entirely inside $\mathcal{H}_{k,-}$. For any domain of linearity $L$ that has been subdivided in this way, the associated monomial $c_Lz^{m_L}$ will be sent to either $c_Lz^{T_{k,+}(m_L)}$ or $c_Lz^{T_{k,-}(m_L)}$ depending on the portion of $L$ being considered. We know from Theorem \ref{theorem:mutDiagConsistent} that $\mathfrak{D}_{\mu_k(\bfs)} = T_k(\mathfrak{D}_\bfs)$, so we can think about $T_k(\mathfrak{D}_\bfs)$ when thinking about the broken line in $\mathfrak{D}_{\mu_k(\bfs)}$.

We know that $e_k^\perp$ lies on the boundary between $\mathcal{H}_{k,+}$ and $\mathcal{H}_{k,-}$. So in order to understand what happens to the subdivided domains of linearity, which originally were in both $\mathcal{H}_{k,+}$ and $\mathcal{H}_{k,-}$, we need to analyze what happens when $\gamma$ crosses $e_k^\perp$. First, consider the original broken line $\gamma$ in $\mathfrak{D}_\bfs$. Suppose that one domain of linearity, $L$, has been subdivided into $L_1$ and $L_2$ such that $\gamma$ crosses $e_K^{\perp}$ at the point where it passes from the first domain of linearity, $L_1$, to the second domain of linearity, $L_2$. By definition, we know that when the monomial $c_{L_1}z^{m_{L_1}}$ passes through $e_k^\perp$, it is mapped to
\[c_{L_1}z^{m_{L_1}}\left(1 + a_{k,_1}z^{v_k} + \cdots + a_{k,r_k-1}z^{(r_k-1)v_k} + z^{r_kv_k} \right)^{|\langle d_ke_k,m_{L_1} \rangle|} \]
and that $c_{L_2}z^{m_{L_2}}$ must appear as a term in this polynomial.

We can then consider the image of $\gamma$ in $T_k(\mathfrak{D}_\bfs)$. If ${L_1 \subseteq \mathcal{H}_{k,-}}$ and ${L_2 \subseteq \mathcal{H}_{k,+}}$, then $c_{L_2}z^{T_{k,+}(m_{L_2})}$ must appear as a term in the polynomial
{\footnotesize
\begin{align*}
    c_{L_1}z^{T_{k,+}(m_{L_1})}\left( 1 + \cdots  + z^{r_kv_k} \right)^{-\langle d_ke_k,m_{L_1} \rangle} &= c_{L_1}z^{m_{L_1} + r_kv_k \langle d_ke_k,m_{L_1} \rangle}\left( 1 + \cdots  + z^{r_kv_k} \right)^{-\langle d_ke_k,m_{L_1} \rangle} \\
    &= c_{L_1}z^{m_{L_1}}\left(z^{-r_kv_k}(1+\cdots+z^{r_kv_k}) \right)^{-\langle d_ke_k,m_{L_1} \rangle} \\
    &= c_{L_1}z^{m_{L_1}}\left(z^{-r_kv_k} + a_{k,1}z^{-(r_k-1)v_k} + \cdots + a_{k,r_k-1}z^{-v_k} + 1 \right)^{-\langle d_ke_k,m_{L_1} \rangle} \\
    &= c_{L_1}z^{T_{k,-}(m_{L_1})}\left(z^{-r_kv_k} + a_{k,1}z^{-(r_k-1)v_k} + \cdots + a_{k,r_k-1}z^{-v_k} + 1 \right)^{-\langle d_ke_k,m_{L_1} \rangle} 
\end{align*}
}
Due to the assumption that the exchange polynomials have reciprocal coefficients - i.e., that $a_{k,i} = a_{k,r_k-i}$ - this polynomial is equal to
\[ c_{L_1}z^{T_{k,-}(m_{L_1})}\left(1 + a_{k,1}z^{-v_k} + \cdots + a_{k,r_k-1}z^{-(r_k-1)v_k} + z^{-r_kv_k} \right)^{-\langle d_ke_k,m_{L_1} \rangle} \]
and therefore $T_k(\gamma)$ satisfies the rules for bending as it crosses
\[ \mathfrak{d}_k' = (e_k^\perp,1+a_{k,1}z^{-v_k} + \cdots + a_{k,r_k-1}z^{-(r_k-1)v_k} + z^{-r_kv_k}) \]
in $T_k(\mathfrak{D}_\bfs)$. Similarly, if $L_1 \subseteq \mathcal{H}_{k,+}$ and $L_2 \subseteq \mathcal{H}_{k,-}$, then $c_{L_2}z^{T_{k,-}(m_{L_2})} = c_{L_2}z^{m_{L_2}}$ must appear as a term in
{\footnotesize
\begin{align*}
    &c_{L_1}z^{T_{k,-}(m_{L_1})}\left(1 + \cdots + z^{r_kv_k} \right)^{\langle d_ke_k,m_{L_1} \rangle} 
    \\=& c_{L_1}z^{m_{L_1}} \left(1 + a_{k,1}z^{v_k} + \cdots + a_{k,r_k-1}z^{(r_k-1)v_k} + z^{r_kv_k} \right)^{\langle d_ke_k,m_{L_1} \rangle} \\
    =& c_{L_1}z^{m_{L_1}} \left(1 + a_{k,r_k-1}z^{v_k} + \cdots + a_{k,1}z^{(r_k-1)v_k} + z^{r_kv_k} \right)^{\langle d_ke_k,m_{L_1} \rangle} \\
    =& c_{L_1}z^{m_{L_1}} \left(z^{r_kv_k} (z^{-r_kv_k} + a_{k,r_k-1}z^{-(r_k-1)v_k} \cdots + a_{k,1}z^{-v_k} + 1) \right)^{\langle d_ke_k,m_{L_1} \rangle} \\
    =& c_{L_1}z^{m_{L_1} + r_kv_k \langle d_ke_k,m_{L_1} \rangle}\left( 1 + a_{k,1}z^{-v_k} + \cdots + a_{k,r_k-1}z^{-(r_k-1)v_k} + z^{-r_kv_k} \right)^{\langle d_ke_k,m_{L_1} \rangle} \\
    =& c_{L_1}z^{T_{k,+}(m_{L_1})}\left( 1 + a_{k,1}z^{-v_k} + \cdots + a_{k,r_k-1}z^{-(r_k-1)v_k} + z^{-r_kv_k} \right)^{\langle d_ke_k,m_{L_1} \rangle}
\end{align*}
}
and therefore $T_k(\gamma)$ also satisfies the rules for bending at $\mathfrak{d}_k'$ in this case. As such, we've verified that for any broken line $\gamma$ in $\mathfrak{D}_\bfs$, its image $T_k(\gamma)$ is also a broken line in $\mathfrak{D}_{\mu_k(\bfs)} = T_k(\mathfrak{D}_\bfs)$. To see that $T_k$ is, in fact, a bijection, we must verify that $T_k^{-1}(T_k(\gamma)) = \gamma$. First, we define $T_k^{-1}: \mathfrak{D}_{\mu_k(\bfs)} \rightarrow \mathfrak{D}_\bfs$ as
\begin{align*}
    T_k^{-}(m) &= \begin{cases}
                    m & m \in \mathcal{H}_{k,+}' \\
                    m - r_kv_k\langle d_ke_k,m \rangle & m \in \mathcal{H}_{k,-}'
                \end{cases}
\end{align*}
where $\mathcal{H}_{k,+}'$ and $\mathcal{H}_{k,-}'$ are defined relative to $e_k'$. Notice, however, that because mutation in direction $k$ sends $e_k$ to $e_k' = -e_k$, we have $\mathcal{H}_{k,+}' = \mathcal{H}_{k,-}$ and $\mathcal{H}_{k,-}' = \mathcal{H}_{k,+}$. As such, showing that $T_k^{-1}(T_k(\gamma)) = \gamma$ amounts to showing that $T_{k,+}^{-1} \circ T_{k,-} = \textrm{id}$ and $T_{k,-}^{-1} \circ T_{k,+} = \textrm{id}$. The first equality follows trivially from the definitions and we can verify the second by observing that
\begin{align*}
    T_{k,-}^{-1} \circ T_{k,+}(m) &= T_{k,-}^{-1} \left( m + r_kv_k \langle d_ke_k,m \rangle \right) \\
    &= (m + r_kv_k \langle d_ke_k,m \rangle) - r_kv_k \langle d_ke_k, m + r_kv_k \langle d_ke_k,m \rangle \rangle \\
    &= m + r_kv_k \langle d_ke_k,m \rangle - r_kv_k \langle d_ke_k,m \rangle - r_kv_k \langle d_ke_k,r_kv_k\langle d_ke_k,m \rangle \rangle \\
    &= m - r_kv_k \langle d_ke_k,m \rangle \langle d_ke_k,r_kv_k \rangle
\end{align*}
By definition, we know that $v_k = p_1^*(e_k)$ and so $\langle d_ke_k, r_kv_k \rangle = 0$ and the above expression reduces to $T_{k,-}^{-1} \circ T_{k,+}(m) = m$, as desired.
\end{proof}
In fact, such a bijection exists for any pair of diagrams $\mathfrak{D}_{\bfs}$ and $\mathfrak{D}_{\bfs'}$ where $\bfs$ and $\bfs'$ are mutation equivalent. The explicit bijection can be obtained by simply iterating the previous proposition for each step in the mutation sequence between $\bfs$ and $\bfs'$.

\begin{figure}
    \begin{center}
        \begin{tikzpicture}
		\draw (-4.5,0) to (3,0);
		\draw (0,-1) to (0,4.5);
		\draw (0,0) to (1,-1);
		\draw (0,0) to (0.67,-1);
		\draw (0,0) to (0.5,-1);
		\draw (0,0) to (0.33,-1);
		
		\draw[green!60!black] (1.5,1) to (0,4);
		\draw[green!60!black] (0,4) to (-4,0);
		\draw[green!60!black] (-4,0) to (-4,-1);
		\node[green!60!black,scale=0.8] at (1.2,2.8) {$z^{(-1,2)}$};
		\node[green!60!black,scale=0.8] at (-1.9,2.8) {$z^{(-1,-1)}$};
		\node[green!60!black,scale=0.8] at (-4.5,-0.5) {$z^{(0,-1)}$};
		
		\draw[purple] (1.5,1) to (0,2.5);
		\draw[purple] (0,2.5) to (-2.5,0);
		\draw[purple] (-2.5,0) to (-2.5,-1);
		\node[purple,scale=0.7] at (0.5,1.5) {$az^{(-1,1)}$};
		\node[purple,scale=0.8] at (-1.5,1.7) {$z^{(-1,-1)}$};
		\node[purple,scale=0.8] at (-3,-0.5) {$z^{(0,-1)}$};
		
		\draw[blue] (1.5,1) to (0,1);
		\draw[blue] (0,1) to (-1,0);
		\draw[blue] (-1,0) to (-1,-1);
		\node[blue,scale=0.8] at (0.65,0.8) {$az^{(-1,0)}$};
		\node[blue,scale=0.8] at (-0.9,0.7) {$z^{(-1,-1)}$};
		\node[blue,scale=0.8] at (-1.5,-0.5) {$z^{(0,-1)}$};
		
		\draw[orange] (1.5,1) to (0.5,0);
		\draw[orange] (0.5,0) to (0.5,-1);
		\node[orange,scale=0.6] at (1.1,0.2) {$z^{(-1,-1)}$};
		\node[orange,scale=0.6] at (0.9,-0.3) {$z^{(0,-1)}$};
		
		\draw[red] (1.5,1) to (1.5,-1);
		\node[red,scale=0.8] at (2,0.5) {$z^{(0,-1)}$};
		
		\draw[fill] (1.5,1) circle [radius=2pt] ;
		\node at (1.7,1.2) {$Q$};
	\end{tikzpicture}
	\caption{The broken lines for $\vartheta_{(0,-1), Q}$ in $\mathfrak{D}_{\bfs}$ for the generalized cluster algebra and generalized torus seed from Example~\ref{ex:genFixedData}.
	\label{fig:genBrokenLine}}
    \end{center}
\end{figure}
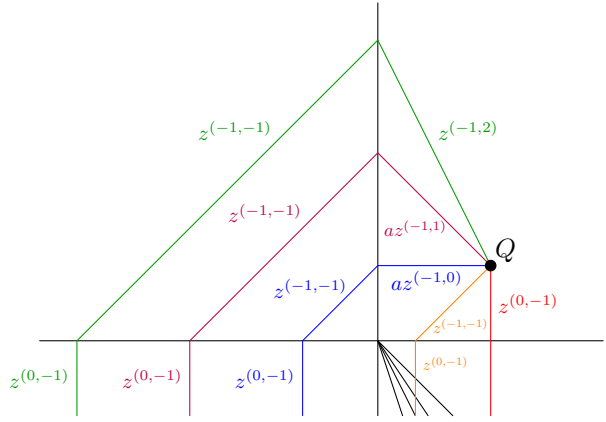

The following proposition is crucial in showing that the generalized cluster variables are, in fact, theta functions.

\begin{prop}\cite[Proposition 3.8]{GHKK}
\label{prop:clusterMonomialPosChamber}
For a point $Q$ in $\textrm{Int}(\mathcal{C}_\bfs^+)$ and a point $m$ in $\mathcal{C}_\bfs^+ \cap M^\circ$, we have
\[ \vartheta_{Q,m} = z^m \]
\end{prop}

\begin{proof}
The proof of this proposition is identical to the proof given for the ordinary version in \cite{GHKK}. The fact that the wall-crossing automorphisms now contain additional terms, which offer more options for scattering, can be accounted for in the choice of the normal vectors $n_i$ in that proof.
\end{proof}

One immediate corollary is that the cluster monomials are also theta functions. As with ordinary cluster algebras, this is a highly desirable property for a basis for generalized cluster algebras.
\begin{cor}\cite[Corollary 3.9]{GHKK}
\label{prop:pathinsidechamber}
Let $\sigma \in \Delta_\bfs^+$ be a cluster chamber. Then for any points $Q \in \textrm{Int}(\sigma)$ and $m \in \sigma \cap M^\circ$, we have $\vartheta_{Q,m} = z^m$
\end{cor}

\begin{proof}
The result follows from Propositions \ref{prop:brokenLineBijection} and  \ref{prop:clusterMonomialPosChamber}.
\end{proof}

As in the ordinary case, Theorem~\ref{theorem:genthetafuncomposition} and Corollary~\ref{prop:pathinsidechamber} give us a way to compute theta functions using path-ordered products within the cluster complex.  This was noted also in \cite[Theorem 5.6]{ Reading_Scattering-Fans}, which relies on \cite[Theorem 3.5]{GHKK} and \cite[Corollary 3.9]{GHKK}.

\begin{prop}
\label{prop:thetaFunctionComputation}
Consider $m_0 \in M^{\circ} \backslash \{ 0 \}$ such that there exists a path $\gamma$ from $m_0$ to some point $Q$ in the positive chamber $\mathcal{C}^+$ which passes through finitely many chambers. Then
\[ \vartheta_{Q,m_0} = \mathfrak{p}_{\gamma,\mathfrak{D}}(z^{m_0})\]
\end{prop}

\begin{proof}
By assumption, we know that the path $\gamma$ from $m_0$ to $Q$ passes through finitely many chambers. Let $\sigma_1$ denote the first chamber through which $\gamma$ passes and let $Q'$ be a point in $\sigma_1$ which lies on $\gamma$. By Proposition~\ref{prop:clusterMonomialPosChamber}, we know that $\vartheta_{Q',m_0} = z^{m_0}$. Let $Q''$ be a point in $\mathcal{C}_{\mathbf{s}}^{+}$ such that the coordinates of $Q'$ and $Q''$ are linearly independent over $\mathbb{Q}$ and let $\gamma'$ denote a path between $Q'$ and $Q''$ which follows $\gamma$ until within the interior of the positive chamber $\mathcal{C}_{\mathbf{s}}^+$, at which point it goes to $Q''$ rather than $Q$. By Theorem~\ref{theorem:genthetafuncomposition}, we know that $\vartheta_{Q'',m_0} = \mathfrak{p}_{\gamma'}\left(\vartheta_{Q',m_0} \right) = \mathfrak{p}_{\gamma'}\left(z^{m_0} \right)$.

Because both path-ordered products and theta functions are independent of the exact location of their endpoints within the interior of a chamber, we therefore have
\[ \vartheta_{Q'',m_0} = \vartheta_{Q,m_0} = \mathfrak{p}_{\gamma'}(z^{m_0})= \mathfrak{p}_{\gamma}(z^{m_0}) \]
\end{proof}

We can then establish a weaker version of Theorem 4.9 of \cite{GHKK}, without the guaranteed positivity of Laurent polynomial coefficients:

\begin{theorem} \label{theorem:almost_positive}
For generalized fixed data $\Gamma$, which satisfies the injectivity assumption, and a choice of initial generalized torus seed $\bfs$, consider a point $Q \in \mathcal{C}_\bfs^+$ and a point $m \in \sigma \cap M^\circ$ for some chamber $\sigma \in \Delta_\bfs^+$. Then $\vartheta_{Q,m}$ expresses a cluster monomial of $\mathcal{A}$ in $\bfs$ as a Laurent polynomial. Moreover, all cluster monomials can be expressed as $\vartheta_{Q,m}$ for some choice of $Q$ and $m$.
\end{theorem}

\begin{proof}
The proof of Theorem 4.9 from \cite{GHKK} holds in the generalized setting, except for the proof of positivity. 
\end{proof}

\begin{remark} \label{rk:wecaretoomuch}
The proof of positivity in \cite{GHKK} uses an earlier result, Theorem 1.13, for which we do not currently have a generalized analogue.  In particular, \cite[Theorem 1.13]{GHKK} states (in the case of ordinary cluster algebras) that the scattering diagram $\mathfrak{D}_{\bfs}$ is equivalent to one such that all walls can be expressed as 
$(\mathfrak{d},f_{\mathfrak{d}})$ where $f_{\mathfrak{d}} = (1+z^m)^c$ with $m = p^*(n)$ for some $n \in N^+$ which is normal to $\mathfrak{d}$ and $c \in \mathbb{Z}_{> 0}$. 
In \cite{langmou}, Mou works under the assumption that walls are indeed expressed in this restricted way, and is able to obtain positivity in that setting.
Since we are allowing polynomial exchanges that are not simply binomials, we allow ourselves to work with scattering diagrams that are not necessarily equivalent to one with walls only of this form.
\end{remark}

\subsection{The $g$-vectors in cluster scattering diagrams}
\label{sec:gVectors}

Recall from Section~\ref{subsec:ordDiagrams} that $g$-vectors can be defined as the tropical points of theta functions. This formulation allows for a definition of $g$-vectors on all types of ordinary cluster varieties ($\mathcal{A}_{\textrm{prin}}$, $\mathcal{A}$, and $\mathcal{X}$). In this section, we give the analogous definition in the context of generalized cluster scattering diagrams and generalized cluster varieties.

There is a $T_{N^{\circ}}$ action on $\Aprin$ that can be specified at the level of cocharacter lattices as
\begin{align*}
    N^{\circ} &\rightarrow N^{\circ} \oplus M, \\
    n &\mapsto (n,p^*(n)).
\end{align*}
Under this $T_{N^{\circ}}$ action, each cluster monomial on $\Aprin$ is a $T_{N^{\circ}}$-eigenfunction as stated in Section \ref{subsec:genAprintoAX}. Via this action, choosing a generalized torus seed $\bfs$ determines a canonical extension of each cluster monomial on $\mathcal{A}$ to a cluster monomial on $\Aprin$. This allows us to define the $g$-vector of a cluster monomial of $\mathcal{A}$.

\begin{definition}[Analogue of Definition 5.6 of \cite{GHKK}]
The \emph{$g$-vector} with respect to the generalized torus seed $\bfs$ associated to a cluster monomial of $\mathcal{A}$ is the $T_{N^{\circ}}$-weight of its lift determined by $\bfs$.
\end{definition}

There is another way to characterize $g$-vectors which is extensible to the other types of generalized cluster varieties.

\begin{definition}[Analogue of Definition 5.8 of \cite{GHKK}]
\label{def:GHKK_5-8}
Consider the generalized cluster variety $\mathcal{A} = \bigcup_{\bfs} T_{N^{\circ},\bfs}$. Let $\overline{x}$ denote a cluster monomial of the form $z^{m}$ on a chart $T_{N^{\circ},\bfs'}$ where $\bfs' = \{ (e_i',\mathbf{a}_i') \}$. Identify $\mathcal{A}^{\vee}(\mathbb{R}^T)$ with $M_{\mathbb{R},\bfs'}^{\circ}$ Because $(z^{e_i'})^{T}(m) \leq 0$ for all $i$, $m$ is identified with a point in the Fock-Goncharov cluster chamber $\mathcal{C}_{\bfs'}^{+} \subseteq \mathcal{A}^{\vee}(\mathbb{R}^T)$. Define $\mathbf{g}(\overline{x})$ to be this point in $\mathcal{C}_{\bfs'}^{+} \subseteq \mathcal{A}^{\vee}(\mathbb{R}^T)$.
\end{definition}

\begin{definition}[Analogue of Definition 5.10 of \cite{GHKK}]
\label{def:GHKK_5-10}
Consider a generalized cluster variety $V = \bigcup_{\bfs} T_{L,\bfs}$. Let $f$ be a global monomial on $V$ and $\bfs$ be a generalized torus seed such that $\left. f \right|_{T_{L,\bfs}} \subset V$ is the character $z^m$ for $m \in \textrm{Hom}(L,\mathbb{Z}) = L^*$. Then the \emph{$g$-vector of $f$}, denoted $\mathbf{g}(f)$, is the image of $m$ under the identifications $V^{\vee}(\mathbb{Z}^T) = T_{L^*,\bfs}(\mathbb{Z}^T) = L^*$.
\end{definition}

From Definitions~\ref{def:GHKK_5-8} and \ref{def:GHKK_5-10}, we obtain the following corollary.

\begin{cor}[Analogue of Corollary 5.9 of \cite{GHKK}]
\label{cor:GHKK_5-9}
Let $\bfs$ be a generalized torus seed and $\overline{x}$ be a cluster monomial on the associated generalized $\mathcal{A}$-variety. The seed $\bfs$ gives an identification $\mathcal{A}^{\vee}(\mathbb{R}^T) = M_{\mathbb{R},\bfs}^{\circ}$ under which $\mathbf{g}(\overline{x})$ is the $g$-vector of the cluster monomial $\overline{x}$ with respect to $\bfs$.
\end{cor}

We can then extend the definition of a $g$-vector beyond the generalized $\mathcal{A}$-variety to any type of generalized cluster variety.

\begin{definition}[Analogue of Definition 5.10 of \cite{GHKK}]
Consider a generalized cluster variety $V = \bigcup_{\bfs} T_{L,\bfs}$. Let $f$ be a global monomial on $V$ and $\bfs$ be a generalized torus seed such that the restriction $\left. f \right|_{T_{L,\bfs}} \subset V$ is the character $z^m$ for some $m \in \textrm{Hom}(L,\mathbb{Z}) = L^*$. We then define the \emph{$g$-vector of $f$}, denoted $\mathbf{g}(f)$, as the image of $m$ under the identifications $V^{\vee}(\mathbb{Z}^T) = T_{L^*,\bfs}(\mathbb{Z}^T) = L^*$.
\end{definition}

Although it is not \emph{a priori} clear from this definition, we will see in Lemma~\ref{lemma:GHKK_7-10} that this definition of $g$-vector is actually independent of the choice of generalized torus seed $\bfs$. As in the ordinary case, this formulation of $g$-vectors allows for a very quick and elegant proof that the $g$-vectors are sign-coherent.

\begin{theorem}[Analogue of Theorem 5.11 of \cite{GHKK}]
\label{thm:g-coherence}
Consider an initial generalized torus seed $\bfs = \{ (e_i,(a_{i,j})) \}$, which defines the usual set of dual vectors $\{ f_i = d_i^{-1}e_i^* \}$. If $\bfs'$ is a mutation equivalent generalized torus seed, then the $i$-th coordinates of the $g$-vectors for the cluster variables in $\bfs'$ are either all non-negative or all non-positive when expressed in the basis $\{ f_1, \dots, f_n \}$.
\end{theorem} 

For both Corollary~\ref{cor:GHKK_5-9} and Theorem~\ref{thm:g-coherence}, the proof given in \cite{GHKK} holds in the generalized setting as well, using the appropriate analogues of intermediate results.

\section{The product structure
of theta functions}

\label{sec:genThetaBasis}

In this section, we develop a number of properties satisfied by theta functions and show how this allows us to construct the $\mathcal{A}$- and $\mathcal{X}$-generalized cluster varieties from $\mathcal{A}_{\textrm{prin}}$.  
To begin, we recall some useful notation from Section~\ref{subsec:ordThetaBasis}.
For a broken line $\gamma$, let $\textrm{Mono}(\gamma) = c(\gamma)z^{F(\gamma)}$ be the monomial attached to its final domain of linearity. Then $c(\gamma)$ denotes the coefficient and $F(\gamma)$ the exponent in that final domain of linearity. Let $I(\gamma)$ and $b(\gamma)$ denote the initial slope and endpoint, respectively, of $\gamma$.

With this notation, we can then define structure constants for the multiplication of theta functions.

\begin{prop}[Analogue of Definition-Lemma 6.2 in \cite{GHKK}]
\label{prop:multiplication_constants_gen}
Let $p_1,p_2,$ and $q$ be points in $\widetilde{M}_{\bfs}^{\circ}$ and $z$ be a generic point in $\widetilde{M}_{\mathbb{R},\bfs}^{\circ}$. There are at most finitely many pairs of broken lines $\gamma_1,\gamma_2$ such that $\gamma_i$ has initial slope $p_i$, both broken lines have endpoint $z$, and $F(\gamma_1) + F(\gamma_2) = q$. Let
\[ a_{z}(p_1,p_2,q) := \sum_{\substack{ (\gamma_1,\gamma_2) \\ I(\gamma_i) = p_i, b(\gamma_i) = z \\ F(\gamma_1) + F(\gamma_2) = q}} c(\gamma_1)c(\gamma_2) \]
The $\alpha_z(p_1,p_2,q)$ are linear combinations of the formal variables $\{a_{i,j} \}$.
\end{prop}

\begin{proof}
The proof of this proposition is identical to that given in \cite{GHKK}, except that
in the generalized setting, the broken line monomials lie in $R[P]$ instead of $\Bbbk[P]$.
However, this does not change the essence of the proof.
\end{proof}

We then obtain the following decomposition of products of theta functions:

\begin{lemma}[Analogue of Proposition 6.4(3) of \cite{GHKK}]
\label{lemma:structureConstants}
Let $p_1,p_2,$ and $q$ be points in $\widetilde{M}_{\bfs}^{\circ}$ and $z$ be a generic point in $\widetilde{M}_{\mathbb{R},\bfs}^{\circ}$. Then
\[ \vartheta_{p_1} \cdot \vartheta_{p_2} = \sum_{q \in \widetilde{M}_{\bfs}^{\circ}} \alpha_{z(q)}(p_1,p_2,q)\vartheta_q \]
for $z(q)$ sufficiently close to $q$. When $z$ is sufficiently close to $q$, $a_z(p_1,p_2,q)$ is independent of the choice of $z$ and we can simply write $\alpha(p_1,p_2,q) := a_z(p_1,p_2,q)$.
\end{lemma}

\begin{proof}
The argument given in \cite{GHKK} for the analogous result for ordinary cluster scattering diagrams holds in our generalized setting. No portion of that argument assumes that the wall-crossing automorphisms are binomials.
\end{proof}

\subsection{The product structure of theta functions for $\mathcal{A}_{\textrm{prin}}$}
\label{subsec:genThetaBasis}

In this section, we will describe the product structure of theta functions on $\Aprin$ and show that the collection of theta functions gives a topological basis for a topological $R$-algebra completion of the upper cluster algebra of $\Aprin$.
Later, in Section~\ref{subsec:genAprintoAX}, we will descend to the $\mathcal{A}$ and $\mathcal{X}$ cases by using the fact that the $\mathcal{A}$-variety appears as a fiber of $\mathcal{A}_{\textrm{prin}} \rightarrow T_{M}$ and the $\mathcal{X}$-variety appears as the quotient $\mathcal{A}_{\textrm{prin}}/T_{N^{\circ}}$.

We wish to associate a formal summation $\sum_{q \in \Aprin^{\vee}(\mathbb{Z}^T)}\alpha(g)(q)\vartheta_q$, with coefficients $\alpha(g)(q) \in \Bbbk[a_{i,j}]$, to each universal Laurent polynomial $g$ on $\Aprin$. In doing so, we follow the structure of Section 6 of \cite{GHKK} for the ordinary case, with modifications when necessary to accommodate our generalized setting. We begin by giving such a summation for a fixed choice of generalized torus seed $\bfs$, then show that the coefficients $\alpha(g)(q)$ are, in fact, independent of the choice of generalized torus seed.

Fix a choice of generalized torus seed $\bfs = \{ (e_i,\{ a_{i,j} \}) \}$. Recall that in the generalized setting, we are working over the ground ring $R = \Bbbk [a_{i,j}]$ rather than over $\Bbbk$. Let $y_i := z^{e_i}$ and $I_{\bfs} = (y_1,\dots,y_n) \subset R[y_1,\dots,y_n]$. 
When principal coefficients are taken as the frozen variables, there is a partial compactification
$ \mathcal{A}_{\textrm{prin}}^{\bfs} \subseteq \overline{\mathcal{A}}_{\textrm{prin}}^{\bfs}$ constructed in the same manner as in \cite[Construction B.9]{GHKK}.

Then, set
\begin{align*}
    \mathbb{A}^{n}_{(y_1,\dots,y_n),R,k} &:= \textrm{Spec}~R[y_1,\dots,y_n]/I_{\bfs}^{k+1}, \\
    \overline{\mathcal{A}}_{\textrm{prin},k}^{\bfs} &:= \overline{\mathcal{A}}_{\textrm{prin}}^{\bfs} \times_{\mathbb{A}^{n}_{y_1,\dots,y_n}} \mathbb{A}^{n}_{(y_1,\dots,y_n),k}.
\end{align*}

The map $\pi: \mathcal{A}_{\textrm{prin}} \rightarrow T_M$ induces a map $\pi:\overline{\mathcal{A}}_{\textrm{prin}}^{\bfs} \rightarrow \mathbb{A}^{n}_{y_1,\dots,y_n}$, which then subsequently induces a map $\pi: \overline{\mathcal{A}}_{\textrm{prin},k}^{\bfs} \rightarrow \mathbb{A}_{(y_1,\dots,y_n),R,k}^{n}$. Let
\[ \widehat{\textrm{up}\left(\overline{\mathcal{A}}_{\textrm{prin}}^{\bfs} \right)} := \lim_{\longleftarrow} \textrm{up}\left( \overline{\mathcal{A}}^{\bfs}_{\textrm{prin},k} \right) \]

For any $g \in \textrm{up}(\Aprin)$, we have $z^ng \in \textrm{up}(\overline{\mathcal{A}}_{\textrm{prin}^{\bfs}})$ where $z^n$ is some monomial in the $y_i$. This fact induces the inclusion
\begin{equation}
    \label{eq:upInclusion}
    \textrm{up}(\Aprin) \subset \widehat{\textrm{up}(\overline{\mathcal{A}}_{\textrm{prin}}^{\bfs})} \otimes_{R[N^{+}_{\bfs}]} R[N]
\end{equation}
where $N_{\bfs}^{+} \subset N$ denotes the monoid generated by $e_1, \dots, e_n$. Let $\pi_N : \widetilde{M}^{\circ} = M^{\circ} \oplus N \rightarrow N$ be the projection map and define $\widetilde{M}_{\bfs}^{\circ,+} := \pi_N^{-1}(N_{\bfs}^{+})$. Let $P_{\bfs} \subset \widetilde{M}_{\bfs}^{\circ}$ be the monoid generated by $(v_1,e_1), \dots, (v_n,e_n)$.

We begin by establishing the following proposition, which defines canonical functions $\vartheta_{q}$ on $\textrm{up}(\overline{\mathcal{A}}_{\textrm{prin}}^{\bfs}) \otimes_{R[N^{+}]} R[N]$ for $q \in \widetilde{M}_{\bfs}^{\circ,+}$ and then shows that two particular collections of such $\vartheta_{q}$ form bases for $\textrm{up}(\overline{\mathcal{A}}_{\textrm{prin},k}^{\bfs})$. For each $\sigma \in \Delta^{+}_{\mathbf{s}}$, let $Q_{\sigma}$ denote a generic point in $\sigma$.

\begin{prop}[Analogue of Proposition 6.4(1,2) of \cite{GHKK}]
\label{prop:vectorSpaceBasis}
\
\begin{enumerate}
    \item Given a point $q \in \widetilde{M_{\bfs}}^{\circ,+}$, the function $\vartheta_{Q_{\sigma},q}$ is a regular function on $V_{\mathbf{s},\sigma,k}$. As $\sigma$ varies, the $\vartheta_{Q_{\sigma},q}$ glue to yield a canonically defined function $\vartheta_{q,k} \in \textrm{up}\left(\overline{\mathcal{A}}_{\textrm{prin},k}^{\bfs} \right)$.
    \item For $q \in \mathcal{A}_{\textrm{prin}}^{\vee}$ and $k' \geq k$, $\left.\vartheta_{q,k'}\right|_{\mathcal{A}_{\textrm{prin},k}^{\bfs}} = \vartheta_{q,k}$. Hence, the collection $\{ \vartheta_{q,k}\}_{k \geq 0}$ canonically defines a function
    \[ \vartheta_q \in \widehat{\textrm{up}\left(\overline{\mathcal{A}}_{\textrm{prin}}^{\bfs} \right)} \otimes_{R[N_{\bfs}^{+}]} R[N]. \]
    Let $\textrm{can}(\mathcal{A}_{\textrm{prin}})$ denote the $R$-vector space 
    \[ \bigoplus_{q \in \mathcal{A}_{\textrm{prin}}^{\vee}(\mathbb{Z}^T)} R \cdot \vartheta_q. \]
    The $\vartheta_q$ are linearly independent, so there is a canonical inclusion of $R$-vector spaces
    \[ \textrm{can}(\mathcal{A}_{\textrm{prin}}) \subset \widehat{\textrm{up}\left(\overline{\mathcal{A}}_{\textrm{prin}}^{\bfs} \right)} \otimes_{R[N_{\bfs}^{+}]} R[N]. \]
\end{enumerate}
\end{prop}

\begin{proof}
The proof given in \cite{GHKK} for the ordinary case holds in our setting.
\end{proof}

In order to associate a formal summation $\sum \alpha(g)(q)\vartheta_q$ to each universal Laurent polynomial $g$ on $\Aprin$, we will first associate a formal summation $\sum \alpha_{\bfs}(g)(q)\vartheta_q$ which depends on the choice of generalized torus seed $\bfs$. To do so, we must first define the function $\alpha_s$.

\begin{prop}[Analogue of Proposition 6.5 of \cite{GHKK}]
\label{prop:unique_alpha_s}
There is a unique inclusion
\[\alpha_{\bfs} : \widehat{\textrm{up}\left(\overline{\A}_{\textrm{prin}}^{\bfs} \right)} \otimes_{R[N_{\bfs}^{+}]} R[N] \hookrightarrow \textrm{Hom}_{\textrm{sets}}\left(\Aprin^{\vee}(\mathbb{Z}^T) = \widetilde{M}_{\bfs}^{\circ}, R \right) \]
given by the map $g \mapsto (q \mapsto \alpha_{\bfs}(g)(q))$. For all $n \in N$, $\alpha_{\bfs}(z^n \cdot g)(q+n) = \alpha_{\bfs}(g)(q)$.
\end{prop}

\begin{proof}
One consequence of Proposition~\ref{prop:vectorSpaceBasis} and Lemma~\ref{lemma:structureConstants} is that every ${g \in \widehat{\textrm{up}(\overline{\A}_{\textrm{prin}}^{\bfs})}}$ can be uniquely expressed as a convergent formal sum $\sum_{q \in \widetilde{M}_{\bfs}^{\circ},+} \alpha_{\bfs}(g)(q)\vartheta_{q}$ where the coefficients $\alpha_{\bfs}(g)(q)$ lie in $R$. This immediately implies the desired unique inclusion.
\end{proof}

\begin{definition}[Analogue of Definition 6.6 of \cite{GHKK}]
\label{def:GHKK_6-6}
Let $g$ be a universal Laurent polynomial on $\textrm{up}(\Aprin)$. On the torus chart $T_{\widetilde{N}^{\circ},\bfs}$ of $\Aprin$, we can write $g = \sum_{q \in \widetilde{M}_{\bfs}^{\circ}} \beta_{\bfs}(g)(q)z^q$. Because $z^mg \in \widehat{\textrm{up}(\overline{\mathcal{A}}_{\textrm{prin}}^{\bfs})}$ for some $m \in \widetilde{M}_{\bfs}^{\circ}$, we can also write a formal expansion $g = \sum_{q \in \widetilde{M}_{\bfs}^{\circ}} \alpha_{\bfs}(g)(q)\vartheta_q$. Let
\begin{align*}
    \overline{S}_{g,\bfs} &:= \{q \in \widetilde{M}_{\bfs}^{\circ} : \beta_{\bfs}(g)(q) \neq 0 \},
    &S_{g,\bfs} := \{ q \in \widetilde{M}_{\bfs}^{\circ} : \alpha_{\bfs}(g)(q) \neq 0 \},
\end{align*}
and $P_{\bfs}$ be the monoid generated by $\{ (v_i,e_i) \}_{i \in I_{\textrm{uf}}}$.
\end{definition}

It follows from the construction of the theta functions that $S_{g,\bfs} \subseteq \overline{S_{g,\bfs}} + P_{\bfs}$.

We are then ready to prove that on $\textrm{up}(\mathcal{A}_{\textrm{prin}})$, the function $\alpha_{\bfs}$ is actually independent of the choice of generalized torus seed $\bfs$.

\begin{theorem}[Analogue of Theorem 6.8 of \cite{GHKK}]
\label{theorem:alpha_properties}
There is a unique function 
\[ \alpha: \textrm{up}\left( \Aprin \right) \rightarrow \textrm{Hom}_{\textrm{sets}}\left(\Aprin^{\vee}\left( \mathbb{Z}^T \right), R \right)\]
such that:
\begin{enumerate}
    \item The function $\alpha$ is compatible with the $R[N]$-module structure on $\textrm{up}\left( \Aprin \right)$ and the $N$-translation action on $\Aprin^{\vee}\left( \mathbb{Z}^T \right)$, i.e.
    \[ \alpha\left( z^n \cdot g \right)(x + n) = \alpha(g)(x) \]
    for all $g \in \textrm{up}\left( \Aprin \right)$, $n \in N$, and $x \in \Aprin^{\vee}\left( \mathbb{Z}^T \right)$.
    \item For any $\bfs$, the formal sum $\sum_{q \in \Aprin^{\vee}(\mathbb{Z}^T)} \alpha(g)(q)\vartheta_q$ converges to $g$ in $\widehat{\textrm{up}\left( \overline{\mathcal{A}}_{\textrm{prin}}^{\bfs} \right)} \otimes_{R[N_{\bfs}^{+}]} R[N]$. 
    \item If $z^n \cdot g$ lies in $\textrm{up}\left( \overline{\mathcal{A}}_{\textrm{prin}}^{\bfs}\right)$, then $\alpha\left(z^n \cdot g \right)(q) = 0$ unless $\pi_N(q) \in N_{\bfs}^{+}$. Moreover,
    \[ z^n \cdot g = \sum_{\pi_{N,\bfs}(q) \in N_{\bfs}^{+} \backslash (N_{\bfs}^{+})_{k+1}} \alpha(z^n \cdot g)(q)\vartheta_q~~\textrm{mod}\left( I_{\bfs}^{k+1} \right) \]
    and the coefficients $\alpha(z^n \cdot g)(q)$ are the coefficients for the expansion of $z^n \cdot g$ when it is viewed as an element of $\textrm{up}\left(\overline{\mathcal{A}}_{\textrm{prin}}^{\bfs} \right)$ and written in terms of the collection $\{ \vartheta_q : q \in \widetilde{M}_{\bfs}^{\circ,+} \backslash M_{\bfs,k+1}^{\circ,+} \}$.
    \item For any generalized torus seed $\bfs'$ reachable from $\bfs$ via a sequence of mutations, the map $\alpha$ is the composition of inclusions
    \[ \textrm{up}\left(\Aprin \right) \subset \widehat{\textrm{up}\left(\overline{\A}_{\textrm{prin}}^{\bfs'} \right)} \otimes_{R[N_{\bfs'}^{+}]} R[N]] \subset \textrm{Hom}_{\textrm{sets}}\left( \Aprin^{\vee}(\mathbb{Z}^T) = \widetilde{M}_{\bfs'}^{\circ}, \Bbbk \right) \]
    from Proposition~\ref{prop:unique_alpha_s} and Equation~\eqref{eq:upInclusion}. This maps the cluster monomial $A \in \textrm{up}\left( \Aprin \right)$ to the delta function $\delta_{\mathbf{g}(A)}$ where $\mathbf{g}(A) \in \Aprin^{\vee}(\mathbb{Z}^{T})$ is its $g$-vector.
\end{enumerate}
Moreover, $\alpha(g)(m) = \alpha_{\bfs'}(g)(m)$ for any generalized torus seed $\bfs'$.
\end{theorem}

\begin{proof}
The proof given in \cite{GHKK} for the ordinary case holds in our generalized setting, with minor modifications due to the polynomial mutation maps of the generalized setting. For details, see Theorem 4.10.5 of \cite{Kelley}.
\end{proof}

This theorem implies, as a corollary, that the theta functions form a topological basis for a natural topological $R$-algebra completion of $\textrm{up}(\Aprin)$.

\begin{cor}[Analogue of Corollary 6.11 of \cite{GHKK}]
\label{cor:6-11}
For $n \in N$, let the map 
\[ n^*: \textrm{Hom}_{\textrm{sets}}(\Aprin^{\vee}(\mathbb{Z}^T),\Bbbk) \rightarrow \textrm{Hom}_{\textrm{sets}}(\Aprin^{\vee}(\mathbb{Z}^T),\Bbbk) \]
denote precomposition by translation by $n$ on $\Aprin^{\vee}(\mathbb{Z}^T)$. Let $\overline{\textrm{up}(\Aprin)} \subset \textrm{Hom}_{\textrm{sets}}(\Aprin^{\vee}(\mathbb{Z}^T)$ be the vector subspace of functions $f$ such that for each generalized torus seed $\mathbf{s}$, there exists $n \in N$ such that the restriction of $n^*(f)$ to $\Aprin^{\vee}(\mathbb{Z}^T) \backslash \pi^{-1}_{N,\mathbf{s}}((N_{\mathbf{s}^+})_k)$ has finite support for all $k > 0$. Then, we have the inclusions
\[ \textrm{up}(\Aprin) \subset \overline{\textrm{up}(\Aprin)} = \bigcap_{\mathbf{s}} \widehat{\textrm{up}(\overline{\cA}_{\textrm{prin}}^{\mathbf{s}})} \otimes_{R[N_{\mathbf{s}}^{+}]} R[N]  \subset \textrm{Hom}_{\textrm{sets}}(\Aprin^{\vee}(\mathbb{Z}^T),\Bbbk) \]
and $\overline{\textrm{up}(\Aprin)}$ is a complete topological vector space under the weakest topology so that each inclusion $\overline{\textrm{up}(\Aprin)} = \bigcap_{\mathbf{s}} \widehat{\textrm{up}(\overline{\cA}_{\textrm{prin}}^{\mathbf{s}})} \otimes_{R[N_{\mathbf{s}}^{+}]} R[N]$ is continuous.

Let $\vartheta_{q} = \delta_{q} \in \overline{\textrm{up}(\Aprin)}$ be the delta function associated to $q \in \Aprin^{\vee}(\mathbb{Z}^T)$. Then $\{ \vartheta_{q} \}_{q \in \Aprin^{\vee}(\mathbb{Z}^T)}$ is a topological basis for $\overline{\textrm{up}(\Aprin)}$ and there is a unique topological $R$-algebra structure on $\overline{\textrm{up}(\Aprin)}$ such that $\vartheta_{p} \cdot \vartheta_q = \sum_r \alpha(p,q,r)\vartheta_r$ with structure constants $\alpha(p,q,r)$ as in Proposition~\ref{prop:multiplication_constants_gen}.
\end{cor}

Next, we show that the theta function $\vartheta_{Q,m_0}$, for $Q \in \sigma \in \Delta^{+}$ and $m_0 \in \mathcal{A}_{\textrm{prin}}^{\vee}(\mathbb{Z}^T)$, is a universal Laurent polynomial on $R[\widetilde{M}^{\circ}]$.

\begin{prop}[Analogue of Proposition 7.1 of \cite{GHKK}]
\label{prop:universalLaurentPolynomial}
Let $\mathbf{s} = \{ (e_i,(a_{i,j})) \}$ be a generalized torus seed. Fix some $m_0 \in \mathcal{A}_{\textrm{prin}}^{\vee}(\mathbb{Z}^T)$. If for some generic choice of $Q \in \sigma \in \Delta^{+}$ there are finitely many broken lines $\gamma$ in $\mathfrak{D}_{\bfs}$ with $I(\gamma) = m_0$ and $b(\gamma) = Q$, then this holds for any generic $Q' \in \sigma' \in \Delta^{+}$. Hence, $\vartheta_{Q,m_0} \in R[\widetilde{M}^{\circ}]$ is a universal Laurent polynomial.
\end{prop}

Even though our statement is similar to the one in \cite[Proposition 7.1]{GHKK}, there are several significant differences in the proof. Because the positivity of theta functions in the generalized setting has not yet been proven, we cannot make use of this positivity as in the argument given in \cite{GHKK} for the ordinary setting. Because we only consider $\vartheta_Q$ for $Q , Q' \in \Delta^+$, however, we make use of the fact that the chambers in the cluster complex can be associated to generalized torus seed data for which the wall functions are simply the mutation maps. Hence, we can assert the positivity of the wall functions within the cluster complex.  For this proof, we need $a_{i,j}$ to be formal variables rather than simply arbitrary elements in $\mathbb{P}$. 
Hence, we consider the generalized cluster algebras over the ring $R$ when defining the fixed data.

\begin{proof}
By Theorem~\ref{theorem:genthetafuncomposition}, we know that for basepoints $Q$ and $Q'$ in different chambers, the theta functions $\vartheta_{Q,m_0}$ and $\vartheta_{Q',m_0}$ are related by a composition of wall-crossings. When the basepoint varies within a chamber, the corresponding theta function does not change. Hence, it is sufficient to check that if $Q \in \sigma$ and $Q' \in \sigma'$ are in adjacent chambers with $Q'$ close to the wall $\sigma \cap \sigma'$, then $\vartheta_{Q,m_0}$ having finitely many terms implies that $\vartheta_{Q',m_0}$ also has finitely many terms.

Fix some generalized torus seed $\mathbf{s}$. Let the wall $\sigma \cap \sigma'$ be in $n_0^{\perp}$ for $n_0 \in \widetilde{N}^{\circ}$ with $\langle n_0, Q \rangle > 0$ and denote the wall-crossing automorphism when moving from $Q$ to $Q'$ by $\mathfrak{p}$. Recall that a chamber of $\mathfrak{D}_{\bfs}$ is called \emph{reachable} if there exists a finite, transverse path between that chamber and the positive chamber $\mathcal{C}_{\bfs}^{+} \subset \mathfrak{D}_{\bfs}$. By Lemma 2.10 of \cite{GHKK}, there exists a bijection between torus seeds that are mutation equivalent to the initial torus seed $\bfs$ and reachable chambers of $\mathfrak{D}_{\mathbf{s}}$. A consequence of this bijection, as described in \cite{Muller}, is that there exists a sequence of mutations $\mu_{k_1}, \dots, \mu_{k_{\ell}}$ and corresponding piecewise linear maps $T_{k_{\ell}}, \dots, T_{k_1}$ such that $\mathbf{s}' = \mu_{k_{\ell}} \circ \cdots \circ \mu_{k_1}(\mathbf{s})$, $\mathfrak{D}_{\bfs'} = T_{k_{\ell}} \circ \cdots T_{k_{1}}(\mathfrak{D}_{\bfs})$, and $T_{k_{i-1}} \circ \cdots \circ T_{k_1}(\sigma') \subset H_{k_{i},-}$ for each $i \in \{ 1, \dots, \ell \}$.

Recall that $\mathfrak{p}(z^m) = z^mf^{\langle n_0, m \rangle}$. When both chambers, $\sigma$ and $\sigma'$ are reachable, $f$ has the form $1 + a_1z^{q} + \cdots + a_{r-1}z^{(r-1)q} + z^{rq}$ for some $q \in n_0^{\perp} \subset \widetilde{M}^{\circ}$ and $r \in \mathbb{Z}_{>0}$. In particular, note that $f$ is a Laurent polynomial where $a_{1}, \dots, a_{r-1}$ are formal variables.  One can verify that $f$ has this form by recalling that the wall-crossing automorphisms associated to the positive chamber $\mathcal{C}_{\bfs'}^{+} \subset \mathfrak{D}_{\bfs'}$ have the form $1 + a_{k,1}z^{v_k} + \cdots + z^{r_kv_k}$ for some $k \in I$ and then applying the appropriate sequence of piecewise linear maps $T_{k_1}, \dots, T_{k_{\ell}}$ to obtain the wall-crossing automorphism $\mathfrak{p}$.

Monomials $z^m$ can be classified into three groups, based on the sign of $\langle n_0, m \rangle$. The arguments for $\langle n_0, m \rangle = 0$ or $\langle n_0, m \rangle > 0$ given in \cite{GHKK} will work here also. Briefly, when $\langle n_0, m \rangle = 0$, the monomial is fixed by $\mathfrak{p}$ and so these terms coincide in $\vartheta_{Q,m_0}$ and $\vartheta_{Q',m_0}$. When $\langle n_0, m \rangle > 0,$ the monomial $z^m$ is sent to $z^mf^{\langle n_0, m \rangle}$, which is by definition a polynomial. So each such $z^m$ in $\vartheta_{Q,m_0}$ corresponds to finitely many terms in $\vartheta_{Q',m_0}$.

The last case is when $\langle n_0, m \rangle < 0$. Consider a broken line in $\mathfrak{D}_{\bfs}$ with endpoint $Q' \in \sigma' \subset \mathfrak{D}_{\bfs}$ and a monomial of the form $cz^m$ with $\langle n_0, m \rangle < 0$ attached to its final domain of linearity. To complete the proof, it remains to show that there are finitely many such broken lines. By way of contradiction, assume that there actually infinitely many.

The direction vector of such a broken line must be towards the wall $\sigma \cap \sigma'$, so its final domain of linearity can be extended to some point $Q'' \in \sigma$. When crossing $\sigma \cap \sigma'$ from $\sigma'$ into $\sigma$, we have
\begin{align*}
    cz^{m} &\mapsto cz^{m}\left(1 + a_1z^{q} + \cdots + a_{r-1}z^{(r-1)q} + z^{rq} \right)^{\langle -n_0,m \rangle}
\end{align*}
Note that the primitive normal vector $-n_0$ appears in this wall-crossing computation rather than $n_0$ because by assumption $\langle n_0, Q \rangle > 0$, so $n_0$ is directed into the chamber $\sigma$ rather than into $\sigma'$. The fact that $a_1, \dots, a_{r-1}$ are all formal variables means that there are no cancellations. Because $\vartheta_{Q,m_0}$ is independent of the location of $Q$ within the chamber $\sigma$, this means there are infinitely many broken lines with initial slope $m_0$ and endpoint $Q$, a contradiction.
\end{proof}

\begin{definition} \cite[Definition 7.2]{GHKK} \label{def:thetaset}
Let $\Theta \subset \Aprin^{\vee}(\mathbb{Z}^T)$ be the collection of $m_0$ such that for any generic point $Q \in \sigma \in \Delta^{+}$, there exist finitely many broken lines with initial slope $m_0$ and endpoint $Q$.
\end{definition}

\begin{theorem}[Analogue of Theorem 7.5 of \cite{GHKK}]
\label{thm:7-5}
Let
\[ \Delta^{+}(\mathbb{Z}) := \bigcup_{\sigma \in \Delta^{+}} \sigma \cap \Aprin^{\vee}(\mathbb{Z}^T) \]
be the set of integral points in the chambers of the cluster complex. Then
\begin{enumerate}
    \item $\Delta^{+}(\mathbb{Z}) \subset \Theta$.
    \item There is an inclusion of $R[N]$-algebras
\begin{align*}
   \textrm{gen}(\mathcal{A}_{\textrm{prin}}) \subset \widehat{\textrm{up}(\mathcal{A}_{\textrm{prin},\bfs})} \otimes_{R[N_{\bfs}^{+}]} R[N]
\end{align*}
where each cluster monomial $Z \in \textrm{gen}(\mathcal{A}_{\textrm{prin}})$ is identified with $\vartheta_{\mathbf{g}(Z)} \in \widehat{\textrm{up}(\mathcal{A}_{\textrm{prin},\bfs})} \otimes_{R[N_{\bfs}^{+}]} R[N]$, where $\mathbf{g}(Z) \in \Delta^{+}(\mathbb{Z})$ denotes the $g$-vector of $Z$.
\end{enumerate}
\end{theorem}

\begin{proof}
The proof of (1) follows from Corollary~\ref{prop:pathinsidechamber} and Proposition~\ref{prop:thetaFunctionComputation}.

For (2), we know from Proposition~\ref{prop:universalLaurentPolynomial} that each $\vartheta_{Q,p}$ is a universal Laurent polynomial in $R[\widetilde{M}^{\circ}]$, for all $p \in \mathcal{A}_{\textrm{prin}}(\mathbb{Z}^T)$ and $Q \in \Delta^{+}$. When $p \in \Delta^{+}(\mathbb{Z})$, we know by Theorem~\ref{theorem:alpha_properties}(4) that $\vartheta_{Q,p}$ is the corresponding cluster monomial. The inclusion follows from Proposition~\ref{prop:vectorSpaceBasis}.
\end{proof}
As an immediate consequence, we obtain:
\begin{cor}[Analogue of Corollary 7.6 of \cite{GHKK}]
    There are canonically defined non-negative structure constants
    \[ \alpha: \Aprin^{\vee}(\mathbb{Z}^T) \times \Aprin^{\vee}(\mathbb{Z}^T) \times \Aprin^{\vee}(\mathbb{Z}^T) \rightarrow \mathbb{Z}_{\geq 0} \cup \{ \infty \}\]
    that are given by counts of broken lines.
\end{cor}

As in the ordinary case, we can verify that the theta functions are well-behaved with respect to the canonical torus action on $\mathcal{A}_{\textrm{prin}}$.

\begin{prop}[Analogue of Proposition 7.7 of \cite{GHKK}]
\label{prop:GHKK_7-7}
For $q \in \mathcal{A}_{\textrm{prin}}^{\vee}(\mathbb{Z})$,   $\vartheta_q \in \textrm{up}(\Aprin)$ is an eigenfunction for the natural $T_{\widetilde{K}^{\circ}}$ action on $\Aprin$ with weight $w(q)$ given by the map $w:\widetilde{M}^{\circ} = (\widetilde{N}^{\circ})^* \rightarrow (\widetilde{K}^{\circ})^*$. Moreover, $\vartheta_q$ is an eigenfunction for the subtorus $T_{N^{\circ}} \subset T_{\widetilde{K}^{\circ}}$ with weight $w(q)$ given by the map $ w: \widetilde{M}^{\circ} \rightarrow \widetilde{M}^{\circ}$ defined by the mapping $(m,n) \mapsto m - p^*(n)$.
\end{prop}

\begin{proof}
The proof given in \cite{GHKK} for the ordinary case holds in the generalized case as well, since we still have $w(v_i,e_i) = v_i - p^*(e_i) = 0$, by definition.
\end{proof}

\subsection{From $\mathcal{A}_{\textrm{prin}}$ to $\mathcal{A}_{t}$ and $\mathcal{X}$}
\label{subsec:genAprintoAX}

As in the ordinary case, our results for $\mathcal{A}_{\textrm{prin}}$ induce similar results on the $\mathcal{A}$ and $\mathcal{X}$ varieties. In this section, we adapt the results of section 7.2 of \cite{GHKK} for our generalized setting.

Similarly to the ordinary case, the choice of generalized torus seed data $\bfs = \{ (e_i,(a_{i,j})) \}$ gives a toric model for the generalized cluster variety $V$. The generalized torus seed $\mathbf{s}$ specifies a fan $\Sigma_{\bfs, V}$. More precisely, the fans for the $\mathcal{A}$ and $\mathcal{X}$ cases are
\begin{align*}
    \Sigma_{\bfs,\mathcal{A}} &:= \{0\} \cup \left\{ \mathbb{R}_{\geq 0}e_i : i \in I_{\textrm{uf}} \right\}, \\
    \Sigma_{\bfs,\mathcal{X}} &:= \{0\} \cup\left\{ -\mathbb{R}_{\geq 0}v_i : i \in I_{\textrm{uf}} \right\}.
\end{align*}
The fan $\Sigma_{\bfs,V}$ then defines a toric variety $TV(\Sigma_{\bfs,V})$. Note that in our setting we consider the toric varieties with base change to the ground ring $R$. The cluster varieties can be seen as (up to codimension 2) blowups $Y \rightarrow \mathrm{TV} \left( \Sigma_{\bfs, V} \right)$ of the toric varieties along the closed subschemes given by the mutation maps.
Note that we can simply follow \cite[Construction 3.4]{GHK} which does not necessarily depend on the mutations being binomials but only on the mutation functions being polynomials in a single variable. 

These toric models allow us to determine the global monomials:

\begin{lemma}[Analogue of Lemma 7.8 of \cite{GHKK}]
\label{lemma:GHKK_7-8}
For $m \in \textrm{Hom}(L_{\bfs},\mathbb{Z})$, the character $z^m$ on $T_{L,\bfs} \subset V$ is a \emph{global monomial} if and only if $z^m$ is regular on $TV(\Sigma_{\bfs,V})$. The character $z^{m}$ is regular on $TV(\Sigma_{\bfs,V})$ if and only if $\langle m, n \rangle \geq 0$ for the primitive generators $n$ of each ray in $\Sigma_{\bfs, V}$. When $V$ is an $\mathcal{A}$-type cluster variety, the set of global monomials exactly coincides with the set of cluster monomials. That is, every global monomial is a monomial in the variables of a single cluster with non-negative exponents on the non-frozen variables.
\end{lemma}

\begin{proof}
The proof given in \cite{GHKK} for the ordinary case holds in the generalized setting, with one minor change: the support of $Z_i$ is now defined by the zero locus of polynomials rather than binomials. I.e., its support is now $1 + a_{i,1}z^{v_i} + \cdots + a_{i,r_i-1}z^{(r_i-1)v_i} + z^{r_iv_i} = 0$ rather than $1 + z^{v_i} = 0$. Subsequent portions of the proof still hold after this change is made.

Note also that the proof in the ordinary case uses the Laurent phenomenon. Because the Laurent phenomenon holds in the reciprocal generalized setting, see Theorem~\ref{theorem:genLaurentPhenom}, this portion of the proof extends to our setting without modification.
\end{proof}

\begin{definition}
For a generalized cluster variety $V = \bigcup_{\bfs} T_{L,\bfs}$, let ${\mathcal{C}_{\bfs}^{+}(\mathbb{Z}) \subset V^{\vee}(\mathbb{Z}^T)}$ denote the set of $g$-vectors for global monomials that are characters on the torus $T_{L,\mathbf{s}} \subset V$ and $\Delta^{+}_{V}(\mathbb{Z}) \subset V^{\vee}(\mathbb{Z}^T)$ denote the union of all $\mathcal{C}_{\bfs}^{+}(\mathbb{Z})$.
\end{definition}

Recall that $\mathcal{A}_t = \pi^{-1}(t)$, where $\pi$ is the canonical fibration $\mathcal{A}_{\textrm{prin}} \rightarrow T_M$. Consider the maps $\rho: \mathcal{A}_{\textrm{prin}}^{\vee} \rightarrow \mathcal{A}^{\vee}$ and $\xi: \mathcal{X}^{\vee} \rightarrow \mathcal{A}^{\vee}_{\textrm{prin}}$ which have tropicalizations
\begin{align*}
    &\rho^T: (m,n) \mapsto m, \\
    &\xi^{T}: n \mapsto (-p^*(n),-n).
\end{align*}
The map $\rho^T$ identifies $\mathcal{A}^{\vee}(\mathbb{Z}^T)$ and the quotient of $\mathcal{A}^{\vee}_{\textrm{prin}}(\mathbb{Z}^T)$ by the natural action of $N$. Let $w: \mathcal{A}^{\vee}_{\textrm{prin}} \rightarrow T_{M^{\circ}}$ be the weight map given by $w(m,n) = m-p^*(n)$. Because $\xi$ identifies $\mathcal{X}^{\vee}$ with the fiber of $w: \Aprin^{\vee} \rightarrow T_{M^{\circ}}$ over the identity element $e$, the map $\xi^{T}$ identifies $\mathcal{X}^{\vee}(\mathbb{Z}^T)$ with $w^{-1}(0)$.

\begin{lemma}[Analogue of 7.10 of \cite{GHKK}]
\label{lemma:GHKK_7-10}
\,
\begin{enumerate}
    \item When $V$ is a generalized $\mathcal{A}$-variety, $\mathcal{C}_{\bfs}^{+}(\mathbb{Z})$ is the set of integral points of the cone $\mathcal{C}_{\bfs}^{+}$ in the Fock-Goncharov cluster complex which corresponds to the seed $\bfs$.
    \item For any type of generalized cluster variety, $\mathcal{C}_{\bfs}^{+}(\mathbb{Z})$ is the set of integral points of a rational convex cone $\mathcal{C}_{\bfs}^{+}$ and the relative interiors of $\mathcal{C}_{\bfs}^{+}$ as $\bfs$ varies are disjoint. The $g$-vector $\mathbf{g}(f) \in V^{\vee}(\mathbb{Z}^T)$ depends only on the function $f$. That is, if $f$ restricts to a character on two distinct seed tori then the $g$-vectors they determine are the same.
    \item For $m \in w^{-1}(0) \cap \Delta^{+}_{\Aprin}(\mathbb{Z})$, the global monomial $\vartheta_{m}$ on $\Aprin$ is invariant under the $T_{N^{\circ}}$ action and thus gives a global function on $\mathcal{X} = \Aprin/T_{N^{\circ}}$. This is a global monomial and all global monomials on $\mathcal{X}$ occur in this way. Moreover, $m = \mathbf{g}(\vartheta_{m})$.
\end{enumerate}
\end{lemma}

\begin{proof}
The proof given in \cite{GHKK} for the ordinary case holds in our generalized setting, as we have proven analogs of all the necessary previous results. We quickly review the proof given in \cite{GHKK} in order to point out each place where we are instead using an analogous result for the generalized setting.

First, consider (1) for a generalized $\mathcal{A}$-variety. By Lemma~\ref{lemma:FG_clusterChamber} and Lemma~\ref{lemma:GHKK_7-8}, the positive chamber $\mathcal{C}_{\bfs}^{+}$ is the Fock-Goncharov cluster chamber associated to $\bfs$ and a maximal cone of a simplicial fan. By Theorem~\ref{theorem:GHKK_2-13}, $\Delta_{\mathcal{A}}^{+}(\mathbb{Z})$ forms a simplicial fan.

The $\mathcal{A}$ case, which includes $\Aprin$, of (2) follows from Section~\ref{sec:gVectors}.
The $\mathcal{X}$ case follows from the $\Aprin$ case. Recall that the map $\widetilde{p} : \Aprin \rightarrow \mathcal{X}$ given by  $z^{(n,m)} \mapsto z^{m-p^*(n)}$ makes $\Aprin$ into a $T_{N^{\circ}}$-torsor over $\mathcal{X}$. Hence, pulling back a monomial on $\mathcal{X}$ yields a $T_{N^{\circ}}$-invariant global monomial on $\Aprin$. By Proposition~\ref{prop:GHKK_7-7}, we have the inclusion $\Delta_{\mathcal{X}}^{+}(\mathbb{Z}) \subseteq w^{-1}(0) \cap \Delta_{\Aprin}^{+}$. Conversely, suppose $m \in w^{-1}(0)$ and $m = \mathbf{g}(f)$ for some global monomial $f$ on $\Aprin$. Then there exists some generalized torus seed $\bfs = \{ (e_i,\mathbf{a}_i) \}$ such that $f$ is represented by a monomial $z^m$ on $T_{\widetilde{N},\bfs}$. Since $m \in w^{-1}(0)$, it must be of the form $m = (p^*(n),n)$ for some $n \in N$. By Lemma~\ref{lemma:GHKK_7-8}, $m$ is non-negative on the rays $\mathbb{R}_{\geq 0}(e_i,0)$ of $\Sigma_{\bfs,\Aprin}$ and therefore $n$ is non-negative on the rays $-\mathbb{R}_{\geq 0}v_i$ of $\Sigma_{\bfs,\mathcal{X}}$ and $z^n$ is a global monomial on $\mathcal{X}$. As such, $\Delta_{\mathcal{X}}^{+}(\mathbb{Z}) = w^{-1}(0) \cap \Delta_{\Aprin}^{+}(\mathbb{Z})$ and the cones for $\mathcal{X}$ are given by intersecting the cones for $\Aprin$ with $w^{-1}(0)$. This also gives (3).
\end{proof}

\subsubsection*{Broken lines and theta functions for the $\mathcal{A}$ and $\mathcal{X}$ cases.}

The maps $\rho$ and $w$ allow us to define broken lines for the $\mathcal{A}$ and $\mathcal{X}$ cases. Recall that each wall in $\mathfrak{D}^{\mathcal{A}_{\textrm{prin}}}_{\bfs}$ has an associated wall-crossing automorphism that is a power series in $z^{(p^*(n),n)}$ for some $n$. Hence, $w(m,n) = w(p^*(n),n) = 0$ for every exponent which appears in one of these wall-crossing automorphisms.

First, consider the $\mathcal{X}$ case. Suppose $\gamma$ is a broken line in $\mathfrak{D}^{\mathcal{A}_{\textrm{prin}}}_{\bfs}$ with both $I(\gamma)$ and the initial domain of linearity lying in $w^{-1}(0)$. Because every exponent that appears in a wall-crossing function lies in $w^{-1}(0)$, the monomials attached to each subsequent domain of linearity must also lie in $w^{-1}(0)$. In particular, this means that $F(\gamma)$ and $b(\gamma)$ both lie in $w^{-1}(0)$. We define the set of broken lines in $\mathcal{X}^{\vee}(\mathbb{R}^T)$ to be the set of such broken lines.
The \emph{$\cX$ theta functions} are the theta functions defined by those broken lines.
Next, consider the $\mathcal{A}$ case. Here, the set of broken lines in $\mathcal{A}^{\vee}(\mathbb{R}^T)$ is defined as $\{ \rho^T(\gamma) \}$ where $\gamma$ ranges over the set of broken lines in $\mathcal{A}^{\vee}_{\textrm{prin}}(\mathbb{R}^T)$. 
Similarly, the \emph{$\cA$ theta functions} are the theta functions defined by those broken lines.

To define functions on the $\cA$ and $\cX$ varieties, we would restrict to the set $\Theta$. Hence, we define
\begin{align*}
    \Theta(\mathcal{X}) &:= \Theta\left( \mathcal{A}_{\textrm{prin}} \right) \cap w^{-1}(0),
\\
        \Theta(\mathcal{A}_{t}) &:= \rho^T\left(\Theta\left( \mathcal{A}_{\textrm{prin}} \right) \right).
\end{align*}

Because $\mathcal{A}^{\vee}(\mathbb{Z}^T)$ is identified with the quotient of $\mathcal{A}^{\vee}_{\textrm{prin}}(\mathbb{Z}^T)$ under the natural $N$-action, it follows that $\Theta(\mathcal{A}_{\textrm{prin}})$ is invariant under $N$-translation and therefore $\Theta(\mathcal{A}_{\textrm{prin}}) = \left( \rho^T \right)^{-1}\left( \Theta(\mathcal{A}_{t}) \right)$. In fact, any section $\Sigma: \mathcal{A}^{\vee}(\mathbb{Z}^T) \rightarrow \mathcal{A}^{\vee}_{\textrm{prin}}(\mathbb{Z}^T)$ of $\rho^T$ will induce a bijection between $\Theta(\mathcal{A}_{\textrm{prin}})$ and $\Theta(\mathcal{A}_{t}) \times N$.

The space $\mathcal{A}_{t} = \cup_{\bfs} T_{N^{\circ},\bfs}$ consists of tori glued by birational maps that depend on $t$. When tropicalized, however, these maps are independent of $t$ and so it is natural to define $\mathcal{A}_t^{\vee}(\mathbb{Z}^T) := \mathcal{A}^{\vee}(\mathbb{Z}^T)$. With that notation, we can state the following results for $\mathcal{A}_t$ and $\mathcal{X}$:

\begin{theorem}[Analogue of Corollary 7.13 of \cite{GHKK}]
For $\mathcal{X}$,
\begin{enumerate}
    \item There are canonically defined structure constants
    \[ \alpha: \mathcal{X}^{\vee} \times \mathcal{X}^{\vee}(\mathbb{Z}^T) \times \mathcal{X}^{\vee}(\mathbb{Z}^T) \rightarrow \mathbb{Z}_{\geq 0} \]
    given by counts of broken lines.
    \item The subset $\Theta(\mathcal{X})$ contains the $g$-vector of each global monomial, i.e. $\Delta^{+}(\mathbb{Z}) \subset \Theta(\mathcal{X})$.
\end{enumerate}

\end{theorem}

\begin{theorem}[Analogue of Theorem 7.16 of \cite{GHKK}] \label{thm:Abasis}
For $V=\mathcal{A}_t$,
\begin{enumerate}
    \item Given a choice of section $\Sigma: \mathcal{A}^{\vee}(\mathbb{Z}^T) \rightarrow \mathcal{A}_{\textrm{prin}}^{\vee}(\mathbb{Z}^T)$, there exists a map
    \begin{align*} \alpha_{V} : \At^{\vee}(\mathbb{Z}^{T}) \times \At^{\vee}(\mathbb{Z}^{T}) \times \At^{\vee}(\mathbb{Z}^{T}) &\rightarrow \Bbbk \cup \{ \infty \},
    \end{align*}
    given by the mapping
    \begin{align*}
    (p,q,r) &\mapsto \sum_{n \in N} \alpha_{\mathcal{A}_{\textrm{prin}}}\left( \sigma(p), \Sigma(q), \Sigma(r) + n \right)z^n(t)
    \end{align*}
    if the summation is finite. Otherwise, $\alpha_{V}(p,q,r) = \infty$. If $p,q,r \in \Theta(\mathcal{A}_t)$, then the summation will be finite.
    \item The subset $\Theta$ contains the $g$-vector of each global monomial, i.e. $\Delta_{V}^{+}(\mathbb{Z}^T) \subset \Theta$.
\end{enumerate}
By taking $t = e$, we obtain these statements for $\mathcal{A}$.
\end{theorem}

For both, the arguments given in \cite{GHKK} for the ordinary setting still hold in our generalized setting.

\section{Companion Cluster Algebras}
\label{sec:companionDiagrams}

Given a generalized cluster algebra, 
Nakanishi and Rupel define the notion of \emph{companion algebras}, which are a pair of   ordinary cluster algebras associated to the generalized cluster algebra \cite{Nakanishi-Rupel}.

Recall that $(\mathbb{P},\oplus)$ is an arbitrary semifield. Let $\mathbb{QP}$ denote the field of fractions of this semifield and $\mathbb{QP}(\mathbf{x}) = \mathbb{QP}(x_1,\dots,x_n)$ denote the field of rational functions in the algebraically independent variables $x_1, \dots, x_n$. In earlier sections we referred to $\mathbb{QP}(\mathbf{x})$ as $\mathcal{F}$, but here it will be convenient to write $\mathbb{QP}(\mathbf{x})$ since we will also want to discuss fields of rational functions in other sets of algebraically independent variables.

Fix a generalized cluster seed $\Sigma = (\mathbf{x},\mathbf{y},B,[r_{ij}],\mathbf{a})$.
The corresponding fixed data $\Gamma$ has lattices $N$, $M$, $M^{\circ}$, $N^{\circ}$ and collections of scalars $\{d_i\}$ and $\{r_i\}$, as specified in Definition \ref{def:gen fixed}.  Consider the initial generalized torus seed data be $\bfs = \{ (e_i, (a_{i,s})) \}$, where $\{e_i\}$ forms a basis for $N$,  $\{d_i e_i\}$  for $N^{\circ}$, and $\{ f_i = \frac{1}{d_i} e_i^* \}$  for $M^{\circ}$. 
As usual, we let $v_i = p_1^* (e_i) \in M^{\circ}$ for $i \in I_{\mathrm{uf}}$. 

Let $\mathcal{A}$ denote the generalized cluster algebra defined by $\Sigma$. We can then state the definitions of the companion algebras as:

\begin{definition}[\cite{Nakanishi-Rupel}]
Let ${}^L\mathbf{x} = \mathbf{x}^{1/\mathbf{r}} = ({}^Lx_1, \dots, {}^Lx_n) := (x_1^{1/r_1},\dots,x_n^{1/r_n})$ in $\mathbb{QP}(\mathbf{x}^{1/\mathbf{r}})$
The \emph{left companion algebra} of $\mathcal{A}$ is the ordinary cluster algebra ${}^L\mathcal{A} \subset \mathbb{QP}(\mathbf{x}^{1/\mathbf{r}})$ with seed $({}^L\mathbf{x},{}^L\mathbf{y}, B[r_{ij}])$. 
Let ${}^L\mathbf{c}_j$, ${}^L\mathbf{g}_j$, and ${}^LF_j$ denote the $c$-vectors, $g$-vectors, and $F$-polynomials of ${}^L\mathcal{A}$.
\end{definition}

\begin{definition} \cite{Nakanishi-Rupel}
Let ${}^R\mathbf{x} = \mathbf{x}$ and ${}^R\mathbf{y} = \mathbf{y}^{\mathbf{r}} = ({}^Ry_1, \dots, {}^Ry_n) = (y_1^{r_1}, \dots, y_n^{r_n})$. The \emph{right companion algebra} of $\mathcal{A}$ is the ordinary cluster algebra ${}^R\mathcal{A} \subset \mathbb{QP}(\mathbf{x})$ with seed $(\mathbf{x},\mathbf{y}^{\mathbf{r}},[r_{ij}]B)$. Let ${}^R\mathbf{c}_j$, ${}^R\mathbf{g}_j$, and ${}^RF_j$ denote the $c$-vectors, $g$-vectors, and $F$-polynomials of ${}^R\mathcal{A}$.
\end{definition}

\begin{remark}
Note that because the mutation formulas used in this article are the transpose of those used by Fomin and Zelevinsky \cite{FZ-I}, our definitions of the companion algebras are not identical to the definitions given in \cite{Nakanishi-Rupel}. As compared to Nakanishi and Rupel, we switch the roles of the exchange matrices $[r_{ij}]B$ and $B[r_{ij}]$ in the definitions of the left and right companion algebras. 
\end{remark}

\subsection{Langlands duality and tropical duality}
\label{sec:duality}

In this subsection, we largely restrict our attention to ordinary cluster algebras and quickly review background material about two types of duality. We also briefly state the definition of the Langlands dual in the generalized setting. In the next subsection, we will discuss one way that these notions appear in the context of generalized cluster algebras.

\subsubsection*{Langlands duality}

In \cite{FG}, Fock and Goncharov give the following definitions for the \emph{Langlands dual} of a set of fixed data and a torus seed:

\begin{definition}[]
Given fixed data $\Gamma$ and torus seed $\bfs$, let $D := \textrm{lcm}(d_1,\dots,d_n)$. The \emph{Langlands dual} of $\Gamma$ is the fixed data $\Gamma^\vee$ defined by $N^\vee := N^\circ$, $I^\vee := I$, $I^\vee_{\textrm{uf}} := I_{\textrm{uf}}$, and $d_i^\vee := d_i^{-1}D$ with $\mathbb{Q}$-valued skew symmetric form $\{ \cdot, \cdot \}^{\vee} := D^{-1} \{ \cdot, \cdot \}$. The \emph{Langlands dual} of $\bfs$ is the torus seed $\bfs^\vee := (d_1e_1,\dots, d_ne_n)$.
\end{definition}
In the generalized setting, this definition is updated as follows.
\begin{definition}
Given generalized fixed data $\Gamma$ and generalized torus seed $\bfs$, the \emph{Langlands dual} of $\Gamma$ is the generalized fixed data $\Gamma^{\vee}$ defined as in the ordinary case, with the addition that $r_i^{\vee} = \textrm{lcm}(r_1,\dots,r_n)/r_i$ and the collection of formal variables $\{ a_{i,j} \}$ remains the same. The \emph{Langlands dual} of $\mathbf{s}$ is then the generalized torus seed $\mathbf{s}^{\vee} := \{ (e_i,(a_{i,j}')) \}_{i \in I_{\textrm{uf}},j \in [r_i^{\vee}-1]}$.
\end{definition}

From these definitions, we can observe the following relationship between the bilinear forms of $\Gamma$ and $\Gamma^{\vee}$.
    \begin{align*}
        \{ e_i, e_j\} d_j &= \frac{1}{d_i} \{ d_ie_i , d_j e_j \} \\
        &= -D^{-1}\{ d_j e_j , d_i e_i \} \left( \frac{D}{d_i} \right) \\
        &= -\{ e_j^{\vee}, e_i^{\vee} \}^{\vee}( d_i^{\vee})
    \end{align*}
It follows that $\epsilon_{ij} = -\epsilon_{ji}^{\vee}$ and that, on the matrix level, $\epsilon = -(\epsilon^{\vee})^T$.
Recall from Definition \ref{def:ordseed} that the exchange matrix $B$ (in the Fomin-Zelevinsky sense) can be represented as
\[
b_{ij} = \epsilon_{ij} = \{ e_i, e_j \} d_j
\]
for a given choice of torus seed $\bfs$. Hence, the Langlands dual of an ordinary cluster algebra defined by the exchange matrix $B = [\epsilon_{ij}]$ is simply the ordinary cluster algebra defined by $-B^T = [-\epsilon_{ij}^T]$. 

\subsubsection*{Tropical duality}

The study of cluster algebras often involves a different type of duality, called \emph{tropical duality}. In \cite{FZ-IV}, Fomin and Zelevinsky define $c$-vectors and $g$-vectors as the tropicalizations of the $A$ and $X$ cluster variables. As noted in earlier sections, Gross, Hacking, Keel, and Kontsevich showed in Lemma 2.10 of \cite{GHKK} that given a choice of seed data $\bfs$, $M^{\circ}_{\mathbb{R}, \bfs} \cong \cA^{\vee }(\mathbb{R}^T)$, where $ \cA^{\vee }(\mathbb{R}^T)$ denotes the tropicalization of the Fock-Goncharov dual variety. By explaining the connection between the lattices used in \cite{GHKK} and the $c$-vectors and $g$ vectors, we can connect these seemingly disparate statements.

Consider the fixed data $\Gamma$ and choice of initial torus seed data $\bfs_{\textrm{in}} = ( e_{\mathrm{in}, i})$. By definition, the basis vectors of $N_{\bfs_{\textrm{in}}}$ are $\{ e_{\textrm{in},i} \}_{i \in I}$ and the basis vectors of $M_{\bfs_{\textrm{in}}}^{\circ}$ are $\{ f_{\mathrm{in},i} \}_{i \in I}$, where $f_{\mathrm{in}, i} = \frac{1}{d_i} e^*_{\mathrm{in},i}$. Recall that $\cA$ cluster variables have the form $z^{m}$ with $m \in M_{\bfs_{\textrm{in}}}^{\circ}$ and that the $\cX$ cluster variables have the form $z^n$ with $n \in N_{\bfs_{\textrm{in}}}$. Because the $c$-vectors and $g$-vectors are the tropicalizations of these cluster variables, we can express them as $\textbf{c}_{\bfs_{\textrm{in}},i} = \sum_k c_{ik} e_{\init, k}$ and $\textbf{g}_{\bfs_{\textrm{in}},i} = \sum_k g_{ik} f_{\init, k}$.

Now, consider an arbitrary seed $\bfs = (e_{\bfs,i})$. Now, the lattice $N_{\bfs}$ has basis vectors $\{ e_{\bfs,i} \}_{i \in I}$ and the lattice $M_{\bfs}$ has basis vectors $\{ f_{\bfs,i} \}_{i \in I}$. The corresponding $\cA$ and $\cX$ cluster variables have $g$-vectors and $c$-vectors which we can write as $g_{\seed, i}= f_{\seed, i} =  \sum_k g_{ik} f_{\init, k}$ and $c_{\seed, i} = e_{\seed, i} = \sum_k c_{ik} e_{\init, k}$. 
We denote by $C_{\seed}^{\epsilon}$ (respectively, $G_{\seed}^{\epsilon}$) the integer matrix with columns $\textbf{c}_{1;\seed}, \dots, \textbf{c}_{n;\seed}$ (respectively, with columns $\textbf{g}_{1;\seed}, \dots, \textbf{g}_{n;\seed}$), where $\epsilon = [\epsilon_{ij}]$ is the matrix defined by $\seed_{\init}$. 

Nakanishi and Zelevinsky proved the following identity, referred to as a \emph{tropical duality}, between the $c$-vectors and $g$-vectors.

\begin{theorem}\cite[Theorem 1.2]{NZ}
For any torus seed $\bfs$ and the associated matrix $\epsilon = [\epsilon_{ij}] = [\{ e_i,e_j \}d_j]$, take $\epsilon_{\init} = \epsilon$ and let $\epsilon_{\bfs}$ denote the matrix associated to the torus seed $\bfs$. Then, 
\begin{align} \label{eq:cgduality}
    (G_{\seed}^{\epsilon} )^T= (C_{\seed}^{-\epsilon^T})^{-1}, 
\end{align}
\end{theorem}

We can also understand this tropical duality in the language of cluster scattering diagrams. In the previous subsection, we saw that replacing $\epsilon$ with $-\epsilon^T$ in the fixed data $\Gamma$ is equivalent to considering its Langlands dual, $\Gamma^{\vee}$. Consider some choice of torus seed $\bfs = \{ e_{\bfs,i} \}_{i \in I}$ associated to $\Gamma$. As explained above, the associated collections of $g$-vectors and $c$-vectors are, respectively, $\{ f_{\bfs,i} \}$ and $\{ e_{\bfs,i} \}$. The associated collections of $g$-vectors and $c$-vectors of the Langlands dual torus seed data $\bfs^{\vee}$ are therefore $\{ f_{\bfs,i}^{\vee} = d_iD^{-1}f_{\bfs,i} \}$ and $\{ e_{\bfs,i}^{\vee} =  d_ie_{s,i} \}$, respectively. One can immediately see that the bases $\{ f_{\bfs,i} \}$ and $\{ e_{\bfs, i}^{\vee} = d_ie_{\bfs,i} \}$ are dual, because
\begin{align*}
    \langle e_{\bfs,i}^{\vee},f_{\bfs,i} \rangle = \langle d_ie_{\bfs,i}, d_i^{-1}e_{\bfs,i}^* \rangle = \langle e_{\bfs,i}, e_{\bfs,i}^* \rangle = 1,
\end{align*}
which implies the tropical duality in the language of seed basis vectors. 

\subsection{Fixed data for the companion algebras}

Fix a generalized cluster seed $\Sigma = (\mathbf{x},\mathbf{y},B,[r_{ij}],\mathbf{a})$, which defines the generalized cluster algebra $\mathcal{A}$. We will assume that $\mathcal{A}$ is a reciprocal generalized cluster algebra, so we can use our construction of cluster scattering diagrams for reciprocal generalized cluster algebras.
The corresponding generalized fixed data $\Gamma$ has lattices $N$, $M$, $M^{\circ}$, $N^{\circ}$; index sets $I$ and $I_{\textrm{uf}}$; collections of scalars $\{d_i\}$ and $\{r_i\}$; and the collection of formal variables $\{ a_{i,j} \}_{i \in I, j \in [r_i -1]}$, as specified in Definition \ref{def:gen fixed}.  Consider the initial generalized torus seed data $\bfs = \{ (e_i, (a_{i,j})) \}_{i \in I}$, where $\{e_i\}_{i \in I}$ forms a basis for $N$,  $\{d_i e_i\}_{i \in I}$  for $N^{\circ}$, $\{ e_i^* \}_{i \in I}$ for $M$, and $\{ f_i = \frac{1}{d_i} e_i^* \}_{i \in I}$  for $M^{\circ}$. 
As usual, we let $v_i = p_1^* (e_i) \in M^{\circ}$ for $i \in I_{\mathrm{uf}}$. 

As explained in earlier section, the generalized cluster algebra $\mathcal{A}$ has an associated pair of companion algebras, $\lL\mathcal{A}$ and $\rR\mathcal{A}$. We can explicitly described the fixed data associated to the left and right companion algebras. In doing so, we will use the superscript ${}^{\mathrm{C}}$ to indicate when we're considering an object in a generic companion algebra and the superscripts $\lL $ and $\rR$, respectively, to denote the corresponding notions in the left and right companion algebras.
The data associated to the left and right companion algebras in the following table:

\begin{center}
\begingroup
\setlength{\tabcolsep}{10pt} 
\renewcommand{\arraystretch}{1.5} 
\begin{tabular}{| c ||c |c |}
\hline
    & Left & Right \\ \hline \hline
    $\{ \cdot , \cdot \}$ &
        $\{ \cdot , \cdot \}$ &  
        $\{ \cdot , \cdot \}$   \\
        \hline
    $\upC d_i$ 
        &$\ld_i = d_ir_i$ 
        & $\rd_i =  \frac{d_i}{r_i}$  \\
        \hline
    $\upC e_i$ 
        & $\lle_i = e_i$  
        & $\re_i = r_ie_i$ \\
        \hline
    $\upC f_i$
        & $\lf_i = \frac{1}{\ld_i} \lle_i^* =  \frac{1}{d_ir_i} e_i^* =\frac{1}{r_i} f_i$
        & $\rf_i =  \frac{1}{\rd_i} \re_i^* =  \frac{r_i}{d_i} \frac{1}{r_i}e_i^* =  f_i$ \\
        \hline
    $\upC N$ &
        $\lN=N$ &
        $\rN= \mathrm{span}\left\{  r_ie_i \right\}$ \\ 
        \hline
    $\upC N^{\circ}$
        & $\lN^{\circ} =$ span $\{r_i( d_ie_i) \} $ & $\rN^{\circ} =$ span $\left\{ d_ie_i \right\}$ \\
        \hline
    $\upC M$
        & $\lL M = $ span $\{ e_i^* \} = M$
        & $\rR M = $ span $\{ \frac{1}{r_i}e_i^* \}$\\
        \hline
    $\upC M^{\circ}$
        & $\lM^{\circ} =$ span $\{ \lf_i \}$
        & $\rM^{\circ} =$ span $ \{ \rf_i \}$  \\
        \hline
    $\upC x_i$
        & $\lL x_i = z^{\lf_i} = x_i^{1/r_i}$
        & $\rR x_i  =z^{\rf_i} = x_i$ \\
        \hline
    $\upC y_i$
        & $\lX_i = z^{\lle_i} = z^{e_i} = y_i$
        & $\rX = z^{\re_i} = z^{r_i e_i} = y_i^{r_i}$ \\
        \hline
    $\upC\hat{y} $ 
        & $\lL \hat{y}_i= z^{(v_i, e_i)} = \hat{y}$
        & $\rR\hat{y}_i = z^{(r_iv_i, r_i e_i)} = \hat{y}^{r_i}$\\
        \hline
    $\upC v_i$
        & $\lv_i = \{ \lle_i, \cdot \} = \{e_i , \cdot \} = v_i $
        & $\rv_i = \{ \re_i, \cdot \}  =\{r_ie_i, \cdot \} = r_i v_i $ \\
        \hline
    $\upC\langle d_i e_i, f_i \rangle$
        & $\langle \lL d_i \lL e_i, \lL f_i \rangle = \langle d_ir_i e_i, \frac{1}{r_i} f_i \rangle = 1$ 
        & $\langle \rR d_i \rR e_i, \rR f_i \rangle = \langle d_ir_i^{-1} (r_i e_i) , f_i \rangle = 1$ \\
        \hline
\end{tabular}
\endgroup
\end{center}
Observe that if $d_i =1$ for all $i$, then the left and right companion algebras are Langlands dual (up to the scalar factor $\mathrm{lcm} (\{ r_i \}_{i \in I})$) to each other.
The term $\frac{d_i}{r_i}$ being rational may raise some questions. 
Note that it is possible for a right companion algebra to have fixed data such that some $\rR d_i$ are non-integral.  In fact, $\rR d_i$ will be non-integral whenever $r_i \neq 1$. To ensure that all $\rR d_i$ are integral, one could scale $\rR d = (\rR d_i)$ by a factor of $\textrm{lcm}(r_i)$. Any computations done using a cluster scattering diagram generated by the scaled fixed data would then require that we account for this scaling. Further, this rational lattice is consistent with the fact that \cite{Nakanishi-Rupel} defines this companion algebra as a subalgebra of $\mathbb{QP}(x^{1/\mathbf{r}})$.

In fact, the unscaled fixed data $\rR \Gamma$ can still be used to define a sensible cluster scattering diagram that allows for easy computation of cluster variables, etc. To understand the impact of non-integral values of $\rR d_i$ on the construction of a cluster scattering diagram, we can consider the impact on the associated lattices and initial cluster scattering diagram. It's clear that $\rR N$ will always be an integral lattice, since $\rR N = \textrm{span} \{ r_ie_i \}$. The scaling of $\rR e_i$ guarantees that $\rR N^{\circ}$ is always an integral lattice, since
\[ \rR N^{\circ} = \textrm{span}\left\{ \rR d_i \rR e_i \right\} = \textrm{span}\left\{ \frac{d_i}{r_i} r_ie_i \right\} = \textrm{span}\left\{ d_ie_i \right\}.\]
Note that this scaling also means that 
\[ \rR M^{\circ} = \textrm{span}\{ \rR f_i\} = \textrm{span}\left\{ \frac{1}{\rR d_i}\left(\rR e_i\right)^* \right\} = \textrm{span}\left\{ \frac{r_i}{d_i}\frac{1}{r_i}e_i^* \right\}= \textrm{span}\left\{ \frac{1}{d_i}e_i^* \right\} =  \textrm{span} \{ f_i \}.\]
By definition, $\rR \mathcal{A}$ has initial scattering diagram
\begin{align*}
    \rR \mathfrak{D}_{\init} = \left\{ ((\rR e_i)^{\perp}, 1 + z^{\rR v_i}) \right\}_{i \in I_{\textrm{uf}}} = \left\{ ((r_ie_i)^{\perp}, 1 + z^{r_iv_i}) \right\}_{i \in I_{\textrm{uf}}},
\end{align*}
where all the wall-crossing automorphisms have integral exponents. As such, the algorithm for producing a consistent scattering diagram that was introduced by Kontsevich and Soibelman in two-dimensions and then extended to higher dimensions by Gross and Siebert can be applied to $\rR \mathfrak{D}_{\init}$. Likewise, other major results of Gross, Hacking, Keel, and Kontsevich \cite{GHKK} which rely on the wall-crossing automorphisms having integer exponents still hold for the resulting consistent scattering diagram.

\begin{example}
\label{example:companionAlgebras}
Consider the generalized cluster algebra from Example~\ref{ex:gen_T_k}, \[\mathcal{A}\left(\mathbf{x},\mathbf{y},\begin{bmatrix} 0 & 1 \\ -1 & 0 \end{bmatrix},\begin{bmatrix} 3 & 0 \\ 0 & 1 \end{bmatrix},((1,a,a,1),(1,1)) \right).\]
This generalized cluster algebra has companion algebras:
\begin{align*}
    {}^L\mathcal{A} &= \left( (x_1^{1/3},x_2), (y_1,y_2), \begin{bmatrix} 0 & 1 \\ -3 & 0 \end{bmatrix} \right) \\
    {}^R\mathcal{A} &= \left( (x_1,x_2), (y_1^3,y_2), \begin{bmatrix} 0 & 3 \\ -1 & 0 \end{bmatrix} \right)
\end{align*}
Recall from Example~\ref{ex:genFixedData} that $\mathcal{A}$ has fixed data $d = (1,1)$, $r = (3,1)$, $I = I_{\textrm{uf}} = \{ 1, 2 \}$, $N = N^{\circ} =  \langle e_1, e_2 \rangle$, $M = M^{\circ} = \langle e_1^*, e_2^* \rangle$, and  skew-symmetric form $\{ \cdot, \cdot \} : N^{\circ} \times N^{\circ} \rightarrow \mathbb{Z}$ specified by the exchange matrix.
The fixed and seed data of its associated companion algebras, in terms of the fixed data of the generalized cluster algebra, are summarized in the following table.

\begingroup
\def\arraystretch{1.25}
\begin{center}
    \begin{tabular}{|c||c|c|}
        \hline
         & ${}^L\mathcal{A}$ & ${}^R\mathcal{A}$  \\ \hline \hline
         $\upC d$ & $(3,1)$ & $\left( \frac{1}{3}, 1 \right)$ \\ \hline
         $ \upC e_i $ & $ e_1, e_2 $ & $3e_1, e_2$ \\ \hline
         $ \upC f_i $ & $ \frac{1}{3}f_1, f_2 $ & $f_1, f_2$\\ \hline
         $\upC N$ & span $\{ e_1, e_2 \}$ & span $\{ 3e_1, e_2 \} $\\ \hline
         $\upC N^{\circ}$ & span $\{ 3e_1, e_2 \}$ & span $\{  e_1, e_2 \}$\\ \hline
         $\upC M$ & span $\{ e_1^*, e_2^* \}$ & span $\left\{ \frac{1}{3}e_1^*, e_2^* \right\}$ \\ \hline
         $\upC M^{\circ}$ & span $ \left\{ \frac{1}{3}f_1, f_2 \right\} $ & span $ \{ f_1, f_2 \}$ \\ \hline
         $\upC v_i$ & $v_1, v_2$ & $3v_1, v_2$ \\ \hline 
    \end{tabular}
\end{center}
The cluster scattering diagrams for $\lL\cA$ and $\rR\cA$ are shown below (on the top and bottom, respectively).
Although the data in the above table is stated in terms of the fixed data of the generalized cluster algebra, the following ordinary cluster scattering diagrams are drawn from the fixed data of the companion algebras and the wall functions are stated in terms of the $\lL x_i$ and $\rR x_i$. In this example, we note that the generalized and companion fixed data mostly coincide, but notably 
\[{}^L f_1 = \frac{1}{3}f_1, \  {}^Re_1 = 3e_1, \textrm{ and } \ {}^R v_1 = 3v_1.\]
Because of this, we need to be careful to avoid confusion about which variables are being used when writing the wall functions. The wall functions in the following diagrams use the variables of the companion algebra, but could also be written using the variables of the generalized algebra. For instance, observe that \[{}^Lf_{\mathfrak{d}_1} = 1 + {}^Lz^{({}^Lv_2)} =  1 + \lL z^{-3({}^Lf_1)} 
= 1 + \lL z^{(-3,0)} =
1+ (\lL x_1)^{-3}  = 1+ (x_1^{1/3})^{-3} = 1+x_1^{-1}= 1+z^{v_2} .\]

\begin{center}
    \begin{minipage}{0.3\textwidth}
        \begin{tikzpicture}[scale=0.65]
		\draw (-3,0) to (3,0);
		\draw (0,-3.4) to (0,3);
		\draw (0,0) to (2.5,-2.5);
		\draw (0,0) to (3,-2);
		\draw (0,0) to (3,-1.5);
		\draw (0,0) to (3,-1);
				
		\node[scale=0.9] at (3.3,0) {$\mathfrak{d}_1$};
		\node[scale=0.9] at (0,3.25) {$\mathfrak{d}_2$};
		\node[scale=0.9] at (2.75,-2.85) {$\mathfrak{d}_6$};
		\node[scale=0.9] at (3.25, -2.25) {$\mathfrak{d}_5$};
		\node[scale=0.9] at (3.25, -1.65) {$\mathfrak{d}_4$};
		\node[scale=0.9] at (3.25,-1) {$\mathfrak{d}_3$};
	\end{tikzpicture}
    \end{minipage}
    \begin{minipage}{0.44\textwidth}
        \begin{align*}
            {}^Lf_{\fd_1} &= 1 + {}^Lz^{(-3,0)}  \\
            {}^Lf_{\fd_2} &= 1 + {}^Lz^{(0,1)} \\
            {}^Lf_{\fd_3} &= 1 + {}^Lz^{(-3,3)} \\
            {}^Lf_{\fd_4} &= 1 + {}^Lz^{(-3,2)} \\
            {}^Lf_{\fd_5} &= 1 + {}^Lz^{(-6,3)} \\
            {}^Lf_{\fd_6} &= 1 + {}^Lz^{(-3,1)}
        \end{align*}
    \end{minipage}
    
    \vspace{5mm}
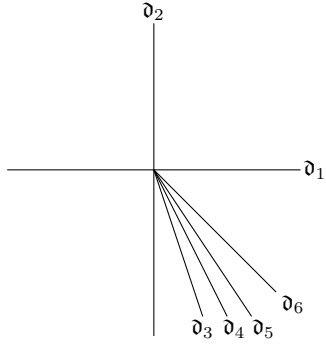
    \begin{minipage}{0.3\textwidth}
        \begin{tikzpicture}[scale=0.65]
		\draw (-3,0) to (3,0);
		\draw (0,-3.4) to (0,3);
		\draw (0,0) to (2.5,-2.5);
		\draw (0,0) to (2,-3);
		\draw (0,0) to (1.5,-3);
		\draw (0,0) to (1,-3);
				
		\node[scale=0.9] at (3.3,0) {$\mathfrak{d}_1$};
		\node[scale=0.9] at (0,3.25) {$\mathfrak{d}_2$};
		\node[scale=0.9] at (2.85,-2.75) {$\mathfrak{d}_6$};
		\node[scale=0.9] at (2.25,-3.25) {$\mathfrak{d}_5$};
		\node[scale=0.9] at (1.65,-3.25) {$\mathfrak{d}_4$};
		\node[scale=0.9] at (1,-3.25) {$\mathfrak{d}_3$};
	\end{tikzpicture}
    \end{minipage}
    \begin{minipage}{0.44\textwidth}
        \begin{align*}
            {}^Rf_{\fd_1} &= 1 + {}^Rz^{(-1,0)} \\
            {}^Rf_{\fd_2} &= 1 + {}^Rz^{(0,3)} \\
            {}^Rf_{\fd_3} &= 1 + {}^Rz^{(-1,3)} \\
            {}^Rf_{\fd_4} &= 1 + {}^Rz^{(-3,6)} \\
            {}^Rf_{\fd_5} &= 1 + {}^Rz^{(-2,3)} \\
            {}^Rf_{\fd_6} &= 1 + {}^Rz^{(-3,2)}
        \end{align*}
    \end{minipage}
\end{center}
\end{example}
\endgroup

\begin{remark}
Note that under the hypothesis 
$d_i = 1$, the left and right companion algebras are, up to isomorphism and the scalar factor $\mathrm{lcm}(\{r_i\}_{i \in I})$, Langlands dual according to the definition given in Section~\ref{sec:duality}.
In particular, let $\mathcal{A}$ be a reciprocal generalized cluster algebra where all $d_i = 1$ with companion algebras $\lL \mathcal{A}$ and $\rR \mathcal{A}$.  Let $\lL \Gamma$ and $\rR \Gamma$ denote the fixed data of $\lL \mathcal{A}$ and $\rR \mathcal{A}$, respectively, then
\begin{align*}
    \left(\lL \Gamma \right)^{\vee} \simeq \rR \Gamma \qquad \textrm{ and } \qquad \left(\rR \Gamma \right)^{\vee} \simeq \lL \Gamma, \quad \text{up to the scalar factor $\mathrm{lcm}(\{r_i\}_{i \in I})$.}
\end{align*}

\end{remark}

Based on the explicit fixed data of the companion algebras, we can make several other useful observations. First, the $c$-vectors of the left companion algebras and the generalized algebras coincide because the torus seed basis vectors are the same. 
On the other hand, for the $g$-vectors we have
\begin{align*}
    \lL  \textbf{g}_{\seed, j} &= \frac{1}{r_j} g_{\seed, j} \\
    &= \frac{1}{r_j} \sum_i g_{ji} \textbf{f}_{\init, i} \\
    &= \frac{1}{r_j} \sum_i \frac{r_i}{r_j} g_{ji} \lL \textbf{f}_{\init, i}.
\end{align*}

Hence, we have $\lL  \textbf{g}_{\seed, j} = [\frac{r_i}{r_j} g_{ji}]_{i \in I}$ as in Corollary 4.2 of  \cite{Nakanishi-Rupel}. 
We can similarly deduce that the $g$-vectors of the right companion algebras and the generalized algebra are the same, since $f_i = {}^Rf_i$. 
The $c$-vectors are related as
\begin{align*}
    \rR \textbf{c}_{\seed, j} & = r_j \textbf{c}_{\seed, j} \\
    &= r_j \sum_i c_{ji} \textbf{c}_{\init, i} \\
    &= \sum_i \frac{r_j }{r_i} c_{ji} \rR\textbf{c}_{\init, i}, 
\end{align*}
which also agrees with Corollary 4.1 of \cite{Nakanishi-Rupel}. 

We can also explore the relationship between mutation of the generalized cluster algebra and mutation of its associated companion algebras.  Consider a reciprocal generalized cluster algebra $\mathcal{A}$ with fixed data $\Gamma$. Again, let us consider a fixed generalized torus seed $\bfs$ and the associated scattering diagram with principal coefficients, $\fD_{\bfs, \prin}$. The fixed data for the left and right companion algebras of $\mathcal{A}$ are as specified earlier in this section.
By definition, the initial scattering diagram for ${}^L\mathcal{A}$
is of the form
\[
\lL\fD_{\seed, \prin}^{\init} = \left\{ \left( \lL\fd_k = (r_k d_k e_k , 0)^{\perp}, 
\lL f_{\fd_k}=1+z^{(v_k,e_k)}\right) , \text{ for } k \in I_{\mathrm{uf}} \right\}.
\]
Note that the dual lattice of $\lM^{\circ}$ is ${\lN^{\circ}} = \textrm{span}\{ r_id_ie_i \}_{i \in I}$. Hence, the primitive vectors normal to the wall $\mathfrak{d}_k$ in $\lL\fD_{\seed, \prin}^{\init}$ have the form $(\pm r_kd_ke_k,0)$. By convention, we will choose to use the primitive normal vector $(r_kd_ke_k,0)$ when calculating path-ordered products. Now, consider the cluster variables $\lL \hat{y}_i= z^{(v_i, e_i)}$ and let
\[
v_i = \sum_k \lL \nu_{ik} \lL f_k =  \sum_k \lL \nu_{ik} \cdot ( r_k f_k).
\]
Mutating $\lL \hat{y}_i$ in direction $k$ is equivalent to applying the wall-crossing automorphism associated to $\lL \fd_k$ to the monomial $z^{(v_i,e_i)}$. Hence, we can calculate $\mu_k(\lL \hat{y}_i)$ by computing
\begin{align*}
    \lL \hat{y}_i = z^{(\lL v_i,e_i)} &\xmapsto{\lL\fd_k} z^{(v_i,e_i)}\left( \lL f_{\fd_k} \right)^{\langle (r_kd_ke_k,0), (v_i,e_i) \rangle} \\
    &= z^{(\lL v_i,e_i)}\left( \lL f_{\fd_k} \right)^{\langle r_kd_ke_k, \sum_{j \in I} \lL v_{ij} \lL f_j \rangle} \\
    &= z^{(\lL v_i,e_i)}\left( \lL f_{\fd_k} \right)^{\sum_{j \in I} \langle r_kd_ke_k,\lL v_{ij} \lL f_j \rangle}
\end{align*}
Observe that
\begin{align*}
    \langle r_kd_ke_k, \lL v_{ij} \lL f_j\rangle = r_kd_k \lL v_{ij} \langle e_k, \lL f_j \rangle = r_k d_k \lL v_{ij} \langle e_k, \frac{1}{r_jd_j} e_j^* \rangle = \begin{cases} \lL v_{ij} & j = k \\ 0 & j \neq k \end{cases}
\end{align*}
Hence, we have
\begin{align*}
    \lL \hat{y}_i = z^{(\lL v_i,e_i)} &\xmapsto{\lL\fd_k} z^{(\lL v_i,e_i)} \left( \lL f_{\fd_k} \right)^{\lL v_{ij}}
\end{align*}
Recall that the variables of the generalized cluster algebra and the left companion algebra are related by $\lL x_i = x_i^{1/r_i}$ and $\lL y_i = y_i$. Recall also that $\lL v_i = v_i$ and therefore $\lL \hat{y_i} = z^{(\lL v_i,e_i)} = z^{(v_i,e_i)} = \hat{y_i}$. Hence, we can rewrite the above map as
\begin{align*}
    \hat{y_i} &\xmapsto{\lL \fd_k} z^{(v_i,e_i)}\left(\lL f_{\fd_k} \right)^{\lL v_{ik}r_k},
\end{align*}
which agrees with the $F$-polynomial transformation given in Proposition 4.3 of \cite{Nakanishi-Rupel}. The analogous computation for mutation in the right companion algebra also agrees with Proposition 4.6 of \cite{Nakanishi-Rupel}.

\vspace{5mm}
\noindent \textbf{Declarations:} M. Cheung was partially supported by NSF grant DMS-1854512. G. Musiker was partially supported by NSF grant DMS-1745638. E. Kelley was partially supported by NSF grants DMS-1745638 and DMS-1937241. The authors are grateful for the hospitality of RIMS in 2019 in Kyoto, Japan during the ``Cluster Algebras 2019'' workshop, where this collaboration began, and to the anonymous reviewers whose careful reading and insightful comments improved the exposition of this paper. On behalf of all authors, the corresponding author states that the authors have no conflicts of interest.

\printbibliography

\end{document}